\documentclass[a4paper,11pt]{amsart}

\usepackage{amsmath, amsthm, amssymb}
\usepackage{enumitem}
\usepackage{xcolor}
\usepackage{hyperref}
\hypersetup{
    colorlinks,
    linkcolor={red!50!black},
    citecolor={blue!50!black},
    urlcolor={blue!80!black}
}
\usepackage{geometry} 
\usepackage{tikz-cd}
\usepackage{stackengine}

\usepackage{overpic}
\usepackage[ansinew]{inputenc}
\usepackage{graphicx}
\usepackage[rightcaption]{sidecap}
\usepackage{caption}
\usepackage{subcaption}

\newtheorem{theorem}{Theorem}[section]
\newtheorem*{mainth}{Main Theorem}
\newtheorem{corollary}[theorem]{Corollary}
\newtheorem{lemma}[theorem]{Lemma}
\newtheorem{proposition}[theorem]{Proposition}

\newtheorem{claim}[theorem]{Claim}

\theoremstyle{definition}
\newtheorem{define}[theorem]{Definition}
\newtheorem{remark}[theorem]{Remark}

\newtheorem*{convention}{Convention}

\makeatletter
\newcommand{\setword}[2]{%
  \phantomsection
  #1\def\@currentlabel{\unexpanded{#1}}\label{#2}%
}
\makeatother

\newcommand{\fol}{\mbox{${\mathcal F}$}}
\newcommand{\fn}{\mbox{$\widetilde {\mathcal F}$}}
\newcommand{\cG}{\mbox{${\mathcal G}$}}
\newcommand{\cB}{\mbox{${\mathcal B}$}}
\newcommand{\cF}{\mbox{${\mathcal F}$}}
\newcommand{\lam}{\mbox{${\mathcal W}$}}
\newcommand{\wl}{\mbox{${\widetilde{\mathcal W}}$}}
\newcommand{\A}{\mbox{${\mathcal A}$}}
\newcommand{\R}{\mbox{${\mathcal R}$}}
\newcommand{\U}{\mbox{${\mathcal U}$}}
\newcommand{\Y}{\mbox{${\mathcal Y}$}}
\newcommand{\Z}{\mbox{${\mathcal Z}$}}
\newcommand{\K}{\mbox{${\mathcal K}$}}
\newcommand{\G}{\mbox{${\mathcal G}$}}
\newcommand{\F}{\mbox{${\mathcal F}$}}

\newcommand{\cW}{\mbox{${\mathcal W}$}}
\newcommand{\WW}{\mbox{${\mathcal W}$}}
\newcommand{\cD}{\mbox{${\mathcal D}$}}
\newcommand{\B}{\mbox{${\mathcal B}$}}
\newcommand{\gs}{\mbox{${\mathcal G}^s$}}
\newcommand{\gu}{\mbox{${\mathcal G}^u$}}
\newcommand{\wgs}{\mbox{$\widetilde {\mathcal G}^s$}}
\newcommand{\wgu}{\mbox{$\widetilde {\mathcal G}^u$}}
\newcommand{\PP}{\mbox{${\mathcal P}$}}
\newcommand{\C}{\mbox{${\mathcal C}$}}
\newcommand{\cC}{\mbox{${\mathcal C}$}}

\newcommand{\HH}{\mbox{${\mathcal H}$}}
\newcommand{\II}{\mbox{${\mathcal I}$}}
\newcommand{\D}{\mbox{${\mathcal D}$}}
\newcommand{\E}{\mbox{${\mathcal E}$}}
\newcommand{\cE}{\mbox{${\mathcal E}$}}
\newcommand{\V}{\mbox{${\mathcal V}$}}
\newcommand{\X}{\mbox{${\mathcal X}$}}

\newcommand{\mt}{\mbox{${\widetilde M}$}}
\newcommand{\rrrr}{\mbox{${\mathbb R}$}}
\newcommand{\oo}{\mbox{${\mathcal O}$}}
\newcommand{\oos}{\mbox{${\mathcal O}^s$}}
\newcommand{\oou}{\mbox{${\mathcal O}^u$}}
\newcommand{\ls}{\mbox{${\Lambda^s}$}}
\newcommand{\lsp}{\mbox{${\Lambda^s_{\Psi}}$}}
\newcommand{\lss}{\mbox{${\Lambda^s_2}$}}
\newcommand{\lu}{\mbox{${\Lambda^u}$}}
\newcommand{\wls}{\mbox{$\widetilde{\Lambda}^s$}}
\newcommand{\wlss}{\mbox{$\widetilde{\Lambda}^s_2$}}
\newcommand{\wlu}{\mbox{$\widetilde{\Lambda}^u$}}
\newcommand{\wwp}{\mbox{$\widetilde{\Phi}$}}

\newcommand{\eps}{\mbox{${\epsilon}$}}
\newcommand{\si}{\mbox{${\mathbb S^2_{\infty}}$}}

\title[Quasigeodesic Anosov flows in hyperbolic $3$-manifolds]{Non $\rrrr$-covered Anosov flows in hyperbolic $3$-manifolds are quasigeodesic}
\author[S.R. Fenley]{Sergio R.\ Fenley} 
\thanks{Research partially supported by Simons foundation
grant 637554, by National Science Foundation
grant DMS-2054909 and by the Institute for Advanced Study.}
\address{Florida State University, Tallahassee, FL 32306, USA}
\email{sfenley@fsu.edu}

\begin{document}
 
 \begin{abstract}
The main result is that if an Anosov flow in a 
closed hyperbolic
three manifold is not $\rrrr$-covered, 
then the flow is a quasigeodesic flow.
We also prove that if a hyperbolic three manifold supports
an Anosov flow, then up to a double cover it supports
a quasigeodesic flow.
We prove the continuous extension property for 
the stable and unstable foliations of any Anosov flow in a 
closed hyperbolic three manifold, and
the existence of group invariant
Peano curves associated with any such flow.


\end{abstract}

\maketitle


\section{Introduction}

This article studies geometric properties of Anosov flows 
in $3$-manifolds, with a particular interest when the
manifold is hyperbolic. The goal is to study the quasigeodesic
property for such flows. A {\em quasigeodesic} is 
a rectifiable curve so that length along the curve is a 
uniformly efficient measure of distance, 
up to a bounded multiplicative distortion and additive
constant, when lifted
to the universal cover. Quasigeodesics
are extremely
important when the manifold is hyperbolic \cite{Th1,Th2,Gr}.
A flow with no point orbits is a quasigeodesic flow
if all its orbits are quasigeodesics.

A suspension Anosov flow, or any suspension flow 
in any closed manifold of any dimension for that
matter, is quasigeodesic \cite{Ze}.
Any closed $3$-manifold fibering over the circle
with pseudo-Anosov monodromy is hyperbolic \cite{Th2},
hence the class of quasigeodesic flows in hyperbolic
$3$-manifolds is very large. 
What should a quasigeodesic flow in a hyperbolic $3$-manifold
look like?
Calegari \cite{Cal3} started the study of such flows,
and this was greatly extended by the work of Frankel \cite{Fra1,Fra2,Fra3}.
Frankel analyzed the universal circle associated
with such flows and also the stable and unstable 
decomposition sets in the
orbit space of such flows. He used the stable and unstable
decompositions to prove that such flows always
have closed orbits \cite{Fra3}.

Hence it is very natural to consider pseudo-Anosov flows 
in this setting and ask: given a pseudo-Anosov flow
in a closed hyperbolic $3$-manifold when is it quasigeodesic?
In this article we will answer this question in the
case of Anosov flows, that is, when there are no $p$-prong
singularities with $p \geq 3$.
More specifically we will 
consider topological Anosov flows:
it is a weakening of the Anosov property allowing for less
regularity. 
See precise definition in Section \ref{prelim}.
Still the orbits of topological Anosov flows
are $C^1$ curves, hence
rectifiable.

In general it is very hard to decide whether an Anosov flow
is quasigeodesic, or not, and this question has remained open until now. 
In fact the first result previously  obtained is a negative one.
A codimension one foliation is called $\rrrr$-{\em covered}
if the leaf space of the foliation lifted to the universal
cover of the manifold is homeomorphic to the real line
$\mathbb{R}$
\cite{Fe1}.
An Anosov flow in a $3$-manifold is called $\rrrr$-covered
if the stable and the unstable $2$-dimensional
foliations of
the flow are $\rrrr$-covered \cite{Fe1}.
It is enough to assume that one of these foliations
is $\rrrr$-covered \cite{Ba1}.
For Anosov flows in $3$-manifolds, there are many examples:
suspensions and geodesic flows for example. 
Using Dehn surgery on periodic orbits of 
suspensions and geodesic flows \cite{Go,Fr}, it
was proved in \cite{Fe1} that there are infinitely many
examples of $\rrrr$-covered Anosov flows in hyperbolic
$3$-manifolds.

The first result in hyperbolic manifolds was 
that any $\rrrr$-covered Anosov flow in a
closed  hyperbolic $3$-manifold
is not quasigeodesic \cite{Fe1}. The reason is as follows:
Let $\Phi$ be an $\rrrr$-covered
Anosov flow in a closed hyperbolic $3$-manifold.
Suppose that $\Phi$ is quasigeodesic.
Notice that an $\rrrr$-covered Anosov flow is necessarily
transitive \cite{Ba1}. 
Then since $\Phi$ is transitive, it follows that $\Phi$ is
uniformly quasigeodesic,
meaning that the quasigeodesic constants are uniform
over all orbits.
For any $\rrrr$-covered Anosov flow in a hyperbolic
$3$-manifold, then up to a double
cover where the stable foliation
is transversely orientable, the following happens:
every periodic orbit is freely homotopic
to infinitely many other periodic orbits \cite{Fe1}. 
Coherent lifts of the freely homotopic orbits all
have the same ideal points in the sphere at infinity
$\si$ of the canonical compactification of the
universal cover.
The uniform quasigeodesic behavior implies 
that the coherent lifts of the periodic
orbits are a globally bounded Hausdorff distance from
each other
\cite{Th1,Th2,Gr}.
Hence they accumulate in the universal cover.
This is impossible for coherent lifts of freely homotopic
periodic orbits.

In this article we obtain a definitive answer to the
quasigeodesic question for Anosov flows
in hyperbolic $3$-manifolds, by
proving that $\rrrr$-covered behavior
is the only obstruction:

\begin{mainth}
Let $\Phi$ be a topological Anosov flow in a 
closed hyperbolic
$3$-manifold.  Suppose that $\Phi$ is not $\rrrr$-covered.
Then $\Phi$ is a quasigeodesic flow.
\end{mainth}

\begin{convention}
\ 1) Any topological Anosov flow in a
hyperbolic $3$-manifold is orbitally equivalent to
a (smooth) Anosov flow by recent results of Shannon 
\cite{Sha}.
Our results and techniques are not dependent on the
regularity of the flow and its foliations. Hence
we will abuse terminology, and many times refer to our
flows as Anosov flows, even when they may only be
topological Anosov flows.

In addition the Main Theorem concerns behavior of the foliations
in the universal cover, so it is unaffected by taking finite
lifts. Therefore whenever necessary we assume that the manifold
is orientable and/or the foliations are transversely orientable.
But orientability is not necessary for many intermediate steps
in the proof.

2) A topological Anosov flow in a $3$-manifold does not necessarily
have invariant one dimensional stable and unstable bundles,
and hence no strong stable and unstable foliations.
Our convention in this article is to denote
the flow invariant two dimensional foliations as the 
stable and unstable foliations. 
This convention will be observed
throughout the article.
Beware that elsewhere
in the literature the term stable foliation for an
Anosov flow in a $3$-manifold 
many times denotes the strong stable foliation, and
the $2$-dim foliation is then called the weak stable
foliation.
We again stress that in the case of a topological
Anosov flow we do not necessarily have
the strong stable foliation.
\end{convention}

The Main theorem has many consequences:

\begin{theorem} \label{qgflow}
Let $M$ be  a hyperbolic $3$-manifold
admitting an Anosov flow. Then up to perhaps a double
cover, $M$ admits a quasigeodesic pseudo-Anosov flow.
In any case $M$ admits a one dimensional foliation
by quasigeodesics, with a dense leaf.
\end{theorem}

Roughly this goes as follows:
if the flow is not $\rrrr$-covered we use the Main Theorem.
Otherwise up to a double cover if necessary, 
the stable foliation is $\rrrr$-covered and
transversely orientable. Then one can produce
a quasigeodesic pseudo-Anosov flow transverse to the
stable foliation \cite{Th3,Cal1,Fe6}.
We explain in detail in Section \ref{moreqg} that the double
cover may be necessary: this happens for example for
many Fried surgeries on geodesic flows of closed,
non orientable hyperbolic surfaces.

Another consequence concerns the continuous extension
property. Suppose a topological  Anosov flow $\Phi$
in a $3$-manifold $M$ has stable leaves which are $C^1$.
This can always be achieved up to orbit equivalence
\cite{BFP}.
The leaves of the stable foliation are Gromov
hyperbolic \cite{Pla,Sul,Gr}.
Hence if $F$ is a leaf of the stable foliation
in the universal cover $\mt$, then it has a 
canonical compactification
into a closed disk $F \cup S^1(F)$. The 
{\em continous extension property} asks whether the
embedding $F \to \mt$ extends continuously to a map
$F \cup S^1(F) \to \mt \cup {\mathbb S}^2_{\infty}$,
see \cite{Ga} and Question 10.2 of \cite{Cal2}:

The continuous extension property was first proved for
fibrations over the circle in the celebrated article
by Cannon and Thurston \cite{Ca-Th}. This was an unexpected
and very surprising result. Since then it has been extended
to some classes of foliations: finite depth foliations
\cite{Fe5},  $\rrrr$-covered foliations 
\cite{Th3,Fe6}, and foliations with one sided 
branching \cite{Fe6}. In general it is a very hard property to
prove. For Anosov foliations this question is
very  complex, in part because these foliations 
do not have compact leaves.
In this article we prove:

\begin{theorem} \label{continuous}
Let $\Phi$ be an Anosov flow in a hyperbolic $3$-manifold.
Then the stable and unstable foliations of $\Phi$ have
the continuous extension property.
\end{theorem}

The non $\rrrr$-covered case is entirely new in terms of techniques:
it is the first result proving the continuous extension property,
where the key tool is not a transverse or almost transverse
 pseudo-Anosov flow
which is quasigeodesic.

The quasigeodesic property 
also has consequences for the computation of the Thurston
norm. We refer the reader to \cite{Mos2} for that.

Finally the main theorem implies the existence of group
invariant Peano curves as follows. Given an Anosov or
pseudo-Anosov flow
$\Phi$ in a closed $3$-manifold $M$, let $\oo$ be the
orbit space of the lifted flow $\wwp$ in the universal
cover $\mt$. This orbit space is always homeomorphic 
to the plane $\rrrr^2$ \cite{Fe1,Fe-Mo}.
The orbit space has one dimensional, possibly singular
foliations $\oo^s, \oo^u$ which are the projections to $\oo$ of
the two dimensional stable and unstable foliations in $\mt$. 
In \cite{Fe6} we produce an ideal boundary $\partial \oo$ and
ideal compactification $\oo \cup \partial \oo$ using only
the one dimensional foliations $\oo^s, \oo^u$. The ideal
boundary is always homeomorphic to a circle and the compactification
is homeomorphic to a closed disk. In addition the fundamental
group $\pi_1(M)$ naturally acts on all these objects.
The results of this article lead to Peano curves associated with
Anosov flows:

\begin{theorem} \label{groupin}
Let $M$ be a hyperbolic $3$-manifold admitting an Anosov flow
$\Phi$. Then there is a group invariant Peano curve
$\eta: \partial \oo \to {\mathbb S}^2_{\infty}$ where
$\partial \oo$ is boundary of the orbit space
$\oo$ of a quasigeodesic
flow in either $M$ or a double cover of $M$.
\end{theorem}

The quasigeodesic flow in Theorem \ref{groupin}
is the original flow $\Phi$
in the case that $\Phi$ is quasigeodesic $-$ that is,
when $\Phi$ is not $\rrrr$-covered; \ or
it is a pseudo-Anosov flow which is transverse to (say) the
stable foliation of either $\Phi$ in $M$ or the lift to
a double cover of $M$ $-$ in the case that $\Phi$ is $\rrrr$-covered.

The first result concerning group invariant Peano curves
was again proved by Cannon and Thurston in the same seminal paper
\cite{Ca-Th}. 
They considered fibers of fibrations, which lift to properly
embedded planes in $\mt$. They proved that the embeddings 
in $\mt$ extend to the
ideal compactifications, and proved that 
the ideal maps are group invariant Peano curves.
There are results concerning Peano curves for 
quasigeodesic pseudo-Anosov flows in hyperbolic $3$-manifolds
\cite{Fe6,Fe7}.
There are similar  results for general quasigeodesic
flows in hyperbolic $3$-manifolfds 
\cite{Fra1,Fra2}. 
Finally there are also many results on
Cannon-Thurston maps from 
the Kleinian group setting, see for example \cite{Mj}.

\begin{remark} To prove the Main theorem 
we only use that
$M$ is atoroidal. By the geometrization theorem
of Perelman this implies that $M$ is hyperbolic,
but that is an extremely hard result.
In this article we will invariably state results and do
arguments using atoroidal as an hypothesis.
By atoroidal we mean ``homotopically atoroidal", that is,
there is no $\pi_1$-injective immersion of a torus into $M$.
There is a slight difference between homotopically atoroidal
and ``geometrically atoroidal" $-$ the second condition
means that there is an embedded incompressible torus 
in $M$. If $M$ has an Anosov flow,
then $M$ is irreducible, so the difference is only for
small Seifert fibered spaces. But there are small Seifert
fibered spaces that admit Anosov flows. 
For example let $S$ be a compact hyperbolic $2$-dimensional
orbifold which has a finite cover which is a surface, and
$S$ is a sphere with $3$ cone points. Then $M = T^1 S$
is a small Seifert fibered space and the geodesic flow
of $S$ 
in $M$ is an Anosov flow.
These small Seifert fibered spaces are geometrically
atoroidal but not homotopically atoroidal.
\end{remark}

\subsection{Examples of $\rrrr$-covered and non $\rrrr$-covered
Anosov flows in hyperbolic $3$-manifolds}
\label{examples}

The results of this article obviously 
lead naturally to the question as to 
the existence of both of these classes and how widespread they
are. It turns out that both $\rrrr$-covered and non
$\rrrr$-covered Anosov flows are very common in hyperbolic
$3$-manifolds. We already mentioned the existence of
infinitely many $\rrrr$-covered examples obtained
by some  Dehn surgeries on closed orbits of suspensions
and geodesic flows \cite{Fe1}. On the other hand any Anosov flow
in a non orientable hyperbolic $3$-manifold is 
non $\rrrr$-covered \cite{Fe4} and one can obtain these
by doing Dehn surgery on suspensions of orientation reversing
hyperbolic diffeomorphisms of the torus.

The wealth of examples was enormously extended by a
recent article of Bonatti and Iakovouglou \cite{Bo-Ia}.
We just mention some results of \cite{Bo-Ia} that show
that both classes are extremely large.

To obtain examples of non $\rrrr$-covered Anosov flows:
Theorems 2 and 3 of \cite{Bo-Ia} 
state that starting with a non $\rrrr$-covered Anosov flow
and doing appropriate Fried Dehn surgeries one obtains
a non $\rrrr$-covered Anosov flow. To get examples in hyperbolic
$3$-manifolds, start with a non $\rrrr$-covered Anosov flow
in a hyperbolic $3$-manifold. Most Dehn surgeries will produce
hyperbolic $3$-manifolds.  Theorems 5 and 6 of \cite{Bo-Ia}
provide many more examples which are non $\rrrr$-covered. 

To obtain examples of $\rrrr$-covered Anosov flows:
Theorems 4 and 7 of \cite{Bo-Ia} produce a wide class
of $\rrrr$-covered Anosov flows by doing Dehn surgeries
on suspensions. If the surgery coefficients are big
in absolute value, then the resulting manifold is
hyperbolic. Corollary 4.1 and Proposition 4.1 of \cite{Bo-Ia}
produce $\rrrr$-covered examples starting with geodesic
flows.

There are many more results in \cite{Bo-Ia} showing
either the $\rrrr$-covered behavior or the non $\rrrr$-covered
behavior for Anosov foliations in $3$-manifolds.

Many other possible examples can possibly be obtained as follows:
start with an example as constructed in \cite{BBY} using
hyperbolic plugs. This is an extremely general
construction producing an enormous amount
of Anosov flows. Do Dehn surgeries on some periodic orbits
to eliminate all tori. The resulting manifold is hyperbolic.
This specific situation was analyzed by Beguin and Yu
\cite{BY} who produced many examples where the flow
is not $\mathbb{R}$-covered, and large classes
of non orbit equivalent such flows.

\subsection{Some ideas on the proof of the Main Theorem}
\label{some}

We will prove the contrapositive of the statement of the
Main theorem.
As a starting point of the analysis we use the following result:
It was proved in \cite{Fe7} that a topological Anosov flow $\Phi$
(even a pseudo-Anosov flow) in a hyperbolic $3$-manifold
is quasigeodesic if and only if there is a global bound $K$,
so that any set of pairwise freely homotopic periodic
orbits of $\Phi$ has cardinality bounded above by $K$.
The strategy of the proof of our Main 
theorem is to assume that $\Phi$ is not quasigeodesic,
hence this boundedness condition
fails. We then prove that this implies that $\Phi$ is
$\rrrr$-covered.

In very rough terms what we will do is the following:
assume that the flow is not quasigeodesic. Then
we will produce
an appropriate essential lamination in $M$ which 
will imply that the flow is $\rrrr$-covered.

To do that we will employ a standard form for the flow
as follows. Using resuls of Candel \cite{Cand} 
it was proved previously in \cite{BFP} that $\Phi$
is orbitally equivalent to a topological Anosov flow
$\Psi$
so that the stable leaves of $\Psi$ are $C^1$ with a leafwise
Riemannian metric in $M$ making every stable leaf a hyperbolic
surface (meaning Gaussian curvature constant and 
equal to $-1$). In 
addition $\Psi$ can be chosen so that the flow lines of $\Psi$
are geodesics in the stable leaves.
Nothing is claimed about the unstable leaves of $\Psi$.

Replace $\Phi$ by $\Psi$ if necessary and assume that
$\Phi$ has all the properties described above.

By assumption we have arbitrarily big sets of pairwise 
freely homotopic periodic orbits.
We realize the free homotopies between the periodic orbits
so that intersections with stable leaves are geodesics
in the respective leaves. 
These are immersed annuli in $M$.
These immersed annuli are chosen so that they 
are uniquely determined by
the periodic orbits.
If there are more orbits freely homotopic to a given
orbit then the annuli can be put together forming
bigger annuli.
In other words one can concatenate
the free homotopies in a nice way.
We take limits of these big annuli when the number of freely
homotopic  orbits goes to infinity.
We can take limits and they are well behaved  in part 
because of the canonical form of the free homotopy annuli:
the stable leaves are hyperbolic surfaces and the intersections
of the annuli with these stable leaves are geodesics
in the leaves. This is one main reason for the setup with
the metric on the stable leaves and the flow lines 
geodesics in these stable leaves.
Even with this setup it is very difficult to understand
the limits. We will explain more in Subsection \ref{more}
once we have defined lozenges and other objects
associated with free homotopies.
The limits of the bigger and bigger
free homotopy annuli, when lifted to the universal cover, are
what we call {\em walls}.
The walls are properly embedded planes
in the universal cover, and
intersect an $\rrrr$ worth set of stable leaves.
The goal is to prove that this is the set of all stable
leaves in $\mt$, and hence $\Phi$ is $\rrrr$-covered.
To do that we study properties of these walls and certain
subsets of the orbit space in $\mt$, called {\em bi-infinite blocks},
 associated
with the walls.
A lot of the analysis is done without assuming that
$M$ is atoroidal. But the atoroidal hypothesis is used
in many places to obtain stronger results.

One goal is to understand the projection of the walls in 
$M$. There are two main cases: either the projection
of a wall does not self intersect transversely or it does.
In the first case one shows, in the case that
$M$ is atoroidal, that the
projection of a wall to $M$ has closure which is a
lamination $-$ in other words any tangent intersection
of two walls 
leads to equality of these walls. This last fact in general
holds only when $M$ is atoroidal.
We study the complementary regions of this
lamination and eventually prove that $\Phi$ is $\rrrr$-covered.
In the second case there are transverse
self intersections when projecting to $M$.
We then study a type of ``convex hull" 
in $\mt$ 
associated with such a wall. 
This convex hull is obtained using the hyperbolic
metric in the stable leaves and using that walls intersect
such leaves in geodesics. There are two subcases: if the
convex hull is all of $\mt$, then the $\rrrr$-covered
property follows. Otherwise we do a very careful
analysis of the boundary
components of the convex hull.
Using this analysis we are then able to produce walls
without transverse self intersections when projected to $M$,
and the proof reduces to the first case.

\begin{remark}
As explained in the beginning of this description
the lack of quasigeodesic
behavior is equivalent to a property involving the number
of freely homotopic periodic orbits, and that is what 
we analyze in this article. In particular we stress 
that the arguments
in this article
never analyze the quasigeodesic property for flow lines
directly. As explained in the description above,
we will analyze freely homotopic periodic orbits
and sets in $M$ or $\mt$ which are generated by them,
which will eventually connect with the $\rrrr$-covered property.
On the other hand, it is true that 
geometric properties of flow lines
are used in several places: for example a global bound of the
Hausdorff distance between corner orbits of lozenges,
also a global bound of the Hausdorff distance between
boundary orbits of bands if the flow does not
have scalloped regions (lozenges and bands to be defined
later).
These properties are true for any Anosov 
flow in a $3$-manifold, and not just in hyperbolic manifolds.
\end{remark}

We remark again that
we will provide more details about the proof of the Main
Theorem in Subsection \ref{more} after we have
defined certain terms used in the proof, such 
as for example perfect fits and
lozenges.

\subsection{Brief description of sections}

The next section describes the main objects and properties
of topological Anosov flows.
The Main Theorem is proved in Sections \ref{limitsoflozenges}
to \ref{completion}. 
Section \ref{limitsoflozenges} analyzes bands and
limits of lozenges and bands.
Section \ref{construction} describes walls associated with
the limits of bands.
The case with a wall with no transverse
self intersections when projected to $M$ is 
worked out in Sections \ref{translates} and
\ref{notransverse}. The case where all walls self intersect
transversely when projected to $M$ is worked out in
Sections \ref{transverse1} to \ref{completion}.
Many examples of walls are described in detail
in Section \ref{examp}.
Theorem \ref{qgflow}, concerning 
 production of quasigeodesic flows,
 is proved in Subsection \ref{moreqg}.
Theorem \ref{continuous}, which proves the continuous
extension for Anosov foliations, 
 is proved in Subsection \ref{continuousextension}.
Theorem \ref{groupin}, which produces group
invariant Peano curves,  is proved in Subsection \ref{groupinvar}.
The final section, Section \ref{future} explains some
open questions in the subject.

\subsection{Acknowledgements}  \
We thank the reviewer of this article who did an outstanding
job, with many corrections and suggestions that substantially helped the
presentation of this article.
We also thank Rafael Potrie for many suggestions to a preliminary
version of this article. We thank Thomas Barthelme who
suggested Theorem \ref{qgflow} to us.

\section{Background on topological Anosov flows}
\label{prelim}

The manifold $M$ is equipped with a Riemannian metric.
For definitions and fundamental properties and results on Anosov flows 
we refer to \cite{An,Ka-Ha}.
In this article we will consider a generalization of
this:

\begin{define}{}{(topological Anosov flow)}
Let $\Phi$ be a flow on a closed
$3$-manifold $M$. 
We say that $\Phi$ is a {\em topological Anosov flow}
if the following 
conditions are
satisfied:

- For each $x \in M$, the flow line $t \to \Phi(x,t)$ is $C^1$,
and it is not a single point.
The tangent vector bundle $D_t \Phi$ is $C^0$ in $M$.

-There are two transverse topological foliations
$\ls, \lu$ which are two dimensional, with leaves
saturated by the flow and so that $\ls, \lu$ 
intersect exactly along the flow lines of $\Phi$.
These are called the stable ($\ls$) and the unstable
($\lu$) foliations of the flow $\Phi$.

- In a stable leaf all orbits are forward asymptotic.
In addition there is $\eps > 0$ so that any two distinct
orbits in a stable leaf, they eventually in the past
always stay more than $\eps$ apart from each other (easiest 
seen in the universal cover). 
Similarly for the unstable foliation.
\end{define}




We remark that a priori there are no stable and unstable
one dimensional 
bundles associated with a topological Anosov flow \cite{BFP}.
Hence, in general there are no one dimensional strong stable
and strong unstable foliations. We only have the two dimensional
stable and unstable 
foliations $\ls, \lu$.

As mentioned in the introduction, if $M$ is atoroidal,
then $\Phi$ is orbitally equivalent to an Anosov flow
\cite{Sha}. For this reason we will abuse notation and
many times 
refer to our flows as Anosov flows. The arguments in this
article only use the topological Anosov flow behavior.

\vskip .05in
\noindent
\underline {Notation/definition:} \ 
We denote by $\pi: \mt \rightarrow M$ the universal
covering of $M$, and by $\pi_1(M)$ the fundamental group
of $M$, considered as the group of deck transformations
on $\mt$.
The lifted 
foliations to $\mt$ are
denoted by $\wls, \wlu$.
If $x \in M$ let $\ls(x)$ denote the leaf of $\ls$ containing
$x$.  Similarly one defines $\lu(x)$
and in the
universal cover $\wls(x), \wlu(x)$.
If $\alpha$ is an orbit of $\Phi$, similarly define
$\ls(\alpha)$, 
$\lu(\alpha)$, etc...
Let also $\wwp$ be the lifted flow to $\mt$.

\vskip .1in
The deck transformations will also be called covering
translations.

We review the results about the topology of
$\wls, \wlu$ that we will need.
We refer to \cite{Fe4,Fe6} for more details and proofs.
We note that in \cite{Fe4} the definitions and results
are stated for the more general situation of pseudo-Anosov
flows, which imply the statements here.
In a pseudo-Anosov flow one allows finitely many
$p$-prong orbits, with $p \geq 3$.
The orbit space is denoted by $\oo$ which is the quotient
space $\mt/\wwp$. 
The orbit space of $\wwp$ in
$\mt$ is homeomorphic to the plane \cite{Fe1,Fe4}.
There is an induced action of $\pi_1(M)$ on $\oo$. Let

$$\Theta: \mt \rightarrow \oo \cong \rrrr^2$$

\noindent
be the projection map. 
It is naturally $\pi_1(M)$ equivariant.
If $L$ is a 
leaf of $\wls$ or $\wlu$,
then $\Theta(L) \subset \oo$ is homeomorphic
to $\rrrr$.
The foliations $\wls, \wlu$ induce $\pi_1(M)$ invariant
$1$-dimensional foliations
$\oos, \oou$ in $\oo$. 
If $x$ is a point of $\oo$, then $\oos(x)$ (resp. $\oou(x)$)
is the leaf of $\oos$ (resp. $\oou$) containing $x$.

A ray in a leaf $\ell$ of $\oos$ or $\oou$ is a connected
component of $\ell \setminus \{ x \}$ where $x$ is a point
in $\ell$. 

\begin{define}
Let $L$ be a leaf of $\wls$ or $\wlu$. 
If $\zeta$ is a ray in $\Theta(L)$ then $A = \Theta^{-1}(\zeta)$
is called a half leaf of $L$.
If $\zeta$ is an open segment in $\Theta(L)$ 
with compact closure,
we define a flow band of $L$ (denoted by $L_1$)
as $L_1 := \Theta^{-1}(\zeta)$.
\end{define}

We abuse convention and call a leaf $L$ of $\wls$ or $\wlu$
{\em periodic} if there is a non trivial covering translation
$g$ of $\mt$ with $g(L) = L$. Equivalently $\pi(L)$ contains
a periodic orbit of $\Phi$. In the same way an orbit
$\gamma$ of $\wwp$ is {\em periodic} if $\pi(\gamma)$ is a 
periodic orbit of $\Phi$. Observe that in general the isotropy group
of an element $\alpha$ of $\oo$ is either trivial, or 
an infinite cyclic subgroup of $\pi_1(M)$.

If $F \in \wls$ and $G \in \wlu$ 
then $F$ and $G$ intersect in at most one
orbit. 
Also suppose that a leaf $F \in \wls$ intersects two leaves
$G, H \in \wlu$ and so does $L \in \wls$.
Then $F, L, G, H$ form a {\em rectangle} in $\mt$.

The following two generalizations of rectangles will
be extremely important in this article:
\ 1) perfect fits $=$ 
the projection
in the orbit space is a  properly embedded rectangle
with one corner removed;
and \ 2) lozenges $=$ the projection in $\oo$
is a rectangle with two opposite corners removed.

\begin{figure}[ht]
\begin{center}
\includegraphics[scale=0.90]{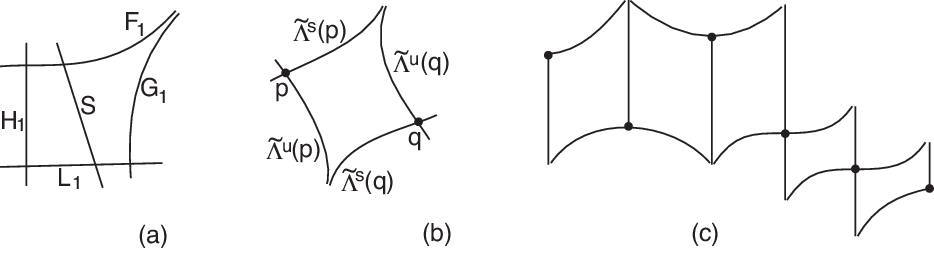}
\begin{picture}(0,0)
%
%
\end{picture}
\end{center}
\vspace{0.0cm}
\caption{a. A perfect fit. The half leaves $F_1, G_1$ of $F, G$
depict the perfect fit between $F, G$. Here $F_1, G_1$ do 
not intersect, but any $S$ intersecting $L$ between 
$G$ and $H$ intersects $F_1$;,
b.  A lozenge. $p, q$ are the corners of the lozenge;
c. A chain of lozenges.}
\label{figure1}
\end{figure}

\begin{define}{(\cite{Fe2,Fe4})}{}
Perfect fits -
Two leaves $F \in \wls$ and $G \in \wlu$, form
a perfect fit if $F \cap G = \emptyset$ and there
are half leaves $F_1$ of $F$ and $G_1$ of $G$ 
and also flow bands $L_1 \subset L \in \wls$ and
$H_1 \subset H \in \wlu$,
so that 
%
the set 

$$\overline F_1 \cup \overline H_1 \cup 
\overline L_1 \cup \overline G_1$$

\noindent
separates $M$,
and its projection to $\oo$ forms a rectangle $R$ with a corner removed:
The joint structure of $\oos, \oou$ in 
$R$ is that of
a rectangle without one corner (which in $\mt$ corresponds
to a ``missing" orbit).
The removed corner in $R$ corresponds to the perfect
fit of $F$ and $G$ which do not intersect.
We also use the terminology perfect fits in $\oo$, and say
that $\Theta(F), \Theta(G)$ make a perfect fit.
\end{define}

We refer to Figure \ref{figure1} (a) for perfect fits.
There is a product structure in the interior of $R$: there are
two stable boundary sides and two unstable boundary sides 
in $R$. An unstable
leaf intersects one stable boundary side (not in the corner) if
and only if it intersects the other stable boundary side
(not in the corner).

%
%
%
%

\begin{define}{(\cite{Fe4})}{}
Lozenges - A lozenge in $\oo$ is a region of $\oo$ whose closure
is homeomorphic to a rectangle with two corners removed.
More specifically two points $x, y$ define the corners 
of a lonzenge if there are 
rays $a, b$ of $\oos(x), \oou(x)$ defined by $x$,
and $c, d$ rays of $\oos(y), \oou(y)$ defined by $y$, 
so that $a$ and $d$ form a perfect fit, and so do $b$ and $c$.
The rays above are the sides of the lozenge,
and are not contained in the lozenge, but are in the
boundary of the lozenge.
We use the same terminology of lozenges in $\mt$, these
are just the pullback by $\Theta^{-1}$ of lozenges in $\oo$,
specifically: 
there are half leaves $A, B$ of
$\wls(p), \wlu(p)$ defined by $p$
and  $C, D$ half leaves of $\wls(q), \wlu(q)$ so
that $A$ and $D$ form a perfect fit and so do
$B$ and $C$. The sides of the lozenge are $A, B, C, D$.
The sides are not contained in the lozenge, but
are in the boundary of the lozenge.
See Figure \ref{figure1} (b).
\end{define}

The previous definition is Definition 4.4 of \cite{Fe4}.
Two lozenges are {\em adjacent} if they share a corner and
there is a stable or unstable leaf
intersecting both of the lozenges, see Figure \ref{figure1} (c).
Therefore they share a side.
A {\em chain of lozenges} is a collection $\{ \mathcal C _i \}, 
i \in I$, of lozenges
where $I$ is an interval (finite or not) in ${\bf Z}$,
so that if $i, i+1 \in I$, then 
${\mathcal C}_i$ and ${\mathcal C}_{i+1}$ share
a corner, see Figure \ref{figure1}, c.
Consecutive lozenges may be adjacent or not.
The chain is finite if $I$ is finite.
A chain $\mathcal C$ is a {\em string of lozenges} if 
no two consecutive lozenges are adjacent. The chain depicted
in (c) of Figure \ref{figure1} is not a string of lozenges.
But the subchain of the right $3$ lozenges in this figure is
a string of lozenges.

\vskip .05in
\noindent
{\bf {Important convention}} $-$ 
We also use the terms half leaves, 
perfect fits, lozenges and rectangles for the projections of these
objects in $\mt$ to the orbit space $\oo$.

\begin{define}{}{} Product region $-$ 
Suppose $A$ is a flow band in a leaf of $\wls$.
Suppose that for each orbit $\gamma$ of $\wwp$ in $A$ there is a
half leaf $B_{\gamma}$ of $\wlu(\gamma)$ with boundary
$\gamma$ so that: 
for any two orbits $\gamma, \beta$ in $A$ then
a stable leaf intersects $B_{\beta}$ if and only if 
it intersects $B_{\gamma}$.
%
%
This defines a stable product region which is the union
of the $B_{\gamma}$.
Similarly define unstable product regions.
\label{defsta}
\end{define}

It follows that $B_{\gamma}$ varies continuously with $\gamma$
(\cite{Fe4} page 641).

\begin{theorem}{(\cite{Fe4}, Theorem 4.10)}{} \label{product}
Let $\Phi$ be a topological Anosov flow. Suppose that there is
a stable or unstable product region. Then $\Phi$ is 
topologically conjugate to a suspension Anosov flow.
\label{prod}
\end{theorem}

We say that 
two orbits $\gamma, \alpha$ of $\wwp$ 
(or the leaves $\wls(\gamma), \wls(\alpha)$)
are connected by a 
chain of lozenges $\{ {\mathcal C}_i \}, 1 \leq i \leq n$,
if $\gamma$ is a corner of ${\mathcal C}_1$ and $\alpha$ 
is a corner of ${\mathcal C}_n$.
If a lozenge ${\mathcal C}$ has corners $\beta, \gamma$ and
if $g$ in $\pi_1(M) - \{ id \}$ satisfies $g(\beta) = \beta$,
$g(\gamma) = \gamma$ (and so $g({\mathcal C}) = {\mathcal C}$),
then $\pi(\beta), \pi(\gamma)$ are closed orbits
of $\Phi$ which are freely homotopic to the inverse of each
other.

\begin{theorem}{(\cite{Fe4}, Theorem 4.8)}{} \label{freelyhomotopic}
Let $\Phi$ be a topological Anosov flow in $M$ closed and let 
$F_0 \not = F_1 \in \wls$.
Suppose that there is a non trivial covering translation $g$
with $g(F_i) = F_i, i = 0,1$.
Let $\alpha_i, i = 0,1$ be the periodic orbits of $\wwp$
in $F_i$ so that $g(\alpha_i) = \alpha_i$.
Then $\alpha_0$ and $\alpha_1$ are connected
by a finite chain of lozenges 
$\{ {\mathcal C}_i \}, 1 \leq i \leq n$ and $g$
leaves invariant each lozenge 
${\mathcal C}_i$ as well as their corners.
\label{chain}
\end{theorem}


This means that each free homotopy of closed orbits generates a chain of lozenges
preserved by a non trivial element $g$ of $\pi_1(M)$.

The Main Theorem concerns the leaf spaces of $\wls, \wlu$. 
If this leaf space is Hausdorff, then it is homeomorphic
to $\rrrr$ \cite{Fe1}.
Suppose that the leaf space of $\wls$ (or $\wlu$) is not Hausdorff.
Two points of this space are non separated if they
do not have disjoint neighborhoods in their respective leaf
space.
For the Main Theorem
it is very important to understand the structure of non
separated leaves.
Here is one example: suppose that $C, D$ are adjacent lozenges
intersecting a common stable leaf.
A possible depiction of this is in Figure \ref{figure1} (c), the
first two lozenges on the left part of the
figure.
The lozenges share a corner $\gamma$
and the stable leaf $\wls(\gamma)$ is entirely contained
in the boundary of $C \cup D$. 
There are two other stable leaves $S, E$ which have half
leaves in the boundary of $C$ and $D$ respectively
By the definition of lozenges, it follows that $S, E$ are
distinct, but they are non separated from each other.

As it turns out the structure of non separated leaves in general
is very similar to the above, and it exhibits a rigid behavior.
The main properties concerning non Hausdorff behavior in the leaf spaces
of $\wls, \wlu$ are gathered in the following statement:

\begin{theorem}{(\cite{Fe4}, Theorem 4.9)}{}\label{nonsep}
Let $\Phi$ be a topological Anosov flow in $M^3$. 
Suppose that $F \not = L$
are not separated in the leaf space of $\wls$.
Let $V_0$ be the component of $\mt - F$ 
containing $L$.
The following happens:
\begin{itemize}
\item{$F$ and $L$  are periodic.}
\item{Let $\alpha$ be the periodic orbit in $F$ and
$\beta$ the periodic orbit in $L$.
$H_0$ be the component of $(\wlu(\alpha) - \alpha)$ 
contained in $V_0$.}
\item{Let $g$ be a non trivial covering translation
with $g(H_0) = H_0$,
and $g$ leaves
invariant the components of $(F - \alpha)$.
Then $g(L) = L$.}
\item{This produces periodic orbits $\pi(\alpha), \pi(\beta)$
of $\Phi$ which are freely
homotopic in $M$.}
\item{Theorem \ref{chain} then implies that $\alpha$ and $\beta$
are connected by
a finite chain of lozenges 
$\{ A_i \}, 1 \leq i \leq n$,
where consecutive lozenges are adjacent and all intersect
a common stable leaf $E$.}
\item{There is an even number of lozenges 
in the chain, see
Figure \ref{figure2}.}
\item{In addition 
let ${\mathcal B}_{F,L}$ be the set of leaves of $\wls$ non separated
from $F$ and $L$.
Put an order in ${\mathcal B}_{F,L}$ as follows:
The set of orbits of $E$ contained
in the union of the lozenges and their sides is an interval.
Put an order in this interval.
If $R_1, R_2 \in {\mathcal B}_{F,L}$ let $\alpha_1, \alpha_2$
be the respective periodic orbits in $R_1, R_2$. Then
$\wlu(\alpha_i) \cap E \not = \emptyset$ and let 
$a_i = \wlu(\alpha_i) \cap C$.
We define $R_1 < R_2$ in ${\mathcal B}_{F,L}$ 
if $a_1$ precedes
$a_2$ in the order of the set of orbits of 
$C$.}
\item{Then ${\mathcal B}_{F,L}$ with this order,
is either order isomorphic to $\{ 1, ..., n \}$ for some
$n \in {\bf N}$; or ${\mathcal B}_{F,L}$ is order
isomorphic to the integers ${\bf Z}$.}
\item{In addition if there are $Z, S \in \wls$ so that
${\mathcal B}_{Z, S}$ is infinite, then there is 
an incompressible torus in $M$ transverse to 
$\Phi$. In particular $M$ cannot be atoroidal.}
\item{Also if there are $F, L$ distinct leaves of 
$\wls$  non separated from each other as above, then there are
periodic orbits $\tau, \eta$ of $\Phi$ which
are freely homotopic to the inverse of each other.}
\item{Finally up to covering translations,
there are only finitely many non Hausdorff
points in the leaf space of $\wls$.}
\end{itemize}
The same happens for $\lu$.
\label{theb}
\end{theorem}

\begin{figure}[ht]
\begin{center}
\includegraphics[scale=1.20]{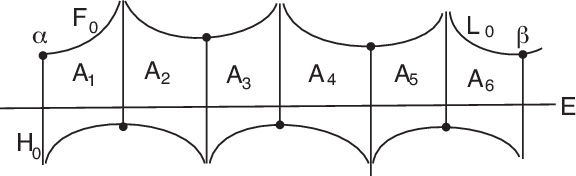}
\begin{picture}(0,0)
%
%
\end{picture}
\end{center}
\vspace{0.0cm}
\caption{A chain of adjacent lozenges $\{ A_i, 1 \leq i \leq n \}$
 connecting
non separated leaves $F, L$ in $\widetilde \Lambda^s$.
Here $F_0$ is a half leaf of $F$ contained in the boundary
of $A_1$ and similarly $L_0$ is a half leaf of $L$ contained
in the boundary of $A_6$.}
\label{figure2}
\end{figure}

Notice that ${\mathcal B}_{F,L}$ is a discrete set in this order.
Many of the results mentioned in the above theorem are not actually
proved in \cite{Fe4}, but this is one place where they are all
stated together.
For detailed explanations and proofs, see
\cite{Fe4}.

\begin{define}{}{}
A product open set is an open set $Y$ in  $\oo$ so that
the induced stable and unstable foliations in $Y$ satisfy the
following: every stable leaf intersects every unstable leaf.
A product open set $Z$ in $\mt$ is $\Theta^{-1}(Y)$ where
$Y$ is a product open set in $\oo$, in particular $Z$
is flow saturated.
\end{define}

For example, if stable leaves $A, B$ both intersect unstable
leaves $C, D$, then the four of them bound a compact rectangle 
whose interior is 
a product open set. In our explorations we will also
consider the closure of the open product set (in $\mt$ or $\oo$) which is the
open set union its boundary components.

We stress that {\underline {product open sets}}
and {\underline {product regions}} in general are not the same,
and the second class is a subset of the first. Any pseudo-Anosov flow has
product open sets, but only suspensions have product regions.

We finish this section with $3$ more technical results, which
are from \cite{Fe7}.

\begin{proposition} (Proposition 4.6 of \cite{Fe7} - periodic
double perfect fits) \label{periodicdouble}
Suppose that two disjoint half leaves of a leaf $S$ of 
$\wls$ (or $\wlu$) make a perfect fit respectively with
$A, B$ which are leaves of $\wlu$ (or $\wls$). Suppose that 
$A, B$ are in the same complementary component of $\mt - S$.
Then $A, B$ and $S$ are all periodic and left invariant
by the same non trivial covering translation $\zeta$.
\end{proposition}

\begin{theorem} (Theorem 5.2 of \cite{Fe7}) \label{lozengebound}
Let $\Phi$ be an Anosov flow. There is a constant $a_0$ depending
only on the geometry of $M$ and the flow $\Phi$, so that if
$\alpha, \beta$ are corner orbits of a lozenge $C$ in $\mt$,
then $\alpha, \beta$ are a bounded Hausdorff distance from
each other: $d_H(\alpha, \beta) < a_0$, where $d_H$ is the
Hausdorff distance in $\mt$.
\end{theorem}

Let $U$ be a product open set, say in $\oo$. A {\em corner}
of $U$ is a point $x$ in the closure of $U$ so that
there are rays $a$ and $b$ of $\oos(x), \oou(x)$ defined
by $x$ so that $a$ and $b$ are contained in the closure of $U$.
In particular $\bar{U} \cap \oos(x) = a$ and $\bar{U} \cap 
\oou(x) = b$. 
The same terminology of corners is used when the product open
set is a subset of $\mt$.

\begin{remark} 
There are at most finitely many non separated Hausdorff 
points in the leaf space of (say) $\wls$ up to deck
transformations \cite{Fe4}. Let $K$ be an upper bound
for this. Hence if $H$ is a set of pairwise non separated
leaves of $\wls$ which has cardinality bigger than $K$, it
follows that at least two of them are deck translates of
each other. Iterating by this deck translate implies that
there are in fact infinitely many leaves which
of $\wls$ which are pairwise non separated from each
other. This produces a region in the orbit space
which is called a {\em scalloped region}, which in 
particular implies the existence of a torus in $M$ transverse
to the flow which necessarily is incompressible \cite{Fe4}.
In particular the manifold cannot be atoroidal.
We do not define scalloped regions formally but refer
the reader to \cite{Fe4}.
\end{remark}

\begin{lemma} \label{noone}
Suppose that $\Phi$ is an Anosov flow which
is not orbitally equivalent to a suspension Anosov flow,
and does not admit a scalloped region.
Then there is no product open set $Q$ in $\oo$ so that
the boundary of $Q$ has a single corner or no corner.
In particular this holds if $M$ is atoroidal.
\end{lemma}

\begin{proof}
This is the statement of Lemma 4.5 of \cite{Fe7}, except
that the hypothesis there is that the flow $\Phi$ is bounded.
The bounded hypothesis (see Definition 1.1 of \cite{Fe7})
means that the cardinality of any set of pairwise
freely homotopic periodic orbits of $\Phi$
is bounded and $\Phi$ is not orbitally equivalent
to a suspension. The bounded hypothesis definitely does not
follow from the hypothesis of this lemma  $-$ for
example consider an $\rrrr$-covered Anosov flow in $M$
hyperbolic.
However, what is used in the proof of Lemma 4.5 of \cite{Fe7}, is the
weaker property that
the set of leaves non separated from a leaf in $\wls$ or $\wlu$
is finite. The previous remark shows that unless there are 
scalloped regions, such sets are bounded in cardinality,
and this is the case if $M$ is atoroidal.
Hence the same proof as the proof of Lemma 4.5 of \cite{Fe7} 
implies the result of this lemma.
\end{proof}

\section{Limits of lozenges and canonical bands}
\label{limitsoflozenges}

The proof of the Main Theorem will involve the analysis of
limits of free homotopies between periodic orbits.
By Theorem \ref{freelyhomotopic}, any free homotopy 
between periodic orbits is represented by a chain of
lozenges.
Therefore we will need to understand limits of
lozenges and limits of chains of lozenges.
This is the goal of this section.
When one takes
limits of the chains of lozenges we will obtain what 
we call blocks and bi-infinite blocks.
Also associated with the lozenges there will be geometric sets
in $\mt$ which are called bands, and the limits of the
geometric sets associated with chains of lozenges will be
called walls.
This analysis will involve
a type of product open sets called ideal
quadrilaterals.

\begin{define} An $(i,j)$ ideal quadrilateral in $\mt$ or $\oo$ 
is a product open set $Q$ satisfying the following conditions.
Here $i + j = 4$. The boundary of $Q$ is made up of $4$ pieces
$S_k, 1 \leq k \leq 4$ which are either leaves, half leaves
or flow bands alternatively in leaves of $\wls, \wlu$ so 
that either $S_k, S_{k+1} (mod 4)$ intersect in a corner
of $Q$ or make a perfect fit. These pieces in the boundary
are called the sides of $Q$.
\end{define}

We stress that this definition is not symmetric in $i, j$. 
Here $i$ is the number of perfect fits originating from $Q$
and $j$ is the number of corners of $Q$.

We will call a side of $Q$ compact if its projection to $\oo$ is
compact.

%

Now we describe the possibilities of $(i,j)$ ideal quadrilaterals.
First of all
a $(0,4)$ ideal quadrilateral is 
an actual quadrilateral or a rectangle in the orbit space.
The previous lemma shows that there are no
$(4,0)$ or $(3,1)$ ideal
quadrilaterals. The $(1,3)$ ideal quadrilaterals will be a
crucial ingredient in this article. Whenever two leaves make
a perfect fit one can produce infinitely many such 
$(1,3)$ ideal quadrilaterals,
by direct checking of the definition of perfect fits.
As for $(2,2)$ ideal quadrilaterals there are two possibilities:
the first is 
that there is no full leaf of $\wls$ or $\wlu$ which is a side of
$Q$,
in which case $Q$ is a lozenge. 
Equivalently the corners alternate with the perfect fits.
The other option is
that there is a full leaf, say $S_1$ contained in $\partial Q$.
In this case this full leaf $S_1$
makes a perfect fit with $S_2$ and $S_4$.
This can be obtained as the union of two $(1,3)$ quadrilaterals,
so that two compact sides are glued together.

Notice that there are other types of product open sets which
are not ideal quadrilaterals: one is depicted in 
Figure \ref{figure2},
the union $\cD$ of the lozenges $A_i, 1 \leq i \leq 6$,
 plus the half
leaves contained in the boundary of both $A_i$ and $A_{i+1}$. 
In this example
the  product open set $\cD$ has two unstable half leaves (subsets
of $\wlu(\alpha), \wlu(\beta)$) contained in the boundary. 
The boundary also contains $3$ stable leaves non separated from
each other, one of which makes a perfect fit with $\wlu(\alpha)$,
as well as parts of other $4$ stable leaves: $F_0, L_0$ and two
other stable leaves between $F, L$.

Two ideal quadrilaterals quadrilaterals $Q_1, Q_2$ are said to be
{\em adjacent} if there a side $S_1$ of $Q_1$ and a side 
$S_2$ of $Q_2$ which are the same set in $\mt$.
In particular if $S_1$ is either a flow band or a half leaf,
then $Q_1, Q_2$ share a corner.

\subsection{Blocks} \label{ss.blocks}

Lozenges and chains of lozenges have been extensively used
to study pseudo-Anosov flows, see for example
\cite{Fe4,Fe5,Fe7,BM,BFrM,BBM1,BBM2,BY} to cite a few references.
Here we have to consider a more general object which also
allows for $(1,3)$ ideal quadrilaterals.

\begin{figure}[ht]
\begin{center}
\includegraphics[scale=1.00]{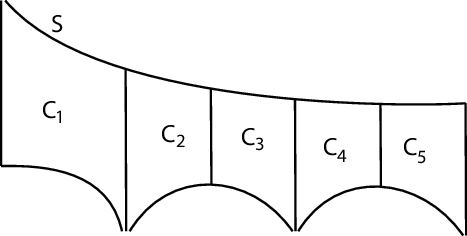}
\begin{picture}(0,0)
\end{picture}
\end{center}
\vspace{0.0cm}
\caption{A finite block. Assuming that $S$ is a stable leaf,
then this is a stable adjacent block. It has 5 elements,
$C_1, ...., C_5$. The element $C_1$ is a lozenge, 
and $C_2$ through $C_5$
are $(1,3)$ ideal quadrilaterals.}
\label{figure3}
\end{figure}

\begin{define} (block, adjacent block, odd block)
A block is a finite or infinite collection indexed by
a connected set in ${\mathbb Z}$. The elements of
the collection are either lozenges
or $(1,3)$ ideal quadrilaterals. Consecutive elements
share a corner. If the consecutive elements
have sides whose interiors intersect,
then the respective sides are the same set in $\mt$,
that is the two ideal quadrilaterals are adjacent.
The block is called adjacent if there is a stable or
unstable leaf which intersects all of the elements.
If there is a stable leaf intersecting all elements
then the block is called a stable adjacent block,
and similarly in the unstable case.
The block is called odd if there is a finite, odd number
of elements in it.
In particular a chain of lozenges is a block, where all the elements
are lozenges.
\end{define}




We adopt the convention that the chains that form
blocks do not have backtracking, in other words, they
are minimal given the path of blocks they traverse.

Figure \ref{figure3}  shows a simple example of a finite adjacent block.
Figure \ref{figure4} shows a more involved example of a finite
block, which is made up of the union of a stable
adjacent block and an unstable adjacent block, both
of which have more than one element.

\vskip .05in
\noindent
{\bf {Convention $-$}} Throughout the article
we assume the following convention:
Suppose that $\C$ is a block. When we
talk about the set of stable leaves which 
intersect $\C$ we also include a leaf $E$ of $\wls$ 
so that there are consecutive elements $C_i, C_{i+1}$ in $\C$
so that $E$ contains a side of $C_i$ and a side of $C_{i+1}$ 
and $E$ separates $C_i$ from $C_{i+1}$. 
Notice that the elements $C_i, C_{i+1}$ may be adjacent along $E$ or not.
Equivalently we could say that this set is the smallest
connected set in the leaf space of $\wls$ which contains all
leaves intersecting an element of $\C$.
Similarly for the unstable foliation.

\begin{figure}[ht]
\begin{center}
\includegraphics[scale=1.00]{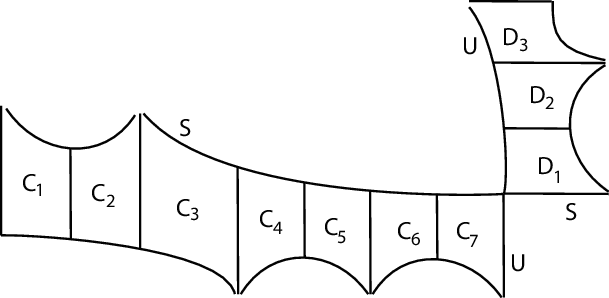}
\begin{picture}(0,0)
\end{picture}
\end{center}
\vspace{0.0cm}
\caption{A more involved finite block.
We assume that $S$ is a stable leaf and $U$ is
an unstable leaf. The block is made up of 10 elements
$C_1, ..., C_7$ and $D_1, D_2, D_3$. The $C's$ form
a stable adjacent block with one lozenge and 6 other
elements which are $(1,3)$
ideal quadrilaterals. In this example in this 
stable adjacent block the
lozenge is neither the first, nor the last
element in this block. The $D's$ form an unstable 
adjacent block, with one lozenge $D_3$ and two
$(1,3)$ ideal quadrilaterals, $D_1, D_2$.}
\label{figure4}
\end{figure}

\subsection{Good form of the flow $\Phi$} \label{ss.goodform}

In \cite{BFP} it is proved that $\Phi$ is orbitally equivalent
to a topological Anosov flow $\Psi$ satisfying the following:

\begin{itemize}
\item{The weak stable foliation of $\Psi$ has smooth leaves.}
\item{There is a leafwise Riemannian metric in the tangent
space of the weak stable foliation $\lsp$ of $\Psi$, such that 
the leaves of $\lsp$ are hyperbolic surfaces with
this metric.}
\item{The flow lines of $\Psi$ are geodesics in the 
respective leaves of $\lsp$ with the hyperbolic metrics given
by the metric.}
\end{itemize}

This is Corollary 5.23 of \cite{BFP}. 
Let $L$ be a leaf of the stable foliation of $\widetilde \Psi$.
The flow lines of $\widetilde \Psi$ in $L$ are geodesics
in $L$ and since this is the stable foliation all these flow lines
have the same ideal point $p$ in $S^1(L)$ in the forward  direction.
In the other direction no two flow lines have the same ideal
point, as they are geodesics and $L$ has a hyperbolic metric.
Finally by continuity and compactness, for any $q$ in $S^1(L) 
\setminus \{ p \}$ there is a single orbit of $\widetilde \Psi$
in $L$ with ideal points $p, q$.
The orbit spaces of $\Phi$ and $\Psi$ are $\pi_1(M)$
equivariantly homeomorphic. This implies that $\Psi$ is
orbitally equivalent to $\Phi$, in particular $\Psi$ is
also a topological
Anosov flow.
In this article we will prove the result for $\Psi$, which 
implies the result for $\Phi$ as well.
So we can start with $\Psi$ instead of $\Phi$.

\vskip .1in
\noindent
{\bf {Conclusion $-$}} 
We will assume that $\ls$ is a foliation by hyperbolic
leaves, and the flow lines of $\Phi$ are geodesics in the
respective leaves of $\ls$.

\vskip .05in
Throughout this article we will work with 
bands and walls
(see later for definitions of these objects)
transverse to the stable foliation $\wls$, starting with a metric
in $M$ so that in each leaf of $\ls$ the metric is a hyperbolic
metric. Conversely we could have considered the unstable
foliation throughout.

\subsection{More information about the proof of the Main theorem}
\label{more}

To prove the Main theorem, we consider bigger and bigger
sets of freely homotopic periodic orbits. 
By Theorem \ref{freelyhomotopic}, this leads to
longer and longer chains of lozenges. If $M$ is atoroidal
we can extract big subchains which are strings of lozenges.
As explained in the introduction, in Subsection \ref{some},
 we take limits of the
annuli associated with the big sets of freely homotopic
periodic orbits. It is very hard to understand the limits
of these by themselves, and we use the associated strings
of lozenges to understand this. 

One basic question here is:
what is the limit of a sequence of lozenges? Clearly 
one possibility is that a sequence of lozenges can 
limit to a lozenge. But it is very easy to see that 
another possibility is that the sequence of 
lozenges can converge to an adjacent chain of lozenges.
We prove in Proposition \ref{limitloz} that there is only
one additional possibility for the limit: an
adjacent block. 
Recall from Subsection \ref{some} that we will 
use walls obtained as   limits
of lifts of annuli associated with big strings of lozenges.
Hence a wall is associated with what
is called a {\em bi-infinite block}. The goal of this
section is to understand limits of lozenges and some properties
of blocks. 

To proceed with the proof of the Main Theorem,
we will need to understand transverse and
tangential intersections of walls. We use the bi-infinite
blocks to understand that. In particular in the case of
$M$ atoroidal this can be used to show that a tangent
intersection between a wall $L$ and a translate $\gamma(L)$
implies that they are exactly the same set.
In the case of a transverse intersection we will prove a 
linking property, which implies both $L$ and $\gamma(L)$ 
intersect exactly the same set of stable leaves in $\mt$.
All of these properties are very hard to understand and
prove using only the walls themselves. The lozenges and
bi-infinite blocks make this study possible.

Then as explained in Subsection \ref{some}: in the case no
transverse intersections between walls, we produce a lamination
in $M$ coming from walls, and eventually show that $\Phi$
is $\rrrr$-covered. In the case of transverse intersections,
we produce a ``convex hull" associated with walls, which
also proves the $\rrrr$-covered property.

\subsection{Bands}

We now describe one of the 
fundamental objects we will use in
this article. 
Eventually we will go from large sets of freely homotopic
periodic orbits to walls in $\mt$ by a limiting process.

\vskip .1in
\noindent
{\bf {Bands associated with lozenges}}

Let $C$ be a lozenge in $\mt$. We build a canonical set which 
we will call a {\em band}
associated with it. Let $\alpha, \beta$ be the corners
orbits of $C$. 
We construct the band (sometimes also called a canonical band)
$B$ associated with $C$ as follows: the boundary of $B$
is the union of the orbits $\alpha, \beta$. For any stable
leaf $E$ intersecting $C$, let $\alpha_E = \wlu(\alpha) \cap E$
and $\beta_E = \wlu(\beta) \cap E$. These are two orbits
in $E$, hence two geodesics in $E$. Let $\ell_E$ be the
geodesic in $E$ with ideal points which are the different
negative ideal points of $\alpha_E, \beta_E$ in $S^1(E)$.
Notice that $\ell_E$ is everywhere transverse to the flow
$\wwp$ restricted to $E$.
The band $B$ is:

$$B \ \ = \ \ \alpha \cup \beta \cup \ \bigcup_{E \in \wls, \  E \cap C 
\not = \emptyset} \{ \ell_E \} $$

For example, when $\pi(\alpha), \pi(\beta)$ are periodic, 
then $B$ is just the lift of an (elementary) Birkhoff annulus
$A$ so that the intersection of $A$ with any stable leaf  of $\Phi$
is the corresponding geodesic in the hyperbolic metric of
its leaf. 

Notice that the band $B$ is a subset of $\mt$.
The band $B$ comes equipped with a one dimensional foliation
which is made up of intersections with leaves of $\wls$, each
leaf of this foliation is a geodesic in the corresponding
leaf of $\wls$. We parametrize the leaves as $\ell_t, 0 \leq t \leq 1$ where (say) $\ell_0 = \alpha, \ \ell_1 = \beta$.
Let $E_t$ be the leaf of $\wls$ containing $\ell_t$.

\vskip .1in
We stress that a band denotes an object
that is very different from a flow band: a flow band is a
bounded interval of flow lines in either a stable or an unstable
leaf. 
By definition a band is a union
of geodesics in different stable leaves, most of which
are transverse to the flow lines and the boundary ones
are tangent to flow lines.

We will later define a wall, which will be a union of bands,
so that consecutive bands share a boundary component.
By {\em infinite strip} we mean the set $[0,1] \times \rrrr$
with the induced topology from inclusion in $\rrrr^2$.

\begin{lemma} \label{band1}
The band $B$ is the image of a continuous embedding
of an infinite strip into $\mt$. In
addition $B$ is a closed subset of $\mt$.
\end{lemma}

\begin{proof}{}
We stress that the lozenge $C$ is not assumed to be periodic,
that is, invariant under a non trivial deck translation. 
The periodic case is much easier.

There is a well defined topology in $T := \cup \{ S^1(E_t),
0 \leq t \leq 1 \}$, see for example \cite{Cal4}, turning 
it into a compact cylinder (this is the initial step in constructing
a universal circle for the foliation $\ls$).

Let $U = \wlu(\beta), V = \wlu(\alpha)$. First it is easy to see 
that the curves $\ell_t$ vary continuously with $t$ when $t$ is in
$(0,1)$. 
For any such $t$ then $\ell_t$ is the 
geodesic in $E_t$, with ideal points the negative ideal points
of $U \cap E_t$ and $V \cap E_t$. 
The negative ideal points of $U \cap E_t$ vary continuously with $t$
in $T$,
as these curves vary continuously. Using ideas associated with
the construction of the universal circle \cite{Cal4} it is easy
to see that the geodesics $\ell_t$ vary continuously.

Notice that $B \cap E_1 = U \cap E_1$.

To get the continuity at say $t = 1$ we consider a point
$p$ in $B \cap E_1$ and a transversal $\tau$ to $\wls$ 
starting at $p$ and intersecting $E_t$ for $t$ close to $1$,
say $t > t_0$.
Then consider the trivialization of 
$R = \cup \{ S^1(E_t), t_0 \leq t \leq 1\}$
using the unit tangent bundle to $\wls$ restricted to $\tau$
($R$ is a subset of $T$ defined above).
The geodesic $U \cap E_t$ has ideal points $x_t, y_t$ in $S^1(E_t)$,
with $x_t$ corresponding to the negative ideal point.
The geodesics $\ell_t$ have one ideal point $x_t$ which clearly
converges to $x_1$ as $t \rightarrow 1$ in the local trivialization.
The other endpoint of $\ell_t$ in $S^1(E_t)$ is the negative
ideal point of $V \cap E_t$ for $0 < t < 1$. 
The geodesics $V \cap E_t$ have one ideal point  which
is the forward ideal point of all flow lines of $E_t$, hence
this is the same as the forward ideal point of $U \cap E_t$,
that is, it is $y_t$. 
Let $z_t$ be the other ideal point
of $V \cap E_t$. Since the flow lines 
$U \cap E_t$ converge to $U \cap E_1$ as $t \to 1$,
it follows that $y_t$ 
converges to $y_1$
in the local trivialization. If the other ideal point $z_t$
of $V \cap E_t$ does not converge to $y_1$ in the
local trivialization of $R$, it follows
that there is a subsequence $z_{t_k}$ converging 
to a point $w \not = y_1$ in $S^1(E_1)$. 
The local trivialization of $R$ implies that
the geodesics in $V$ with ideal points
$y_{t_k}$ and $z_{t_k}$ converge to the geodesic
in $E_1$ with ideal points $y_1$ and $w$.
Since $V$ is properly embedded in $\mt$, the limit geodesic
is also in $V$. In other 
words $V$ intersects $E_1$. But this is impossible,
since $V, E_1$ make a perfect fit.
Hence the ideal points $z_t$ converge to $y_1$.

This shows that $\ell_t$ (with ideal points $x_t, z_t$)
converges to $\ell_1$ (with ideal points $x_1, y_1$)
as $t \rightarrow 1$
in the local trivialization of $R$.

\begin{figure}[ht]
\begin{center}
\includegraphics[scale=1.00]{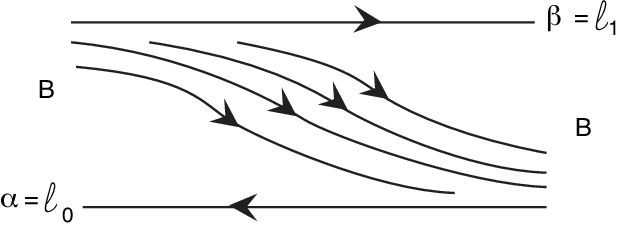}
\begin{picture}(0,0)
\end{picture}
\end{center}
\vspace{0.0cm}
\caption{A band $B$. 
The figure depicts the foliation $\wls \cap B$,
with its geometric behavior: the leaves in the interior
are asymptotic to the boundary leaves in the negative
flow direction.
The arrows in $\ell_0$ and $\ell_1$
indicate the flow direction in these flow lines.}
\label{figure5}
\end{figure}

\vskip .05in
Now we prove that $B$ is a closed subset of $\mt$.
We orient the curves $\ell_t$
in a continuous manner and so that for $t = 1$ the orientation
is the positive flow direction in $\ell_1 = \beta$.
With this choice, for any $0 < t < 1$, it follows
that any ray of $\ell_t$ going against the orientation
in $\ell_t$ is asymptotic to a
backwards ray (with respect to the flow)
of  $\ell_1 = \beta$, 
see Figure \ref{figure5}.
This property will be used throughout the article in 
some varied settings so we explain it carefully here,
following the notation of the first part of the proof.
The geodesics $\ell_t$ have ideal points $z_t, x_t$. The
$x_t$ converge to $x_1$, so the direction of $x_t$ in
$\ell_t$ corresponds to the direction of $x_1$ in $\ell_1$.
In $\ell_1$, the point $x_1$ is the negative ideal point of
the flow line $\ell_1$.
In addition $\ell_t$ is asymptotic (in $E_t$) in the $x_t$ direction
with the flow line $U \cap E_t$. The flow line $U \cap E_t$ is asymptotic
in the negative flow direction with $U \cap E_1 = \ell_1$.
It follows that $\ell_t$ is asymptotic in the $x_t$
direction with $\ell_1$ in the negative flow direction
of $\ell_1$ $-$ as we wanted to prove.
In the same way
any positive ray of $\ell_t$ ($0 < t < 1$) is asymptotic to 
a ray of $\alpha$ which is backwards with respect to the flow
in $\ell_0 = \alpha$. Notice that in $\alpha$ the flow
direction and the positive orientation direction differ from each
other, 
see Figure \ref{figure5}.

We now prove the closed property.
Suppose that $p_i$ is a sequence in $B$ which converges to $p$
in $\mt$. We want to show that $p$ is in $B$.
If there is a subsequence of $p_i$ which is either entirely
contained in $E_0$ or $E_1$ this follows since $\alpha$ and
$\beta$ are properly embedded in $\mt$, being flow lines
of $\wwp$.

Hence up to removing finitely many terms we assume that $p_i$
is never in $E_0$ or $E_1$. Let $t_i \in (0,1)$ with $p_i 
\in E_{t_i}$. If up to a subsequence $t_i$ is in a compact 
subinterval $J$ of $(0,1)$, then the following happens.
Fix a compact transversal $\tau'$ to $\wls$, with 
$\tau'$ contained  in $B$  and
intersecting exactly the leaves $\ell_t$ for $t \in J$.
By compactness of $J$ the following occurs:
for any $k > 0$, 
then in the positive direction of any leaf $\ell_t$, 
after a fixed length $d_k$ from $\tau'$ along any $\ell_t$
($t \in J$), then $\ell_t$ is 
less than $1/k$ from $\ell_0 = \alpha$.
Again using that $\alpha, \beta$ are properly embedded in $\mt$, it
follows that $p_i$ has to have distance from $\tau'$ along $\ell_{t_i}$
which is bounded: this is because if not, then they get closer and closer
to either $\alpha$ or $\beta$ and since $\alpha$ and $\beta$ are
properly embedded, the sequence $p_i$ would escape in $\mt$.
Hence 
sequences in the set $\cup \{ \ell_t, t \in J\}$ 
can only converge to
points in the same set.

Finally suppose that $p_i$ is in $\ell_{t_i}$ where for 
simplicity we assume that $t_i \rightarrow 1$
and so that $p_i$ converges to $p$ in $\mt$.
We also know that there are $q_i \in \ell_{t_i}$ with
$q_i$ converging to $q$ a point in $\ell_1$.
We want to show that $p$ is also in $\ell_1$. 
Suppose this is not the case. 

First think of the case that $p$ is in $E_1$ as well.
Since we are assuming that $p$ is not in $\ell_1$, it follows
that $\ell_{t_i}$ also converges to a geodesic $v$ in $E_1$
through the point $p$, 
besides converging to $\ell_1$. 
This uses that the hyperbolic metrics in different leaves of
$\wls$ vary continuously, so any connected component of
a limit of geodesics is a geodesic.
Fix a round disk $D$ in $E_1$ containing in the interior both
$q$ and $p$.
 Fix a neighborhood of $D$
in $\mt$ where the foliation $\wls$ is topologically a product.
Hence there are approximating convex disks $D_i$ in $E_{t_i}$. 
The points $p_i, q_i$ are in $D_i$ for $i$ big enough, so
the segments of $\ell_{t_i}$ between $p_i$ and $q_i$ are also contained
in $D_i$, as $D_i$ is convex. As a result these segments
have bounded length. Therefore the segments must converge
to a subsegment of $\ell_1$ (because of bounded length),
showing that $p$ is in $\ell_1$.
This finishes this case.

Suppose now that $p$ is not in $E_1$. Hence $p$ is in a leaf
$F$ of $\wls$ non separated from $E_1$. 
Recall that in general the foliation $\ls$ is not
$\rrrr$-covered. However, by construction
$\ell_t$ is always contained in the region of $\mt$ bounded
by $U \cup V$. The closure of this region is disjoint 
from $F$: this is easiest seen in the orbit space
$\oo$: $\Theta(\ell_t)$ is contained in the region
bounded by $\Theta(U), \Theta(V)$. Any $F$ non separated from $E_1$
has $\Theta(F)$ outside of this region. Hence $\ell_t$ cannot limit
in points in $F$. 

This proves that $B$ is properly embedded, proving the lemma.
\end{proof}

\begin{figure}[ht]
\begin{center}
\includegraphics[scale=1.00]{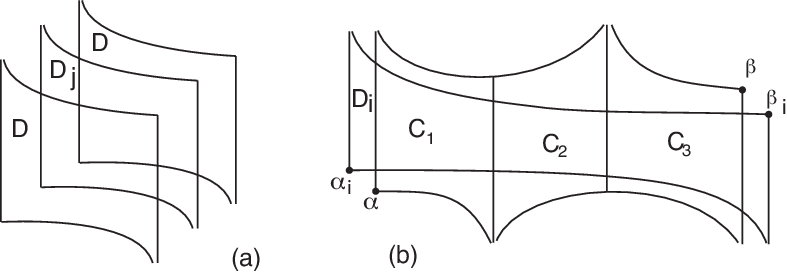}
\begin{picture}(0,0)
\end{picture}
\end{center}
\vspace{0.0cm}
\caption{(a) A sequence of lozenges $D_i$ converging
to a single lozenge $D$, (b) A sequence of lozenges $D_i$ converging
to a union of $3$ lozenges $C_1, C_2, C_3$. In the figure the lozenges
$D_i$ are long horizontally and in the limit split into $3$ lozenges.}
\label{figure6}
\end{figure}

\subsection{Limits of lozenges} \label{ss.limitslozenges}

In Theorem 5.2 of \cite{Fe7} we proved that there is 
a global constant $a_0 > 0$
so that if $\alpha, \beta$ are corners of a lozenge, then
the Hausdorff distance between $\alpha$ and $\beta$ is
finite and bounded above by $a_0$.
See Theorem \ref{lozengebound} here.
It will be crucial for us, when constructing walls in $\mt$,
to understand well limits of sequences of lozenges.
Suppose that $D_i$ is a sequence of lozenges with corners
$\alpha_i, \beta_i$. Suppose that the sequence $\beta_i$ 
converges to an orbit $\beta$. Then bounded Hausdorff distance
implies that $\alpha_i$ all intersect a fixed compact set.
So up to subsequence the $\alpha_i$ converges to an orbit $\alpha$.

The simplest situation is that such a limit is a lozenge
itself. This is certainly possible, for example in the case of
the geodesic flow, and for any $\rrrr$-covered Anosov flow.
This is depicted in Figure \ref{figure6} (a)
where the sequence
$D_i$ converges to the lozenge $D$.
It is very easy to see that another possibility
is that the sequence $D_i$ converges to a union of adjacent
lozenges, this happens for example if there is a free Seifert
piece of the Anosov flow with a ``hyperbolic blow up", as in
\cite{Barb-Fe}. This situation is depicted in Figure \ref{figure6} (b)
where the limit of the sequence $D_i$
is the union of the lozenges $C_1, C_2, C_3$.
In the figure, the corners $\beta_i$ of $D_i$ converge
to $\beta$. The other corners $\alpha_i$ of $D_i$ converge
to $\alpha$. The figure depicts the situation that
$\alpha, \beta$ are connected by a finite
chain of lozenges.

However there is still a more complicated possibility.
This also happens in the examples of \cite{Barb-Fe} mentioned
above, if they are transitive.
This is depicted in Figure \ref{figure7}.
This figure is more complex
and harder to understand, so we explain it in more detail. For simplicity
let us assume that the vertical leaves in the figure
are unstable leaves. The
curved leaves, slightly horizontal, are stable leaves.
In the figure we depict one of the lozenges of the 
sequence, denoted $D_i$, this has corners $\alpha_i, \beta_i$
which converge to orbits $\alpha, \beta$, when $i \to \infty$.
The limits $\alpha,
\beta$ are connected by a collection of 3 objects: the
first is an actual lozenge with corners $\alpha, \alpha'$,
this is $C_1$ in the figure. The second object is $C_2$, which is
a $(1,3)$ ideal quadrilateral with corners $\alpha', \beta',
\beta"$. The third object is $C_3$, which is also a $(1,3)$
ideal quadrilateral, and has corners $\beta, \beta', \beta"$.
There is a stable leaf intersecting all of $C_1, C_2$ and $C_3$.
In addition there are further lozenges associated with
this situation: the pair of adjacent quadrilaterals $C_2, C_3$
is contained in a pair of adjacent lozenges $B_2, B_3$.
For this property see Proposition \ref{periodicdouble}.
The lozenge $B_2$ is the union of $C_2$ and $Z_2$. The set
$Z_2$ is disjoint from the limit of the $D_i$. Similarly
for $C_3$. The leaf $\wlu(\beta)$ is periodic, however
$\beta$ is not the periodic orbit in it. The periodic orbit
in $\wlu(\beta)$ is the orbit $\delta$. The corners
of the periodic lozenge $B_3 = C_3 \cup Z_3$ are $\beta"$ and $\delta$.
In the same way there is a periodic orbit $\mu$ in $\wls(\alpha)$
and $\mu, \alpha$ are distinct orbits.
This will be used throughout this article.
There is a lozenge $B$ which has one side in $\wlu(\alpha')$
and one side in $\wls(\alpha)$. It is visible in Figure \ref{figure7}.
This lozenge intersects the lozenge $C_1$, however 
$B$ neither contains $C_1$ nor is contained in $C_1$.

But there is still a slightly more complicated possibility:
this is similar to what is depicted in Figure \ref{figure7},
but with a different possible positioning of the lozenge
in the adjacent block: specifically it could be that the
lozenge is not an end element of the block, but rather
it is in the middle. In that case there is an even number
of $(1,3)$ quadrilaterals before the lozenge, and
an even number of quadrilaterals after the lozenge.
An example of this is given in Figure \ref{figure4},
as described before,
the stable adjacent block $\cB$ contained in this block
is such that $\cB$ 
has 2 $(1,3)$ quadrilaterals
$C_1, C_2$, then a lozenge $C_3$, then 4 more
$(1,3)$ quadrilaterals $C_4, \cdots, \C_7$.

\begin{convention}
There is non uniqueness involved in the objects obtained
in the limiting process above: the lozenge $C_1$ is 
uniquely determined, and so is the union $C_2 \cup C_3$
(including the common boundary side). 
This union is a $(2,2)$ ideal quadrilateral which is not
a lozenge. It can be split in uncountably many ways
as a union of two $(1,3)$ ideal quadrilaterals. 
In other words, the orbit $\beta"$ is not uniquely
determined by the limiting process. 
We will do the following: using 
$\wlu(\alpha'), \wlu(\beta)$ and the stable leaf $S$ entirely
contained in $\partial (C_2 \cup C_3)$, Theorem \ref{nonsep} shows
that $S$ is periodic, and it has a unique periodic
orbit $\nu$ in it. As a convention we will always choose the
splitting of $(C_2 \cup C_3)$ to that $\beta"$ is the
periodic orbit $\nu$ in $S$.
This convention will always be
used throughout the article whenever we
express a $(2,2)$ ideal quadrilateral which is not a lozenge
as a union of two $(1,3)$ ideal quadrilaterals.
\end{convention}

\begin{figure}[ht]
\begin{center}
\includegraphics[scale=1.20]{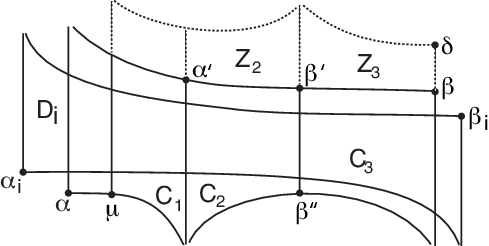}
\begin{picture}(0,0)
\end{picture}
\end{center}
\vspace{0.0cm}
\caption{A sequence of lozenges $D_i$ converging
to a union of $3$ elements $C_1, C_2, C_3$. 
The first is a lozenge, and $C_2, C_3$ are
$(1,3)$ ideal quadrilaterals.}
\label{figure7}
\end{figure}

Our next goal is to show that these are the
only possibilities for the limit of lozenges.
This will be a crucial ingredient in the proof of the
results of this article.

\begin{proposition} \label{limitloz}
Let $(D_k)$ be a sequence of lozenges with
corners $\alpha_k, \beta_k$, which both converge respectively
to $\delta_0, \delta_1$. Then the lozenges $D_k$ converge
to either: \ 1)  a lozenge; or \ 2) 
an adjacent chain of lozenges connecting
$\delta_0, \delta_1$ containing more than
one lozenge; or \ 3) a lozenge and a collection of an even number of 
consecutively adjacent $(1,3)$ ideal quadrilaterals, all intersecting
a common stable or unstable leaf; or \ 4) a lozenge and
two collections of an even number of consecutively adjacent
$(1,3)$ ideal quadrilaterals, again all intersecting a common
stable or unstable leaf. In the last case the lozenge
is between the two families of $(1,3)$ ideal quadrilaterals.
In any case the total number of elements (lozenges and possibly
$(1,3)$ quadrilaterals) is always odd, and there is always
one lozenge in the limit.
\end{proposition}

\begin{remark} \label{samenotation}
We abuse notation here: we include
in the adjacent
chain of lozenges also the half leaves which are common
boundaries of adjacent lozenges in the chain.
With this convention the ajacent chain of lozenges becomes
a connected set, which is a product open set.
The same convention applies to blocks.
In other words the limit of the sequence
$D_k$ is an adjacent block.
We also abuse notation and many times denote the
same way the block $\cB$ and the open set in $\mt$ associated 
with it.
\end{remark}

\begin{proof}
Notice that $\Phi$ is not orbitally equivalent to a suspension,
since there are lozenges.

The case where one of the sequences 
$(\wls(\alpha_k)), (\wlu(\alpha_k)), (\wls(\beta_k)),
(\wlu(\beta_k))$ has a constant subsequence
(which would be equal to $\wls(\delta_0)$ in the first case),
is easier to deal with than the general case. We do the
more complicated case when this does not happen. The
simpler case is left to the reader 
and follows along the same lines as the general proof.
Up to a subsequence we may assume the following
facts: each of the sequences above
is by pairwise distinct leaves, every element of any
of these sequences
intersects a foliated
neighborhood of the corresponding limit orbit; and finally
the sequences are monotonic in the respective leaf spaces.


Let $I^s_k$ be the open interval of stable leaves intersecting
the lozenge $D_k$.
Similarly let $I^u_k$ be the open interval of
unstable leaves.
We first deal with the case that the sequences $(I^s_k)$ or $(I^u_k)$
are not monotonic in $k$.
This happens if an only if either $\delta_0$ is in $D_k$ for all
$k$,
or $\delta_1$ is in $D_k$ for $k$ for all $k$
(these two cases are mutually exclusive).
These cases are also easier than when $(I^s_k)$ is monotone.
Suppose that $\delta_0$ is in $D_k$ for all $k$.
Let $\wls_0(\alpha_k)$ be the half leaf of $\wls(\alpha_k))$ contained
in the closure of $D_k$. This converges to a half leaf of $\wls(\delta_0)$.
Any limit of $\wls_0(\alpha_k)$ is bounded by $\wlu(\delta_1)$ and since $\beta_k$
converges to $\delta_1$, there cannot be any other leaf in the
limit. The same happens for the limit of the half leaves $\wls_0(\alpha_k)$
of $\wlu(\alpha_k)$
contained in the closure of $D_k$. It follows that $\wlu(\delta_1)$
makes a perfect fit with $\wls(\delta_0)$ and similarly
$\wls(\delta_1)$ makes a perfect fit with $\wlu(\delta_0))$ $-$
in other words $\delta_0, \delta_1$ are the corners
of a lozenge $D$, and $D$ is exactly the limit of 
$D_k$. This finishes the proof in this case.

From now on assume that $(I^s_k)$ is monotonic with $k$.
In addition one of the sequences $(I^s_k)$,
$(I^u_k)$ is monotonically increasing and the other
is monotonically decreasing.

It follows that $(I^s_k)$ converges to an
interval, whose interior, denoted by $I^s$ has $\wls(\delta_0),
\wls(\delta_1)$ as boundary points
$-$ and possibly other leaves as well.
Similarly let $I^u$ be the interior of the limit of $(I^u_k)$.
This is obvious for the monotonic increasing sequence. Suppose
that $(I^s_k)$ is monotonic decreasing. A priori it could be that
it converges to a single leaf $S$. But then $\wlu(\alpha_k)
\cap S$ converges to the orbit $\delta_0$ which is in $S$.
This in turn would force $\wlu(\beta_k) \cap S$ to escape to infinity in $S$,
and cannot converge to an orbit $\delta_1$. We conclude that
both $I^s, I^u$ are non empty.

We first define a limit product set. The boundary of this set
has $\delta_0, \delta_1$ as the only corners. The union of this set
and its boundary will be shown to have the structure defined
in the statement of the lemma.
We define the set $R$ to be the set of flow lines $\eta$ of $\wwp$
so that $\wls(\eta) \in I^s$ and $\wlu(\eta) \in I^u$.

We claim that $R$ is the interior of the limit of the lozenges
$D_k$.
Let $\eta$ be a flow line in $R$. Hence for $k$ big enough
$\wls(\eta)$ is in $I^s_k$ and $\wlu(\eta)$ is in $I^u(k)$.
In particular $\eta$ is contained in the lozenge $D_k$.
This also works for a small product neighborhood of $\eta$.
Hence $\eta$ is contained in any possible limit of the
sequence $(D_k)$.
Conversely if $\eta$ is contained in the interior of a 
limit of the $(D_k)$ then $\wls(\eta)$ is contained in $I^s$,
and $\wlu(\eta)$ is contained in $I^u$ so $\eta$ is in
$R$.
This proves the claim.

Finally we show that $R$ is a product open set.
Let $X$ be in $I^s$ and $Y$ in $I^u$. Then 
we explained above that there is a neighborhood $N_0$
of $X$
in $I^s$ which is contained in $I^s_k$ for any $k$
big enough, and similarly for a neighborhood
$N_1$ of $Y$ in $I^u$ for the unstable direction.
Any stable leaf $Z$ in $N_0$ intersects the lozenge $D_k$
and likewise any unstable leaf $V$ in $N_1$ intersects
the lozenge $D_k$. But $D_k$ is a lozenge, so it is
a product open set, and it
follows that $Z$ intersects $V$. 
In addition
the intersection
$Z \cap V$ is in $R$.
It follows that the set $R$ is product foliated: 
each leaf of $\wls_{|R}$ intersects each leaf of 
$\wlu_{|R}$. 
Since $I^s, I^u$ are open intervals then $R$ is open as well,
that is, it is a product open set.

Notice that the orbits $\delta_0, \delta_1$ uniquely determine
the intervals $I^s, I^u$.
Therefore there is only one possibility
for the set $R$. This shows that any subsequence of the original
$(D_k)$ will produce the same region $R$ and proves that 
the sequence of lozenges converges to a region with interior $R$.

\vskip .1in
Suppose without loss of generality that $I^s_k$ is increasing, 
in which case $I^u_k$
is decreasing.
It follows that $\wlu(\beta_k) \cap \partial D_k$ converges to a single leaf:
this is because $I^u_k$ is decreasing and there is only
one limit of the sequence $\wlu(\beta_k) \cap \partial D_k$ which is in the boundary of $I^u$.
This limit is a half leaf of $\wlu(\delta_1)$ defined
by $\delta_1$.

Suppose that the limit of $\wls(\alpha_k) \cap \partial D_k$ is not contained
in $\wls(\delta_0)$. We claim that the limit
is a finite collection containing a half leaf of 
$\wls(\delta_0)$ and full leaves $X_1, ..., X_j$.
In addition $\wlu(\delta_1)$ makes a perfect fit
with the last of these, because there are no product 
regions. This argument using that there are
no product regions will be used many
times, so the first time we give more details.

Let $\eta$
be an orbit in $R$ so that $\wlu(\eta)$ intersects
$\wls(\delta_0)$. We already proved that $\wls(\eta)$ intersects 
$\wlu(\delta_1)$.
First let $Z$ be the first unstable leaf between $\wlu(\eta)$
and $\wlu(\delta_1)$ not intersecting $\wls(\delta_0)$.
The hypothesis that the limit of $\wls(\alpha_k) \cap \partial D_k$
is not contained in $\wls(\delta_0)$, implies that $Z$
is not $\wlu(\delta_1)$.
We define an unstable leaf $U$ as follows:
$U$ intersects $\wls(\eta)$, $U$ is between $\wlu(\eta)$ and
$\wlu(\delta_1)$ $-$ including possibly $\wlu(\delta_1)$ which
is what we actually want to show happens.
In addition either $U = \wlu(\delta_1)$ or for any
$V$ unstable intersecting $\wls(\eta)$ between $U$ and $\wlu(\delta_1)$
then $V$ does not intersect any leaf non separated from $\wls(\delta_0)$.
Again the hypothesis that the limit of $\wls(\alpha_k) \cap 
\partial D_k$ is not contained in $\wls(\delta_0)$ implies
that $Z, U$ are distinct leaves intersecting $\wls(\eta)$.
If $U$ is not $\wlu(\delta_1)$ then consider the subregion
$R'$ of $R$ cuttoff by $\wls(\eta), U$ and accumulating on
$\wlu(\delta_1)$, see Figure \ref{figure8}.
This is a product region, and by construction it is a stable
product region. This is impossible. It follows
that $U = \wlu(\delta_1)$. 
In addition there are only finitely many leaves
in the limit of $\wls(\alpha_k) \cap D_k$: otherwise
these limit leaves would be part of the boundary
of a scalloped region. This again is a contradiction
because in that case a half leaf of $\wlu(\delta_1) (= U)$ could
not intersect the same leaves as a half leaf of $Z$
defined by $Z \cap \wls(\eta)$.

We conclude that the limit of $\wls(\alpha_k) \cap \partial D_k$
is the union of a half leaf of $\wls(\delta_0)$ and
$X_1 \cup \cdots X_j$. In addition $\wlu(\delta_1)$
makes a perfect fit with $X_j$.



\vskip .1in
Since $\wls(\delta_0), X_1,... X_j$ are pairwise non
separated, there is a periodic orbit
$\eta_0$ in $\wls(\delta_0)$ and there are
lozenges $G_0, G_1,... G_{2j}$ all intersecting
a common stable leaf, so that $G_0$ has a side in $\wls(\delta_0)$,
and $G_{2i-1}, G_{2i}$ have a side in $X_i$,
and the union of these two sides is $X_i$ minus the periodic
orbit in $X_i$.

So far we have obtained parts of the boundary of $R$: 
a half leaf of $\wls(\delta_0)$, the leaves
$X_1,...,X_j$ and a half leaf of $\wlu(\delta_1)$.
We now continue analyzing the boundary of $R$ 
starting from $\delta_1$. 
It all depends on whether $\wls(\delta_1)$ makes
a perfect fit with $\wlu(\delta_0)$ or not.

Suppose first that $\wls(\delta_1)$ makes a perfect fit
with $\wlu(\delta_0)$. In this case we show that $R$ is
a union of a lozenge with corner $\delta_0$ and 
$2j$ blocks which are contained in $G_i$, $1 \leq i \leq j$.
This is fairly easy: $\wls(\delta_1)$, as it makes
a perfect fit with $\wlu(\delta_0)$, must intersect
all lozenges $G_i, 1 \leq i \leq 2j$, cutting them
into $(1,3)$ ideal quadrilaterals contained in $R$. 
Let now $A$ be the leaf of $\wlu$ which makes a perfect fit
with both $\wls(\delta_0)$ and $X_1$. Then $A$ intersets
$\wls(\delta_1)$ in an orbit denoted by $\delta_3$. 
Then $\delta_0, \delta_3$ are the corners of
a lozenge and $R$ is the union of this lozenge
and the $2j$ $(1,3)$ ideal quadrilaterals described before.
This finishes the proof in this case.


\begin{figure}[ht]
\begin{center}
\includegraphics[scale=1.00]{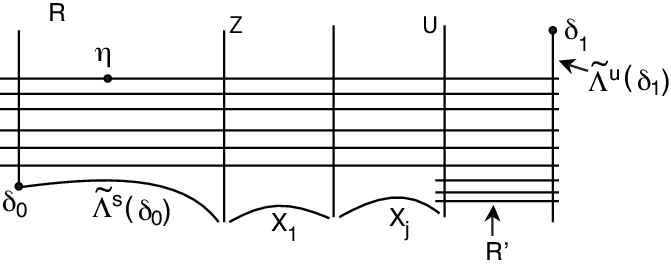}
\begin{picture}(0,0)
\end{picture}
\end{center}
\vspace{0.0cm}
\caption{This situation results in a disallowed
product region $R'$. The figure displays part of the 
product open set $R$ and its boundary, which contains
half leaves of all of $\wlu(\delta_0), \wls(\delta_0),
\wls(\delta_1), \wlu(\delta_1)$ as well as at least the
stable leaves $X_1, \cdots, X_j$. For simplicity here
$j = 2$.} 
\label{figure8}
\end{figure}

\vskip .1in
The next case to analyze is when $\wls(\delta_1)$ does
not make a perfect fit with $\wlu(\delta_0)$.
Recall that $\wlu(\delta_1)$ makes a perfect
fit with $X_j$, and $X_j$ is a periodic leaf,
hence $\wlu(\delta_1)$ is also periodic and
there is a periodic orbit $\eta_1$ in $\wlu(\delta_1)$.
Let $h$ be a non trivial deck translation which
fixes all the $G_i$, hence $h$ also fixes $\eta_1$.
In addition $\wls(\delta_0)$ is also fixed by $h$
and there is a periodic orbit $\eta_0$ in $\wls(\delta_0)$
fixed by $h$.

The limit of $\wls_0(\beta_k)$ contains a half
leaf of $\wls(\delta_1)$ and other leaves
which are non separated from this one. The last
such leaf which is in the limit of $(D_k)$, call it $B$,  makes a perfect fit with
$\wlu(\delta_0)$.

\vskip .05in
Suppose first in this case that $\delta_1 = \eta_1$.
Then $h$ fixes $\delta_1$ and hence also fixes $B$,
and consequently also fixes
$\wlu(\delta_0)$. As $h$ also fixes $\wls(\delta_0)$,
it now follows that $\delta_0$ is periodic
and is equal to $\eta_0$. 
In this case the orbits $\delta_0, \delta_1$ are 
connected by a chain of adjacent lozenges all intersecting
a common stable leaf.
In this case $\delta_0, \delta_1$ and all lozenges
are fixed by $h$.

\vskip .1in
Suppose finally that $\eta_1$ is not equal to $\delta_1$.
Suppose first that $\wls(\eta_1)$ separates $\wls(\delta_1)$
from $\delta_0$. Then, since $R$ is a product open set, 
and $\wls(\eta_1)$ intersects $R$, it follows that $\wls(\eta_1)$
intersects $\wlu(\delta_0)$. 
This is impossible since $\wls(\eta_1)$ has a half leaf in
the boundary of $G_{2j}$ and so $\wls(\eta_1)$ cannot intersect
$\wlu(\delta_0)$. This case cannot happen.

The remaining possibility is that $\wls(\delta_1)$ separates
$\wls(\eta_1)$ from $\delta_0$.
In fact this case is actually possible.
This is similar to the situation that the limit of $\wls(\alpha_k)$
is not contained in a single leaf.
In this case we have that there are lozenges
$H_0, H_1,... H_{2k}$ all intersecting a common stable
leaf and so that consecutive ones share an unstable
side, and $H_{2k}$ has a unstable side contained in $\wlu(\delta_0)$.
The orbit $\delta_0$ is not the periodic orbit in $H_{2k}$.
The lozenges $G_i$ and $H_m$ are disjoint from each other
and even the closures of $G_0$ and $H_0$ are disjoint from
each other.

The leaf $\wls(\delta_1)$ 
cuts the lozenges $G_i, 0 \leq i \leq 2j$, and one
complementry component is a $(1,3)$ quadrilateral
contained in $R$. 
In the same way the leaf $\wls(\delta_0)$ cuts
the lozenges $H_i, 0 \leq i \leq 2k$,
and one complementry component is a $(1,3)$
quadrilateral contained in $R$.  
Finally between $H_0$ and $G_0$ there is
a $(0,4)$ quadrilateral $N$  (a rectangle) contained
in $R$: it has one side in $\wls(\delta_0)$,
one side in $\wls(\delta_1)$, one unstable side
in the boundary of $H_0$, and one unstable side
in the boundary of $G_0$. The union 
$(H_0 \cap R)  \cup N \cup (G_0 \cap R)$ (and
the flow bands in common boundaries) is a lozenge, with a corner
in $\wls(\delta_0)$ which is also in the 
common boundary of $H_0$ and $H_1$ and a corner
in $\wls(\delta_1)$ which is also in the common
boundary of $G_0$ and $G_1$. In this case
again the limit region $R$ is the
union of a lozenge and an even number of $(1,3)$
ideal quadrilaterals: $2j$ coming from the $G_i's$
and $2k$ coming from the $H_i's$. Unlike the
case that $\wls(\delta_1)$ makes a perfect fit
with $\wlu(\delta_0)$, in this case the lozenge
is in the middle of the chain from $\delta_0$ to $\delta_1$
and not in the beginning of it.

\vskip .05in
This finishes the analysis that there is
an $X$ non separated 
from $\wls(\delta_0)$ contained in the boundary of $R$.
The same happens if 
there is $X$ non separated
from $\wls(\delta_1)$ contained in the boundary of $R$. Notice
that in either case there is no $Y$ unstable 
entirely contained in $\partial R$,
so that $Y$ is non separated
from either $\wlu(\delta_0)$ or $\wlu(\delta_1)$.
Switching the roles of stable and unstable this also deals with
the case that there is $Y$ unstable non separated from $\wlu(\delta_0)$
or $\wlu(\delta_1)$ contained in $\partial R$. In this case
there is no other stable leaf in the boundary of $R$. 
This is equivalent to $I^u_k$ increasing. In 
addition $\delta_0, \delta_1$ are connected by either a chain
of lozenges or a  lozenge and an even number of $(1,3)$
ideal quadrilaterals, all intersecting a common unstable leaf.

Finally we are left with the case that there is no other 
stable, or unstable  leaf in $\partial R$.
In this case the same analysis as above shows that
$\wls(\delta_0)$ makes a perfect fit with $\wlu(\delta_1)$
and $\wlu(\delta_0)$ makes a perfect fit with $\wls(\delta_1)$.
In other words $\delta_0, \delta_1$ are the corners of a lozenge.

This finishes the proof of Proposition \ref{limitloz}.
\end{proof}

\begin{define} (basic block) A basic block is a limit
of a sequence of lozenges. It can be either a single lozenge,
a stable adjacent block with an odd number $> 1$ of elements,
or an unstable adjacent block with an odd number $> 1$ of elements.
\end{define}

In the cases of adjacent blocks it could be that the elements
are all lozenges, or one element is a lozenge and the others
are pairs of adjacent $(1,3)$ ideal quadrilaterals.

\subsection{Bands and finite adjacent blocks}

In this subsection we construct bands associated
with a basic block.
The bands will be canonical, that is, entirely determined
by the block, in fact determined by the corner orbits of
the block.
Let $\C$ be a basic block with elements
$C_i, 0 \leq i \leq 2j$.

\vskip .1in
\noindent
{\bf{Bands for finite adjacent blocks intersecting a common stable leaf}}

The constructions of bands will be similar to the construction
of the band associated to a single lozenge.
First we consider the case that all elements in $\C$ intersect
a common stable leaf. Then every $C_i$ intersects exactly
the same set of stable leaves, which is an open interval $I_0$.
Let $\alpha$ be the initial corner in $C_0$ and $\beta$ be
the final corner in $C_{2j}$. 
We adjoin the leaves $\wls(\alpha), \wls(\beta)$ to
the open interval $I_0$ to produce
a closed interval $I$ in the leaf space $\oo^s$ of $\wls$.

For $E = \wls(\alpha)$ we define $\eta_E := \alpha$,
and similarly for $E = \wls(\beta)$ we define $\eta_E := \beta$.
For a leaf $E$ in $I_0$ it follows that $E$ intersects
both $\wlu(\alpha)$ and $\wlu(\beta)$, each intersection
is a flow line which is a geodesic in $E$. For such $E$ we define 
$\eta_E$ to be the geodesic with ideal points
the negative ideal points of the flow lines 
$\wlu(\alpha) \cap E$ and $\wlu(\beta) \cap E$.
The band associated with $\cC$ is defined as

$$B \ \ := \ \ \bigcup_{E \in I} \{ \eta_E \}$$

\begin{remark} In the case of a stable adjacent block, the
band is transverse to the flow lines except
at the boundary leaves.
Notice that this band is canonical, that is, unique given the
block, in fact unique given $\alpha, \beta$. 
\end{remark}

Just as in the previous subsection one can prove the following:

\begin{lemma}\label{band2}
The band $B$ associated to  block that intersects
a common stable leaf is the image of a continuous embedding of
an bi-infinite strip into $\mt$. The band $B$ is a closed subset
of $\mt$.
\end{lemma}

\begin{proof}
The proof is very similar to that of Lemma \ref{band1}.
Let $\fol$ be the foliation in $B$ induced by $\wls \cap B$.
Just as in Lemma \ref{band1}, any pair of interior leaves are asymptotic
to each other in both directions, 
so they are a bounded distance from each other.

Parametrize $I$ as $E_t, 0 \leq t \leq 1$, with $\wls(\beta) = E_1$.
Let $\ell_t = E_t \cap B$. Just as in Lemma \ref{band1} one proves that
$\ell_t$ converges to $\ell_1$ when $t \rightarrow 1$ and
similarly for $t \rightarrow 0$.
This shows the continuity.

Now we prove the closed property for $B$. 
Most of the proof of Lemma \ref{band1} carries over to our situation.
It suffices to consider the situation that
$p_i \in \ell_{t_i}$ converges to $z$, and $t_i \rightarrow 1,
t_i < 1$, the other cases are dealt
with exactly as in Lemma \ref{band1}.
Hence $z$ is in a leaf $F$ of $\wls$, with $F$ non separated
from $E_1$. The case $F = E_1$ is again dealt with as in
Lemma \ref{band1}. So suppose that $F$ is not $E_1$.
Now this is different from Lemma \ref{band1}, because the
block $B$ may have more than one element and hence there may be
other stable leaves in the boundary of the open interval $I_0$
 which are not
separated from $\wls(\beta)$.
We analyze this more carefully.

Recall that the corners of the block $\C$ are $\alpha, \beta$.

Since $p_i$ converges to $z$ in $F$, then up to subsequence 
we assume that $\ell_{t_i}$ also converges to a geodesic
$\zeta$ in $F$. At least one of the ideal points
of $\zeta$ is a negative ideal point in $F$. Call it $y$.
We now consider a closed interval of leaves of $\wls$ intersecting
a transversal $\tau$ with an endpoint $z$ and intersecting 
$E_{t_i}$ for $i$ big enough. Consider the local trivialization
of the annulus $A = \cup_{x \in \tau} S^1(\wls(x))$. 
Since $\ell_{t_i}$ converges to $\zeta$ one of the ideal points
of $\ell_{t_i}$ converges to $y$ in the local trivialization of $A$.
It follows that there is a closed
interval $J$ of \underline{unstable} leaves all intersecting
$F$ with $\wlu(z)$ in the interior of $J$ satisfying
the following: the ideal
points of $\ell_{t_i}$ in $E_{t_i}$ converging to $y$ in $A$
are negative ideal points of flow lines
$U_i \cap E_{t_i}$, where the  $U_i$ are appropriate unstable leaves
contained
in $J$.
Notice that since $U_i$ intersects $F$ it cannot be either 
$\wlu(\alpha_i)$ or $\wlu(\beta_i)$.
But by construction the ideal points of $\ell_{t_i}$ 
are the negative ideal points of $\wlu(\beta_i) \cap E_{t_i}$
and $\wlu(\alpha_i) \cap E_{t_i}$. Since $U_i$ is neither $\wlu(\alpha_i)$
nor $\wlu(\beta_i)$ this is a contradiction.
This finishes the proof of the lemma.
\end{proof}

\vskip .1in
\noindent
{\bf {Bands for basic blocks 
intersecting a common unstable leaf}}

In this case the band will split into subbands
if $j > 0$.

Suppose first that there are only lozenges in the basic block $\C$,
that is $C_i$ is a lozenge for any
$0 \leq i \leq 2j$.
For each lozenge $C_i$ construct the band
associated to it. Recall that the band was uniquely
determined by the corners.  The union is the band $B$
associated with the block.
Each corner of a lozenge $C_i$ is 
an orbit which is tangent to the band
constructed. The band $B$ 
can be decomposed into finitely many
subbands, and outside of the boundaries of these 
subbands, the subbands are transverse to the flow lines
in the respective leaves of $\wls$.
For each individual lozenge, the band satisfies the continuity
and closed properties of Lemma \ref{band2}, and so does the union.

\vskip .05in
Now suppose that there are $(1,3)$ ideal quadrilaterals
in $\C$.
To be precise an example is when the first element $C_0$ is a lozenge
and the other elements $C_i, 1 \leq i \leq 2j$ are $(1,3)$
ideal quadrilaterals.
For the lozenge $C_0$ we consider the 
band $B_0$ associated to it.
This satisfies the properties of continuity and being closed
of Lemma \ref{band2}.
For the other elements, consider them in pairs $C_{2i-1}, C_{2i}$.
For simplicity of notation consider $C_1, C_2$ with
corner orbits $\zeta$ in $\partial C_1$ and $\mu$ in 
$\partial C_2$, hence
$\wlu(\zeta) = \wlu(\mu)$.
The associated subband will be denoted by $B$.
Let $I$ be the closed interval of stable leaves from
$\wls(\zeta)$ to $\wls(\mu)$, parametrized as $E_t, 0 \leq t
\leq 1$ with $E_1 = \wls(\mu)$.
Let $V$ be the unstable leaf entirely contained
in $\partial(C_1 \cup C_2)$.
Let $\ell_0 = \zeta, \ell_1 = \mu$.
For $0 < t < 1$ let $\ell_t$ be the geodesic in $E_t$ with
ideal points the negative ideal points of the
flow lines $V \cap E_t$ and
$\wlu(\zeta) \cap E_t$.
Now let 

$$B_i = \bigcup_{0 \leq t \leq 1} \ell_t$$

\noindent
be the band associated with $C_{2i-1} \cup C_{2i}$.
Let $B$ be the union of $B_i, 0 \leq i \leq 2j$.

Notice that in this case $\zeta, \mu$ are asymptotic
in the negative direction and the directions of the
flow lines $\zeta, \mu$
agree.
Just as in Lemma \ref{band1}, the $\ell_t$ vary 
continuously. For $0 < t < 1$ different $\ell_t$ are asymptotic
to each other in both directions. This is because 
the ideal points are either the negative ideal points
of $V \cap E_t$ or the negative ideal point of $\wlu(\zeta) \cap E_t$.
Hence for different $t$ the corresponding rays of $\ell_t$ are
asymptotic to each other.
The proof that $B$ is closed is easier in this case since
there are no other stable leaves in the limit of $E_t, \ t \rightarrow
1$ which are between $V$ and $\wlu(\mu)$, which is the situation
dealt with in
Lemma \ref{band1}.
A similar description holds when the lozenge in the block
is neither the initial nor the final element of the block.

\begin{remark} In the case of an unstable adjacent block, the
band associated to the block can have two possibilities:
Let $C_i, 0 \leq i \leq 2j$ be the elements of the block.

- If there are only lozenges in the block then the 
band has $2j + 1$ subbands, one for each lozenge $C_i$.
In this case the lozenges are periodic, so the subbands
are lifts of elementary Birkhoff annuli. One can push
the tangencies of the union of the Birkhoff annuli
with the intermediate corners, producing a single Birkhoff
annulus which is not elementary. This is well known and
used in the theory of pseudo-Anosov flows. However
notice two facts: \ i) the induced ``pushed" band
will not intersect all leaves of $\wls$ in geodesics,
only the outermost ones, \ ii) the pushed band is not
unique given the chain of lozenges. Both of these
properties are crucial for us in the ensuing analysis.

- If there is one lozenge and $2j$ elements which are $(1,3)$
ideal quadrilaterals, then the band has $j + 1$
subbands, one for the lozenge, and one for each pair of
adjacent $(1,3)$ quadrilaterals.
\end{remark}

\vskip .1in
\noindent
{\bf {The bounded distance property.}}

Given a finite adjacent block $\C$ with band $B$, we let $\fol$
be the foliation in $B$ induced by intersecting $\wls$ with $B$.
By construction this foliation has leaf space which is 
a closed interval.
This is a one dimensional foliation in $B$.

\begin{proposition} \label{bounded}
Let $B$ be the band associated with a
finite adjacent 
block $\C$. Then any two leaves of $\fol$ are a finite and bounded
Hausdorff distance from each other as subsets of $\mt$. 
The bound depends on $B$.
\end{proposition}

\begin{proof}{}
Obviously the Hausdorff distance 
in question depends at least on how many elements
there are in $B$.

The band $B$ is decomposed into finitely many subbands,
so that each subband has boundary components which are flow lines
and in the interior the subband is transverse to the flow.
We do the proof for such a subband, which for this proof
we assume is
the original band.
This is enough to prove the result.
We explain more: if the block $\cC$  is a stable adjacent block then
there is a single band and no subdivision.
If $\cC$ is an unstable adjacent block, then either 
\ i) $\cC$ only has lozenges, in which case the
result subdivision is just a single lozenge and
the band associated to this lozenge, or \ ii) $\cC$
is a union of one lozenge and finitely many pairs
of adjacent $(1,3)$ quadrilaterals. Then the subband is
either associated to a lozenge, as before, or to a pair
of adjacent $(1,3)$ quadrilaterals. For the purposes here, the union of this pair
is seen as a $(2,2)$ quadrilateral.

For simplicity of notation, in this proof we also denote
the resulting subblock as $\cC$.
Associated with $\C$ there is a product open set $R$ with
only two corners $\alpha, \beta$. 
Notice that with the subdivision of the original band, the
resulting $\C$ may not be a basic block.
This occurs only in the final possibility
described above, when the resulting $\C$ is a union of two adjacent
$(1,3)$ ideal quadrilaterals.
After the decomposition, there are unique unstable leaves
$U, V$ with a half leaf in $\partial R$, except in the last case.
In the last case one of $U, V$ is entirely contained in $\partial R$,
and the other has a bounded flow band contained in $\partial R$.
We first prove the following:

\begin{claim} \label{claim-firstbound}
$\alpha, \beta$ are a finite Hausdorff distance
from each other. 
\end{claim}

\begin{proof}
This is true if $\C$ is a lozenge, by
Theorem 5.2 of \cite{Fe7}. 
The other case is that $\C$ is the union of two
adjacent $(1,3)$ quadrilaterals.

A subtle point here: if the original block is the limit of
a sequence of lozenges, then the initial and final corners
are a bounded distance from each other, again by the theorem
above. However, it does not immediately follow that the intermediate
corners are a bounded distance or even finite distance
from each other.

Suppose next that
$\C$ is the union of two adjacent $(1,3)$ quadrilaterals
intersecting a common unstable leaf, then $\alpha, \beta$ are
in the same unstable leaf, which we  may assume is $U$.
Then there are two adjacent lozenges $C_1, C_2$
and the $(1,3)$ quadrilaterals
are obtained by cutting $C_1, C_2$ along $U$. 
Up to switching $C_1, C_2$,
suppose $\alpha$ is in a stable side of $C_1$ and $\beta$
in a stable side of $C_2$. 
Then $V$ is the unstable leaf entirely contained in
$\partial(C_1 \cup C_2)$.
Let $\gamma$ be the periodic orbit in $V$.
Backwards rays of $\alpha, \beta$ are asymptotic, hence
a finite Hausdorff distance from each other. 
The leaf $V$ makes a perfect fit with $\wls(\alpha)$.
By Theorem 5.7 of \cite{Fe7} 
a forward ray in $\wls(\alpha)$
is a finite Hausdorff distance from a backwards ray in $V$.
Hence a forward ray of $\alpha$ is a finite distance
from a backwards ray of $\gamma$. The same holds for
a forward ray of $\beta$ and a backwards ray of $\gamma$.
It follows that forward rays in $\alpha$ and $\beta$ are a finite
Hausdorff distance from each other. 
This implies that $\alpha, \beta$ are a finite Hausdorff
distance from each other.

The last situation to analyze is the case that $\C$ is a stable
adjacent basic block. 
If it is a union of lozenges, the claim follows from the
lozenge case, so we assume $\C$ has a lozenge and pairs
of adjacent $(1,3)$ ideal quadrilaterals.
We first deal with the case that the lozenge is at the
end of the block.
We refer to a previous
figure, Figure \ref{figure7}. The block $\C$
contains one lozenge and pairs
of $(1,3)$ ideal quadrilaterals, and $\C$ has corners
$\alpha, \beta$ as in Figure \ref{figure7}
(so $\{ \alpha \cup \beta \} = \{ \delta_0 \cup \delta_1 \}$).
Up to switching the
roles of $\delta_1, \delta_0$ we may assume that 
$\wls(\delta_1)$ makes a perfect fit with $\wlu(\delta_0)$
as in the figure.
Again by Theorem 5.7 of \cite{Fe7},  a forward ray of $\delta_1$ is a finite
Hausdorff distance from a backward ray of $\delta_0$.
In addition $\wls(\delta_0)$ is non separated from stable leaves
$X_i, 1 \leq i \leq j$, all in $\partial R$,
with $X_j$ making a perfect fit with
$\wlu(\delta_1)$. 
Again by Theorem 5.7 of \cite{Fe7} a backward ray
of $\delta_1$ is a finite Hausdorff distance from a forward
ray of a flow line in $X_j$. The $X_i$ have periodic
orbits which project to freely homotopic orbits of $\Phi$.
Hence forward rays in distinct $X_i$ are a finite Hausdorff
distance from each other (this could also be obtained by
the harder result of Theorem 5.7 of \cite{Fe7}). 
But $\wls(\delta_0)$ also has a periodic orbit which projects
to $M$ to a periodic orbit freely homotopic to one of these
orbits. A forward ray of $\delta_0$ is also a finite
Hausdorff distance from these, so we conclude that 
a forward ray of $\delta_0$ is a finite Hausdorff distance from
a backward ray of $\delta_1$.
We conclude that $\delta_0, \delta_1$ are a finite Hausdorff 
distance from each other in $\mt$.

The last case is when the lozenge is in the middle of the 
block.
Then we have a finite number of pairs of adjacent $(1,3)$
quadrilaterals, one lozenge, and then more pairs of
adjacent $(1,3)$ quadrilaterals. For each pair of 
adjacent $(1,3)$ the union is a $(2,2)$ ideal quadrilateral.
The previous case shows that the two corners are a finite
Hausdorff distance from each other. When we hit the lozenge,
the corners are a bounded Hausdorff distance from each other.
Then more pairs of adjacent $(1,3)$ quadrilaterals, which
again have corners a finite Hausdorff distance from
each other.

This finishes the proof of the claim.
\end{proof}

Next we analyze the interior leaves of $\F$, and 
prove that the whole band $B$ is a finite
Hausdorff distance from (say) $\beta$.
Hence suppose that there are $p_i$ in $B$ so that $d(p_i,\beta)$ converges
to infinity. In particular $p_i$ are eventually not in $\alpha \cup \beta$
by the claim above.
Up to a subsequence and deck translations $g_i$
we assume that $g_i(p_i)$ converges to $p$.
Let $R$ be the product open set defined by the block $\C$
with corners $\alpha, \beta$. 
Let $I^s$ ($I^u$) be the open set of stable (unstable) leaves
intersecting $\C$. 
Then $g_i(R)$ is a product open
set with corners $g_i(\alpha), g_i(\beta)$.

We prove that up to a subsequence $(g_i(R))$ converges 
to a product open set.
One concern is that $(g_i(R))$ may converge to a single stable
or single unstable leaf. 
We first analyze this.

\begin{claim} \label{claim-intervalnotshrinking}
$g_i(I^s), g_i(I^u)$ cannot converge
up to subsequence to a single leaf.
\end{claim}

\begin{proof}
We first analyze the images $g_i(I^s)$. Suppose that
up to a subsequence $g_i(I^s)$ converges to a single
leaf $F$. Up to subsequence and up to changing $p_i$ slightly in its
stable leaf if necessary, there are distinct
leaves
$E, L$ of $\wls$ which have at least half leaves contained
in the boundary of $R$ and so that $\wlu(p_i)$ intersects
both $E$ and $L$. The slight change is that
a priori $\wlu(p_i)$ could intersect just one stable
leaf or half leaf in $\partial R$, and make a perfect
fit with some stable leaves in $\partial R$. 
In addition there only finitely many stable leaves which
have at least a half leaf contained in $\partial R$, so $E, L$
can be chosen independent of $i$ up to subsequence.

If $g_i(I^s)$ converges to a single
leaf $F$, then so do the sequences $g_i(E), g_i(L)$, 
this is because $g_i(p_i)$ converges in $\mt$.
This implies that $g_i(\wlu(\alpha)), g_i(\wlu(\beta))$ 
escape compact sets in $\mt$.
In particular the distance in $\mt$ from $g_i(p_i)$ to 
the union
$g_i(\wlu(\alpha)) \cup g_i(\wlu(\beta))$
converges to infinity. This is a property that we will show is
impossible.
The property implies that distance from $p_i$ to
$\wlu(\alpha)  \cup \wlu(\beta)$ converges to infinity.
Let $\ell_i$ denote the intersection $B \cap \wls(p_i)$.
The geodesics

$$\wlu(\alpha)  \cap \wls(p_i), \ \ \ \wlu(\beta) \cap \wls(p_i), 
\ \ \ \ell_{i}, \ \ {\rm in} \ \ 
\wls(p_i),$$

\noindent
are the sides of an ideal triangle $T_i$ in $\wls(p_i)$. Notice that
all ideal triangles in the hyperbolic plane are isometric.
The point $p_i$ is on one side of this ideal triangle
in $\wls(p_i)$.
A basic property of the hyperbolic plane implies that any
point in any side of an ideal triangle is a bounded distance
from the union of the two other sides, hence the distance
from $p_i$ to \
$(\wlu(\alpha)  \cap \wls(p_i))   \cup  (\wlu(\beta)  \cap \wls(p_i)$
in  $\wls(p_i)$
is bounded above. It follows that the distance between these
sets in $\mt$ is also bounded above.
This contradicts the property above.
In particular this implies
that $g_i(I^s)$ cannot converge to a single stable leaf
even up to subsequence. 

We cannot apply the same argument to $g_i(I^u)$ because the
construction is not symmetric. 
For example the leaves of $\wlu$
are not necessarily smooth. 
Suppose that $g_i(I^u)$ converges to a single leaf $W$.
In this case it follows that
$d(g_i(p_i),g_i(\wlu(\alpha)))$ and $d(g_i(p_i),g_i(\wlu(\beta)))$ both converge
to zero. This is distance in $\mt$.
In particular $d(p_i, \wlu(\alpha)), \ d(p_i,\wlu(\beta))$,
also converge to $0$, again distance is in $\mt$.
Now recall that there is a uniform positive size
so that  each point in $\mt$ has a product neighborhood of $\wls$ 
of at least this size, and in addition a leaf of $\wls$
intersects each such neighborhood in a single plaque of the
foliation.
In particular $d(p_i,\wlu(\alpha)) \to 0$ implies that $\wls(p_i)$ intersects
$\wlu(\alpha)$,  and in addition distance in $\wls(p_i)$ between
$p_i$ and $\wls(p_i) \cap \wlu(\alpha)$ converges to zero
as $i \to \infty$.
The same holds for $\beta$.
Now we analyze
the situation in $\wls(p_i)$. As before let $\ell_i = B \cap \wls(p_i)$.
As in the case for $I^s$, \ 
$\wlu(\alpha) \cap \wls(p_i), \wlu(\beta)  \cap \wls(p_i), \ell_i$ are the 
sides of an ideal geodesic triangle in $\wls(p_i)$.
The points $p_i$ are in one side of this triangle.
We proved above that the distance in $\wls(p_i)$  from $p_i$ to {\underline {each}} of the 
other two
sides of this geodesic triangle converges to zero (when $i \to \infty$).
But this is impossible as seen from the $\wls$ leaf:
there is a positive lower bound
to the maximum distance from a point in a side of an ideal
triangle in ${\bf H}^2$ to the other two sides.
We conclude that $g_i(I^u)$ also cannot converge to a single 
(unstable) leaf. 
This proves the claim.
\end{proof}

%

By the claim
it follows that 
up to removing finitely
many elements we may assume that the $g_i(R)$ pairwise intersect
each other.
We next show that the 
intervals $g_i(I^s), g_j(I^s)$ do not branch from each
other. This means the following:
there are no $F \in g_i(I^s)$, $L \in g_j(I^s)$ which
are distinct but non separated from each other. 
Suppose this is not the case and let $F, L$ as above. There is
$q$ in $g_i(R) \cap g_j(R)$. In addition $F$ intersects $g_i(R)$
so there is $q_1$ in $F \cap g_i(R)$. But $g_i(R)$ is a product open set,
so $\wls(q_1) = F$ intersects $\wlu(q)$ and the intersection is
in $g_i(R)$. In particular $\wlu(q)$ intersects $F$. 
In the same way one proves that $\wlu(q)$ intersects $L$, so $F, L$
cannot be distinct but non separated from each other.
It follows that the union $\cup_{i \in {\mathbb N}} g_i(I^s)$
is an open interval in the stable leaf space.
Hence $(g_i(I^s))$ converges as $i \to \infty$ to an 
interval with non empty interior, which
we denote by  $J^s$. In the same way $g_i(I^u)$
converges to an interval with non empty interior, denoted by  $J^u$.
It is easy to see that 
the pair $J^s, J^u$ determines a product open set, which
is non empty. In addition the corners of $g_i(R)$ are $g_i(\alpha),
g_i(\beta)$. The condition that $p_i$ is not a bounded distance
from $\alpha \cup \beta$ implies that these corners converge
to infinity. This produces a product open set
with no corners. 
This is impossible by Lemma \ref{noone}.

\vskip .1in
This contradiction shows that 
there is $a_2 > 0$ (a priori depending on 
$B$) so that $B$ is contained
in the $a_2$ neighborhood of $\beta$.
Recall that $B = \cup \{ \ell_t, 0 \leq t \leq 1 \}$.
Hence any $\ell_t$ is in the $a_2$ neighborhood of $\beta$.
In order to prove that the Hausdorff distance is finite
we need to prove the converse inclusion for some 
fixed size neighborhood of $\ell_t$.
This is much easier. 
To prove this 
we produce sequences $p_n$ in $\ell_t$ and $q_n$ in $\beta$
as follows, a bounded distance from each other
that will show $\beta$ is in a bounded neighborhood from $\ell_t$.
Recall that $d$ is the ambient distance
in $\mt$.  Choose a point $p_0$ in $\ell_t$.
The curve $\ell_t$ is a geodesic in $\wls(\ell_t)$ and hence it is
properly embedded in $\mt$. For $n > 0$ inductively choose
$p_{n+1}$ in $\ell_t$ so that $d(p_n, p_{n+1}) = 2 a_2 + 1$.
In addition choose these points $p_n$ so they are
nested in $\ell_t$. Hence they escape in $\ell_t$
and also in $\mt$.
Similarly for $n < 0$, choose $p_{n-1}$ so that 
$d(p_n,p_{n-1}) = 2 a_2 +1$ and also $p_n$ nested in $\ell_t$
for $n \leq 0$.
Then for each $n$ choose $q_n$ in $\beta$ with
$d(p_n,q_n) < a_2$.
Notice that $d(q_n,q_{n+1}) < 4 a_2 +1$. 
In addition $q_n$ escapes in $\mt$.

Now we use the following fact: given $4 a_2 + 1$ there is a
constant $a_3$ so that if $x, y$ are in a flow line $\zeta$ of $\wwp$
and $d(x,y) < 4a_2 + 1$, then the length of $\zeta$ between
$x$ and $y$ is less than $a_3$. Otherwise we find $x_n, y_n$
in same flow lines of $\wwp$, with $d(x_n,y_n) < 4a_2 + 1$,
but flow length from $x_n$ to $y_n$ is $> n$. Up to subsequence
and deck translates we assume $x_n \to x$ and $y_n \to y$.
Then the flow lines through $x$ and $y$ are not the same,
contradicting that the orbit space of $\wwp$ is Hausdorff.

We note another important property:
there are no $x_n, y_n$ escaping in opposite rays of
$\beta$ with $d(x_n,y_n)$ bounded. Otherwise the orbit
space of $\wwp$ would not be Hausdorff.
Hence as $n \to \infty$, the $q_n$ escape in a single
ray of $\beta$. The same happens for $n \to -\infty$.

By the fact above 
the length in $\beta$ from $q_n$ to $q_{n+1}$ 
is uniformly bounded by $a_3$, independent of $n$ or $t$.
It only depends on $4 a_2 + 1$.
Hence there is $a_4 > 0$ so that a ray of $\beta$ is 
in the $a_4$ neighborhood of 
a ray in $\ell_t$.
It follows that a ray
of $\ell_t$ is a finite and bounded Hausdorff distance from a ray of $\beta$.
The bound is independent of $\ell_t$ or $t$.

\vskip .1in
There is still one more property to prove which is the following:
different rays of $\ell_t$ are
not a finite distance from the same ray of $\beta$.

One of the rays in $\ell_t$ is aymptotic to a backwards ray of $\beta$.
There are two possibilities. 
Suppose first the subblock $\C$ is not
a pair of adjacent $(1,3)$ ideal quadrilaterals intersecting a common
unstable leaf. Under this condition,
then as explained in the proof of Claim \ref{claim-firstbound},
a forward ray of $\beta$ is a bounded distance from a backward
ray of $\alpha$, and a forward ray of $\alpha$ is a bounded distance
from a backward ray of $\beta$. Since $\ell_t$ also has a ray
asymptotic to a backward ray of $\alpha$, this ray is a bounded
distance from a forward ray of $\beta$. 
In other words two disjoint rays of $\ell_t$ cannot be
a bounded distance from the same ray of $\beta$.
This finishes the analysis in this case.

The last case is that the subblock $\C$ is a union of two adjacent
$(1,3)$ ideal quadrilaterals both intersecting a common unstable leaf.
Then $\alpha, \beta$ are in the same unstable leaf and $\ell_t$
has a ray asymptotic to a backward ray of $\beta$.
As in the previous setup $U = \wlu(\beta)$,
and $V$ is the other unstable leaf in the boundary of the block.
The other ray of $\ell_t$ is asymptotic to a backards ray 
of $E_t \cap V$. 
This ray of $\wls(\ell_t) \cap V$ is a finite Hausdorff
distance from a forward ray in $\wls(\beta)$ (since $U, \wls(\beta)$
make a perfect fit). Hence this ray is a finite Hausdorff 
distance from a positive ray of $\beta$ (again by Theorem
5.7 of \cite{Fe7}). 
Again this shows that disjoint rays of $\ell_t$ cannot
be a bounded distance from the same ray of $\beta$
and concludes in this case as well.

This completes the proof of Proposition \ref{bounded}.
\end{proof}

We now improve this result for bands associated
with lozenges.
The {\em{thickness}} of a band $B$ is the supremum
of the Hausdorff distances between leaves of $\F$ in $B$.
This uses the ambient distance in $\mt$.

\begin{corollary} \label{boundthickness}
The thickness of bands associated with lozenges
is globally bounded.
\end{corollary}

\begin{proof}
The proof of the previous proposition implies this result.
Suppose the result is not true.
Let $B_i$ be a sequence of bands with
corners $\alpha_i, \beta_i$ and $B_i$  associated
with lozenges. Suppose there are $p_i$ in $B_i$ so that 
$d(p_i,\beta_i)$ converges to infinity.
Following the steps of the proof of Proposition,
first one does not need to split the band into
subbands. Claim \ref{claim-firstbound} is not needed
as in the case of lozenges, it is exactly
Theorem 5.2 of \cite{Fe7}.
Claim \ref{claim-intervalnotshrinking} works exactly
the same. Then the discussion after this claim
first shows that there is a global $a_2 > 0$ so that
$B_i$ is contained in a bounded distance neighborhood
of $\beta_i$. The discussion after that proves the
other containment with a different constant.
\end{proof}

\begin{remark}
The global bound result of Corollary \ref{boundthickness} is
not true for general bands. 
This is very simple: given an Anosov flow there is a
positive constant $a_1 > 0$ so that if two orbits
have points within $a_1$ of each other, then the stable
of one of them intersects the unstable of the other one.
This bounds the number of orbits of $\wwp$ which are all pairwise within
a fixed bound from each other, but their stable and unstable
leaves do not intersect. So all that is needed is to
have for each $n$ a periodic chain of lozenges with more
than $n$ corners. This happens if there is an infinite
adjancent chain of lozenges, for instance in 
the Bonatti-Langevin example \cite{Bo-La}.
In fact it happens whenever there is a scalloped region.
\end{remark}


\subsection{Limits of bands of lozenges}

Here we prove an improvement of Proposition \ref{limitloz}.
We first define geometric convergence.

\begin{define} \label{geomconv}
(geometric convergence) A sequence of sets $(V_k)$
in $\mt$ converges geometrically to a set $V$ if the following holds:

$-$ For any $p$ in $V$ there are $p_k$ in $V_k$ with
$p_k \rightarrow p$.

$-$ If $(k_i)$ is a subsequence of the integers and $(p_{k_i})$ is
a sequence which converges to $p$ in $\mt$, and $p_{k_i}$ is in $V_{k_i}$
for all $i$, then $p$ is in $V$.
\end{define}

\begin{proposition} \label{limitofbands}
Let $(D_i)$ be a sequence of lozenges with corners
$\alpha_i, \beta_i$ which converge respectively to
$\alpha, \beta$ and so that $(D_i)$ converges 
to a basic block $\C$. Then the 
bands $B_i$ associated with $D_i$ converge geometrically to the
band $B$ associated with $\C$.
\end{proposition}

\begin{proof}
As in the proof of Proposition \ref{limitloz},
let $I^s_i$ be the open interval in the leaf space
of $\wls$ of leaves intersecting the lozenge $D_i$,
and let $I^s$ be the interior of the limit of the $I^s_i$.
Similarly for $I^u_i$ and $I^u$.

We start by proving the first property
of Definition \ref{geomconv}. First of all,
the corners of $D_i$ converge to $\alpha, \beta$,
so the full orbits $\alpha, \beta$ are in the limit of
the $D_i$.
We consider the case where $\C$ is an adjacent block 
with more than one element,
the case where $\C$ is a single lozenge is simpler.

First suppose that $\C$ is an adjacent block whose
elements intersect a common stable leaf.
We refer to Figure \ref{figure6}, b, \ and 
Figure \ref{figure7}
for the arguments.
Let $E$ be a leaf of $\wls$ which is in $I^s$.
Then for big enough $i$, $E$ is in $I^s_i$.
The band $B$ intersects $E$ in a geodesic,
which we denote by  $\ell_E$
with ideal points the negative ideal points of $\wlu(\beta) \cap E$
and $\wlu(\alpha) \cap E$.
The intersection of $B_i$ with $E$ is a geodesic $\ell_i$ with
ideal points the negative ideal points of $\wlu(\beta_i) \cap E$
and $\wlu(\alpha_i) \cap E$. Since $\alpha_i$ converges to $\alpha$
and similarly for $\beta_i$, then the geodesics $\ell_i$ converge
to $\ell_E$. Hence any $p$ in $\ell_E$ is the limit of a sequence $(p_i)$
with $p_i$ in $\ell_i$.
This also deals with the case that $\cC$ is a lozenge.

Suppose now that $\C$ is an unstable adjacent block with 
more than one element.
In this case we have that the $I^u_i$ is forced to increase
with $i$.
In addition this case is more complicated because it is
not only $\wlu(\alpha)$ and $\wlu(\beta)$ that have
leaves or half leaves contained in the unstable boundary
of $\cC$.
Here we think of $\cC$ as a union of a lozenge and $(2,2)$
quadrilaterals, so by corner orbit of $\cC$ we mean
a corner of one of these elements.

First we deal with the leaves $E$ in $I^s$ so that
$E$ has a corner of the block $\cC$.
In that case $E \cap B$ is an orbit of $\wwp$, which
we denote by $\ell_E$, this orbit is a corner of the block.
As in the previous case 
the intersection of $B_i$ with $E$ is the geodesic, 
denoted by $\ell_i$, with ideal points the
negative ideal points of $\wlu(\alpha_i) \cap E$ and
$\wlu(\beta_i) \cap E$.
Without loss of generality we assume that the corner
$\ell_E$ is the limit of $\wlu(\alpha) \cap E$.
Then as in the proof of Lemma \ref{band2} it follows that
$\ell_i$ converges to $\ell_E$ in the sense that
any point in $\ell_E$ is the limit of points in $\ell_i$.

For the leaves $E$ in $I^s$ so that $E$ does not
contain a corner of $\cC$ then there are leaves
$U, V$ of $\wlu$ so that $E \cap B$ is equal
to the geodesic $\ell_E$ with ideal points the negative
ideal points of $U \cap E$ and $V \cap E$.
If $\ell_i = B_i \cap E$, then again $\ell_i$ converges
to $\ell_E$.
This finishes the proof that any $p$ in $B$ is the limit of $(p_i)$
with $p_i$ in $B_i$.

\vskip .05in
Now suppose that $p$ is the limit of 
$(p_{k_i})$ with 
$p_{k_i}$ in $B_{k_i}$. We need to show that $p$ is in $B$.
The limit of any subsequences of $(\alpha_i), (\beta_i)$ still
converges to $\alpha, \beta$ respectively. The band
$B$ does not change by taking a subsequence. Hence for the purposes
of this argument we may assume for simplicity of notation that
$(k_i)$ is the original sequence of positive integers.
With this understanding let $p$ be the limit of $(p_i)$ with
$p_i$ in $B_i$.
If up to subsequence the $p_i$ are in $\wls(\beta_i)$
(or $\wls(\alpha_i))$ then $p_i$ are in fact in $\beta_i$
(or in $\alpha_i$). Hence the limit $p$ is in $\beta$
(or $\alpha$) $-$ because the orbit space of $\wwp$ is
Hausdorff. Therefore $p$ is in $B$. So we assume from now
on that $p_i$ is never in $\wls(\beta_i) \cup 
\wls(\alpha_i)$.

\vskip .05in
Suppose first that all elements in $\C$ intersect a common
stable leaf.
Recall the open intervals $I^s_i$. To these we adjoin
$\wls(\alpha_i), \wls(\beta_i)$ to obtain closed intervals
$J^s_i$ in the leaf space of $\wls$. In addition adjoin
$\wls(\alpha), \wls(\beta)$ to $I^s$ to obtain the closed
interval $J^s$.
Parametrize $J^s$ as $E_t, 0 \leq t \leq 1$ with
$\wls(\alpha) = E_0$ and $\wls(\beta) = E_1$.

Let $U_i = \wlu(\alpha_i), V_i = \wlu(\beta_i),
U = \wlu(\alpha), V = \wlu(\beta)$.

Suppose first that $I^s_i$ eventually
increases, so assume it always increases.
Hence $p_i$ is in
a leaf $E_{t_i}$ in $I^s$, $0 < t < 1$.

Suppose up to subsequence that $t_i$ converges to $t_0$.
Suppose first that $0 < t_0 < 1$. 
The point $p_i$ is in the geodesic $\ell_i$ in $E_{t_i}$ whose
endpoints are the negative ideal  of $U_i \cap E_{t_i}$
and $V_i \cap E_{t_i}$. 
We have that $U_i$ converges to $U$ and $V_i$ converges to $V$,
hence $U_i \cap E_{t_i}$ converges to $U \cap E_{t_0}$.
In addition the leaves $E_{t_i}$ and $E_{t_0}$ are asymptotic
to each other near the direction of the negative ideal point
of $U \cap E_{t_0}$. Similarly
for $V_i$. This shows that any
limit point of points in $\ell_i$ are in the geodesic $\ell$ in
$E_{t_0}$ with ideal points the negative ideal points
of $U \cap E_{t_0}$ and $V \cap E_{t_0}$.
It follows that $p$ is in $\ell$. But by definition, $p$ is
in the band associated with the block $B$.
This finishes the case $0 < t_0 < 1$.

Now without loss of generality, assume that $t_0 = 1$,
that is $E_{t_0} = \wls(\beta)$.
If $p$ is in $E_1$, and $p$ is not
in $\beta$, then a proof as in Lemma \ref{band1}
shows this is a contradiction. Suppose now that $p$ is in $F$
distinct from $\wls(\beta)$, but non separated from $\wls(\beta)$. 
Then the proof of Lemma \ref{band2} shows that this
is impossible.
We conclude that $p$ is in $B$, and that $B$
 is the geometric limit
of the bands $B_i$.

Now suppose that $I^s_i$ is not monotone, in which case it
follows that $\cC$ is a single lozenge (recall we are still
assuming  that the limiting block has all elements intersecting
a common stable leaf).
Suppose first that up to subsequence then $p_i$ is
in $I^s$. Then the above proof also works in this case.
Suppose now that $p_i$ is not in $I^s$.
Without loss of generality
we assume that $\wls(p_i)$ converges to $\wls(\beta)$ (and maybe
other leaves). The other case would be $\wls(p_i)$  converges
to $\alpha$ which is dealt with similarly.
It follows that $p$ is either in $\wls(\beta)$ or
in a leaf $F$ non separated from $\wls(\beta)$.
Then a proof as in Lemma \ref{band2} implies that 
$p$ has to be in $\beta$.

\vskip .1in
This finishes the proof when the limiting block has 
all elements intersecting a common stable leaf.
Next assume that the limit block $\C$ is not a single
lozenge, and its elements all intersect
a common unstable leaf. The same proof as above applies, with the
difference that the lozenges $D_i$ increase in the stable
direction, instead of decreasing.

This finishes the proof of Proposition \ref{limitofbands}.
\end{proof}

\section{Construction of walls}
\label{construction}

We now define another object that will be extremely important
for our analysis.
Recall that throughout this article we 
assume a metric in $M$ which makes each leaf of $\ls$
into a hyperbolic surface.
Conversely one could have considered the unstable
foliation in place of the stable foliation, and an adjusted
definition of walls below. 
Recall that if $B$ is a band associated to a block
$\cC$ with corners $\alpha, \beta$ then the set of
stable leaves intersected by $B$ is a compact interval
$I$ in the leaf space $\wls$ with endpoints
$\wls(\alpha), \wlu(\beta)$. 
The collection of unstable leaves intersected by $B$ is
contained in the closure of an open segment. The
open segment is the set of unstable leaves intersecting
$\cC$, and the additional points include $\wlu(\beta),
\wlu(\alpha)$. There are further points only
when $\cC$ is an unstable adjacent block with more than
one element.
The band $B$ intersects leaves of $\wls$ in geodesics,
inducing a foliation $\cF$ in $B$. If one puts
an orientation in $\cF$ then in one and only one of $\alpha, \beta$
the orientation agrees with the positive flow direction.

\begin{define}{(Walls)}  \label{def-walls}
Let $a_0 > 0$ be a global bound on the thickness of any
band which is the limit of bands associated with 
lozenges. 
A wall $W$ is a bi-infinite union of bands $B_i, i \in 
\mathbb{Z}$,  
each associated with a respective finite adjacent block
having thickness bounded by $a_0 + 1$,
and in addition satisfying:
\ 1) $B_i$ intersects $B_j$ only if $|i-j| = 1$ and then they
only intersect in a boundary component of each,
\ 2) The set of stable leaves intersected by $W$ is
a properly embedded copy of the reals $\rrrr$ in the stable
leaf space. 
\ 3) The set of unstable leaves intersecting $W$ which 
are not boundary points of this set is a properly embedded copy of $\rrrr$
in the unstable leaf space.
A partial wall $W$ is a finite union of $B_i$ satisfying \ 1)
and so that the set of stable leaves intersected
by $W$ is a compact interval, and a similar condition for the
unstable leaves.
\end{define}

A {\em boundary point} of the set of unstable leaves intersected
by $W$ is an unstable leaf $U$ so that $U$ intersects $W$ but
there is not an open  transversal $\tau$ to $\wlu$ intersecting $U$
and so that all unstable leaves intersecting $\tau$ also intersect
$W$. Does this happen? Yes, it does if there is an unstable 
adjacent block which is not a single lozenge. Let $B$ be such a block
associated to a product open set $R$ with corners $\alpha, \beta$.
Then $R$ has other corners, let $\gamma$ be one such other corner.
Then $\gamma$ is contained in $W$, and $\wlu(\gamma)$ intersects
$W$, but only one component of $\mt \setminus \wlu(\gamma)$
intersects $W$. Then $\wlu(\gamma)$ is one such boundary point.

\begin{remark} Why the condition on thickness? This is because
we want to take limits of partial walls and walls and still
get objects of the same type. If there is no global bound,
for example if there is a scalloped region, one could
do the following: consider longer and longer stable adjacent
blocks. Then the limit of these is not a block, and
will not be a wall $-$ the set of stable leaves intersected
by the limit is not a properly embedded copy of $\rrrr$ in
the leaf space, just to name one problem.
In the case of atoroidal manifolds, there is an upper
bound on chains of lozenges intersecting a common stable
or unstable leaf, so they satisfy the global bound
in the definition, and the restriction is vacuous.
\end{remark}

We now describe the construction of walls associated with 
large sets of freely homotopic periodic orbits of $\Phi$.
In Theorem 5.10 of \cite{Fe7}
we proved that if $\Phi$ is an Anosov flow in $M^3$ atoroidal and
$\Phi$ is not quasigeodesic, then the following happens:
for each $i > 0$ there is a chain $\C_i$ which is a string
of lozenges with $2i$ lozenges and with periodic corners.
A chain of lozenges is a {\em string of lozenges} if
no two consecutive lozenges are adjacent. In other words
consecutive lozenges always intersect opposite quadrants of
the corner in question.
The projection to $M$ of the collection of corners produces
a set of pairwise distinct freely homotopic orbits of $\Phi$.
Let $\beta_{ij}$, $-i \leq j \leq i$ be the consecutive
corners of $\C_i$. 
For each $i$ consider the union of the 
bands associated with the lozenges in $\C_i$. This union
is a partial wall.
The last property uses the fact that $\C_i$ is a string
of lozenges.
Otherwise one cannot guarantee that the set of stable leaves
intersect by the union is a compact interval $-$ but 
this could be true even if the chain is not a string of lozenges.
The band associated with corners $\beta_{i,j}$
and $\beta_{i,j+1}$ is denoted by $B_{i,j}$.
For each $i$ the union of the bands
$B_{i,j}, -i \leq j \leq i-1$ $-$ which is a partial
wall $-$ is denoted by $O_i$.

We will take limits of the partial walls associated with
this sequence of string of lozenges as follows:
Let $O_i$ be a sequence of partial walls as above. 
Pick points $p_i$ in $B_{i,0}$.
By compactness of $M$, there is a subsequence, still denoted
by $p_i$ so that $g_i(p_i)$ converges to a point $p_0$ in $\mt$.
When we take a subsequence we can shorten the $O_i$ so that
it still has $2i$ elements.
Since any of the bands are bands associated with 
lozenges, they have uniformly bounded thickness
by Corollary \ref{boundthickness}.
Hence up to subsequence we may assume that one of the corners,
say $g_k(\beta_{i,0})$ converges to an orbit $\gamma_0$.
Since the Hausdorff distance between orbits which are corners
of a lozenge are globally bounded, we may assume that
$g_i(\beta_{i,1})$ also converges, and to an orbit $\gamma_1$.
By Proposition \ref{limitofbands} the bands $g_i(B_{i,0})$ converge
to the band associated with $\gamma_0, \gamma_1$.
This is an odd adjacent block. We denote it by $D_0$.
The results of the previous section show this is a
partial wall, with uniformly bounded thickness.

Using a diagonal process assume that 
for each $n$, the bands $g_i(B_{i,n})$ also converge as $n \to \infty$,
let the limits be denoted by $D_n$, and 
let the limits of the boundary orbits be $\gamma_n, \gamma_{n+1}$.
Notice that given $n$, these bands only exist if 
$|n+1|  \leq i$.
Since $n$ is fixed, this is satisfied whenever $i$ is big
enough.

Any limit of the partial walls $g_i(O_i)$ is a wall.
The collection of all such limits is denoted by $\WW$.
Just for definiteness we state the following result:

\begin{lemma} $\WW$ is the collection of all walls.
In addition if $W_n$ is a collection of walls obtained by
the process above, and $W_n$ converges to a set $W$, then
$W$ is also a wall and $W$ is contained in $\WW$.
\end{lemma}

\begin{proof}
The first result is because a wall is a union of partial
walls, so if we take an increasing collection of partial
walls with union the whole wall, then the collection
of partial walls converges to the wall.

The second property is also easy to obtain:
let $p$ in $W$, then it is the limit of $p_n$ in $W_n$.
For each $n$ choose a partial wall in $W_n$ with $2n$ elements
and ``centered around" $p_n$. 
Since $W_n$ converges to $W$, it follows that these partial
walls converge as well. The limit has to be $W$, so
$W$ is a wall, and is contained in $\WW$.
\end{proof}

We stress again that the important point is that if $W_n$
is a collection of walls or partial wals with points $p_n$
converging to $p_0$, then up to a subsequence $W_n$ converges
to a wall or a partial wall. 

\begin{remark} In Section \ref{examp} we give several
examples of walls in atoroidal and toroidal manifolds.
Looking at the examples in that section may help in understanding
some of the results in the following sections.
\end{remark}

%
%

\section{Translates with no transverse intersections}
\label{translates}

In this section we work toward 
the proof of the Main theorem under the following
condition: suppose there is a wall
$G$ as constructed in the limit process of the last section,
such that no deck translate of 
$G$ intersects $G$ transversely.
This is existence, this is not saying that all limit
walls satisfy this property.
By a {\em transversal intersection}  of a wall $W$
we mean there is $g$ a deck transformation, so that
$g(W), W$ intersect in a point $p$ satisfying the
following: $p$ is in a leaf $E$ of
$\wls$, and $W \cap E, \ g(W) \cap E$ are transverse geodesics
in $E$ intersecting in $p$.

\vskip .05in
\noindent
{\bf {Standing hypothesis}} Throughout the proof of the
main theorem, we assume that there are bigger and bigger
strings of freely homotopic orbits and we produce 
{\underline {all}} limits
of bands associated with them, as in the previous
section, generating the closed set of walls, denoted
by $\mathcal L$.
This setup does not use that $M$ is atoroidal.
This is because we are using {\underline {strings}} of lozenges,
and limits associated to those. In this case the thickness
of the individual blocks even in the limit is always bounded.
The atoroidal hypothesis will be used in some specific 
places to obtain further results.
\vskip .05in

Let then $W$ be a wall and suppose that $g(W), W$ intersect,
but not transversely. Ideally one would like to prove that
this always implies that $W = g(W)$. However this is
not necessarily true in general. But when $M$ is atoroidal, 
this turns out to be true.
We analyze these non transverse intersections to start with.

\begin{define}{(corners and hinges of walls)}
Given a wall $W$, a corner of $W$ is a flow line of
$\wwp$ which is contained in $W$. 
A hinge of $W$ is a corner of $W$ so that its 
unstable leaf is an interior point of the set of
unstable leaves intersected by the wall.
\end{define}

Notice that a corner of $W$ 
is always a corner $\eta$ of an element
of the bi-infinite block associated with $W$ so that
$\eta$ is contained in $W$. This is a flow line of $\wwp$.
In general there are corners of $W$ which 
are not hinges of $W$. They are exactly the leaves of
$W \cap \wls$ (the foliation $\cF$) which are contained
in a boundary point of the set of unstable leaves
intesected by $W$.

If $W$ is obtained from the limit of a 
 sequence of strings of lozenes
lozenges, then a hinge is a limit of corners of the lozenges
in the sequence. 
A corner which is not a hinge occurs if there is 
a unstable  adjacent block $B$ in the limit 
with more
than one element.
The intermediate corners of $B$ are not hinges, while
the end corners are hinges. 
In a stable adjacent block, only the end orbits are corners
in $W$. No other corner of an element in the block is
contained in $W$.


Let $\alpha$ be an orbit of $\wwp$. A {\em quadrant} $Q$ 
associated with 
$\alpha$ is one of the connected components of $\mt -
(\wls(\alpha) \cup \wlu(\alpha))$. There are four
quadrants associated with $\alpha$.
Two quadrants are adjacent if they intersect a common
stable or unstable leaf.
Let $\alpha$ be an orbit of $\wwp$ contained in a wall
$W$. There are two subbands of $W$ which have $\alpha$
as a boundary component. Each of these subbands enters
a distinct quadrant associated with $\alpha$.
The quadrants may be adjacent or not.

\vskip .05in
\noindent
{\bf {Study of 
properties of non transverse intersections
between translates}}

These properties, which we describe now,  are a progression to prove that
non transversal intersection implies $W = g(W)$
when $M$ is atoroidal.
In Lemma \ref{commoncorner} we prove that 
if $W, g(W)$ intersect, but non transversely,
then $W, g(W)$ share a corner.
In Lemma \ref{cornerperio} we prove that if a wall $W$ has a 
periodic corner, then all corners are periodic and there
are no $(1,3)$ ideal quadrilaterals in the bi-infinite
block $\C$ associated with $W$.
In Lemma \ref{morethanone} we prove that if the block 
$\C$ associated with $W$ 
has a basic block which is an adjacent block with more
than one element, all of which are lozenges, then all
the corners of $W$ are periodic.
All of these results do not strictly need $M$ atoroidal
just the bound that is assumed in Definition \ref{def-walls}.
In Lemma \ref{sameband} we need to assume $M$ 
atoroidal $-$ here we prove that if $W, g(W)$ 
share a corner $\eta$ and enter the same quadrant $Q$ of $\eta$,
then $W, g(W)$ agree on the subbands entering $Q$.
For this we also need that $g$ preserves orientations.
In the proof of this result, the atoroidal hypothesis
is only used when the corners of $W$ are periodic.
The next result is Lemma \ref{eitherperiodic} which does not
need $M$ atoroidal: assume $g$ preserves orientations,
$W, g(W)$ share a corner orbit $\eta$, but $W, g(W)$
do not enter exactly the same quadrants at $\eta$.
Then all corners of $W, g(W)$ are periodic.
The final result is Proposition \ref{identical}, which 
shows that if $M$ atoroidal, and $g$ preserves orientations,
then non transversal intersection implies $W = g(W)$.
Some of the proofs are very intricate because there are 
many possibilities to be analyzed.

\begin{lemma} \label{commoncorner}
Let $W$ be a wall.
Let $g$ in $\pi_1(M)$ so that
$g(W)$ and $W$ intersect, but not transversely.
Then $W$ and $g(W)$ share a corner orbit.
\end{lemma}

\begin{proof}
Let $p$ be a point of intersection of $W$ and $g(W)$.
Let $E$ be the leaf of $\wls$ which contains $p$.
Then $W \cap E$, \ $g(W) \cap E$ are geodesics in $E$ 
intersecting at $p$, but not transversely. As the leaves
of $\wls$ have the hyperbolic metric, it follows that 
$W \cap E \ = \ g(W) \cap E$.
Denote this intersection by $\mu$.

If $\mu$ is tangent to the flow, then it is a corner 
of both $W$ and $g(W)$ and  the result is obvious.
Assume then that $\mu$ is transverse to the flow
and we will produce a common corner of $W$ and $\gamma(W)$.

Let $\C$ be the bi-infinite block associated with $W$
and $\E$ the bi-infinite block associated with $g(W)$.
There are several possibilities to consider. 
Let $R$ be the basic adjacent block in $\C$ which contains
$\mu$ and $R_g$ be the one in $\E$ containing $\mu$.
By that that we mean that the end corners of $R$ are
hinges of $W$, no interior corner of $R$ is a hinge,
and similarly for $R_g$.
Both $R$ and $R_g$
are adjacent blocks (each of which could be a lozenge), 
hence associated with 
product open sets, denoted the same way,
 see Remark \ref{samenotation}.
It follows that $\mu$ intersects
all unstable leaves which intersect $R$
and likewise for $R_g$.  
This is not a priori true for the stable foliation.
We also stress that it is not necessarily true that
$R_g  = g(R)$ (but there is such a block $R'$ in
$\cC$ so that $R_g = g(R')$).

\vskip .05in
\noindent
{\bf {Case 1 $-$ One of $R$ or $R_g$ 
is a stable adjacent
block.}}

Up to switching $R, R_g$ assume that $R$
is a stable adjacent block.
Let $\beta$ be a corner of the band 
associated with $R_g$.
Let $V = \wlu(\beta)$.
Since $R$ is a stable 
adjacent block, it follows that $\partial R \cap V$
is a a half leaf of $V$, which we denote by $V'$.
Let $V_2 = \partial R_g \cap V$. Notice that 
the common curve $\eta$ shows that $V_2, V'$ share
a non degenerate interval. 
Let $\alpha$ be the other corner of $R$.

Notice that $\wls(\beta)$ has a half
leaf with boundary $\beta$ which is $\partial R \cap \wls(\beta)$.
We denote this half leaf by $F$.
Suppose first that there is an orbit $\gamma$ in 
$V \setminus V'$ which is in the boundary of $R_g$, 
in other words $R_g$ goes beyond $L$.
Without loss of generality assume that $\wls(\gamma)$ intersects
$R_g$, and let $E = \wls(\gamma) \cap R_g$.
Since $R_g$ is a product open set it follows that 
every unstable leaf intersecting $E$ intersects
$\eta$, and consequently intersects $F$.
Conversely every leaf intersecting $\eta$ has to intersect $F$,
so this implies that 
$\wls(\beta)$ makes a perfect fit with $\wlu(\alpha)$.
Suppose first that $E$ is a full half leaf of $\wls(\gamma)$.
Then since there are no
product regions, it follows that there are unstable
leaves non separated from $\wlu(\alpha)$ in the boundary
of $R_g$ between $\wls(\beta)$ and $\wls(\gamma)$.
The region bounded by $\wls(\gamma), \wls(\beta), \wlu(\gamma)$
and the non separated leaves we just produced is contained
in $R_g$ and it is a union of $(2,2)$ quadrilaterals.
it follows that $\beta$ is also a corner of $R_g$.
It is a hinge of $W$ but not a hinge of $g(W)$.

On the other hand suppose that $\wls(\gamma) \cap R_g$ is
not a half leaf, and let $\gamma'$ be
the other boundary component of this set. Then 
$\wlu(\gamma')$ is non separated from 
$\wlu(\alpha)$. If it does not make
a perfect fit with $\wls(\beta)$, then there is another
unstable leaf $U$ entirely contained in $\partial R_g$,
non separated from $\wlu(\alpha)$ and in between 
$\wls(\beta)$ and $\wls(\gamma)$. This reduces to
the previous case, showing that $\beta$ is a corner
(non hinge) of $R_g$.
Finally suppose that $\wlu(\gamma')$ makes a perfect fit with
$\wls(\beta)$. We already know that $R_g$ is an unstable
adjacent block with more than one element in this case. 
$F$ is a side of two consecutive elements in $R_g$,
and again it follows that $\beta$ is a corner of $R_g$.

The remaing case is that $V_2 \subset V'$. 
To analyze this case we start from $\alpha$ instead of
$\beta$. Let $U = \wlu(\alpha)$, and $U' = \partial R \cap U$,
and so $U'$ is a half leaf of $U$ with boundary point
$\alpha$. Let $U_2 = \partial R \cap U$. 
By the analysis we just did we can deal with all
cases unless $U_2 \subset U'$. 
Suppose first that $V'$ is a half leaf of of $V$: 
this implies that either a half leaf of $\wls(\alpha)$ is
contained in $\partial R$ or $U_2$ has a point 
beyond $\alpha$. In the second case we get by the
analysis above that $\alpha$ is a hinge of $B$ and
a corner of $B'$ which is not a hinge. In the
first case we get that $\alpha$ is an extremal
corner of both $R, R_g$ and so $\alpha$ is a
hinge of both $B$ and
$B'$.
Finally suppose that both $U_2$ and $V_2$ are bounded 
intervals of orbits in $U, V$ respectively. But this
implies that the projection $\Theta(R_g)$ in $\oo$
is a rectangle with compact closure. This is impossible
for the blocks we are considering.

This finishes the analysis when one of the blocks
$R$ or $R_g$ is a stable adjacent block (including the
case one or both are lozenges).

\begin{figure}[ht]
\begin{center}
\includegraphics[scale=1.00]{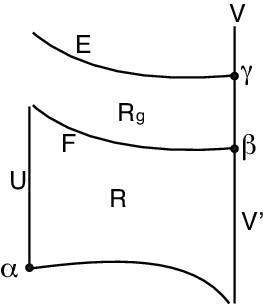}
\begin{picture}(0,0)
\end{picture}
\end{center}
\vspace{0.0cm}
\caption{Depiction of a situation in Case 1. The 
stable leaves $F, E$ make perfect fits with unstable leaves
non separated from $\wlu(\alpha)$.}
\label{figure9}
\end{figure}

\vskip .1in
\noindent
{\bf {Case 2  $-$
The blocks $R$ and $R_{\gamma}$ are
both unstable adjacent blocks with
more than one element.}}

There are several possibilities, some of them similar
to situations in Case 1.
The geodesic $\mu$ in $E$ has ideal points
the negative ideal points of two distinct flow lines,
defining two unstable leaves $U, V$ so that $\mu$
has ideal point the negative ideal points of $U \cap E$
and $V \cap E$.
Also the block $R$ has a subblock which contains $\mu$,
this is either a lozenge or a union of two adjacent 
$(1,3)$ ideal quadrilaterals forming a  $(2,2)$ ideal
quadrilateral which is not a lozenge. 
In either case let this subblock be denoted by $\B$.
In the same way we define $\B_g$: it is the
subblock of $R_g$ containing $\mu$.


Suppose first that the subblocks $\B, \B_g$ are
both lozenges. 
Both have a side in $V$. If $\partial \B, \ \partial \B_g$
share a half leaf of say $V$, then $\B, \B~$ have sides in 
exactly the same leaves so $\B = \B_g$), and they share a corner,
as we wanted to prove. Otherwise $\partial \B \cup \partial
\B_g$ contains $V$. This is impossible: 
$\partial \B \cap V$ (which is a side of $\B$) makes a perfect
fit with a stable leaf denoted by $L$ which intersects 
$U$. In the same way $\partial \B_g \cap U$ (which is a side
of $\B_g$) makes a perfect fit with a stable leaf $F$ which
intersects $V$. Since these two perfect fits are on the same
side of $\wls(\mu)$ it follows that $L, F$ intersect,
but are not the same. This is impossible.

Suppose now that one of the subblocks is a lozenge
and the other is a union of a pair of $(1,3)$ 
ideal quadrilaterals which have a common stable side.
Without loss of generality assume that $\B$ is the 
union of the two $(1,3)$ quadrilaterals, and has
unstable sides in $U, V$, entirely containing $V$.
Since $\B_g$ has a side in $V$ it follows that it
has a corner in $V$, which we denote by $\gamma$. 
Then a half leaf of $\wls(\gamma)$ is contained in
$\partial \B_g$. But this is impossible as
this half leaf intersects the side of $\B$ contained
in $U$, so $\B, \B_g$ cannot intersect the same set
of unstable leaves.

Finally  suppose that both blocks are unions of pairs
of $(1,3)$ ideal quadrilaterals. The unstable sides
of these blocks are contained in $U, V$. 
Each of the pairs $\B, \B_g$ has  a full unstable
leaf contained in its boundary $-$ this unstable leaf
can only be $U$ or $V$. 
Suppose first that $U$ is contained in the boundary
of (say) $\B$ and $V$ is contained in the boundary
of $\B_g$. Then $U$ makes a perfect fit with two 
stable leaves which intersect $V$ and similarly
for $V$. This is impossible.
Hence either $U$ is contained in the unstable  boundary
of both $\B, \B_g$ or $V$ is. Without loss of generality
assume that $U$ is contained in both unstable boundaries.
Hence the stable leaves in the boundary of $\B, \B_g$
making a perfect fit with $U$ are the same leaves,
and so are the intersections with the unstable leaf $V$.
It follows that $\B = \B_g$ and hence they 
share corner orbits.

This finishes the proof of Lemma \ref{commoncorner}.
\end{proof}

We also need the following two simple results:

\begin{lemma} \label{cornerperio}
Let $W$ be a wall, with associated bi-infinite
block $\C$. Suppose that a corner of $\C$ is periodic.
Then there are no $(1,3)$ ideal quadrilaterals in $\C$, 
that is, the elements of $\C$ are all lozenges.
In addition all corners of $\C$ are periodic.
\end{lemma}

\begin{proof}
Suppose 
that a corner $\delta$ of $\C$ is periodic,
left invariant by a non trivial deck transformation
$g$. 
Up to taking a power of $g$ assume it preserves
both transversal orientations.
Recall that a corner of $W$ is also a corner of $\C$.

Suppose by way of contradiction that 
$\delta$ is a corner of an element $C$ 
of $\C$ which is a $(1,3)$ ideal 
quadrilateral. There is $C_1$ another $(1,3)$ quadrilateral
in $\C$ which is adjacent to $C$ along a flow band.
Their union is a $(2,2)$ quadrilateral, which we denote
by $C^*$. Let $E$ be the unique stable or unstable
leaf which is entirely contained in $\partial C^*$.
This leaf has a (unique) periodic orbit which we denote by
$\gamma$.
Recall that the 
splitting of $C^*$ into $C \cup C'$ was done along
the stable or unstable leaf of $\gamma$ which is not 
$E$. It follows that $\delta$ is not in this leaf,
and $\delta$ is in a corner of $C$ contained in a half
leaf in the boundary of $C$.

Let $\delta_1$ be the other corner of the union $C \cup C_1$.
We proved above that 
a half leaf of either $\wls(\delta)$ or $\wlu(\delta)$
 is a side of $C$. 
Assume without loss of generality that a half leaf of $\wls(\delta)$
is a side of $C$.
Since $g(\wls(\delta)) = \wls(\delta)$ and
$E$ makes a perfect fit with $\wls(\delta)$ it follows
that $g(E) = e$. This uses the orientation hypothesis.
In the same way $E$ makes a perfect fit with $\wls(\delta_1)$
so $g(\wls(\delta_1)) = \wls(\delta_1)$.
On the other hand $\wlu(\delta_1) = \wlu(\delta)$, so
also $g(\wlu(\delta_1)) = \wlu(\delta_1)$.
So $g$ fixes the intersection of $\wls(\delta_1)$
with $\wlu(\delta_1)$, which is exactly $\delta_1$.
In other words $g(\delta_1) = \delta_1$.
This is a contradiction because we obtain two distinct orbit
$\delta, \delta_1$ in $\wlu(\delta)$ which are both invariant
under $g$.

It follows that $\delta$ cannot be a corner of a $(1,3)$
ideal quadrilateral in $\C$. Hence $\delta$ is the corner
of two lozenges in $\C$. The other corners of these lozenges are also
invariant under $g$. We can now restart the analysis with
these two invariant corners and proceed. 
It follows that all corners of $\C$ are invariant under
$g$ and there are no $(1,3)$ ideal quadrilaterals in $\C$.
\end{proof}

This lemma also works for partial walls.

\begin{lemma} \label{morethanone}
Let $W$ be a wall with associated bi-infinite block $\C$.
Suppose that $\C$ has a basic block $R$  which is either 
a stable adjancent block or an unstable adjacent block,
which has more than one
lozenge. Hence the corners of $R$ are 
periodic.
It follows that all corners of $\C$ are periodic.
\end{lemma}

\begin{proof}
If the block has more than one lozenge, then by
the description of basic blocks, it follows that
this basic block only has lozenges.

Assume without loss of generality that $R$ is a stable
adjacent block.
Let $C_0,C_1,..., C_{2j}$ be the elements of $R$.
Let $\alpha, \beta$ be the corners of $R$ with
$\alpha$ a corner of $C_0$. 
Let $Z$ be the stable leaf entirely contained in
$\partial(C_0 \cup C_1)$. Since $j \geq 1$,
it follows that $Z$ is non separated from
another stable leaf, which has at least a half
leaf contained in $\partial R$. It 
follows that $Z$ is periodic.
Since $\wlu(\alpha)$ makes a perfect fit with $Z$,
it follows that $\wlu(\alpha)$ is also periodic,
and let $\gamma$ be the periodic orbit in $\wlu(\alpha)$.
Let $\delta$ be the periodic orbit in $Z$. Notice that
$\wlu(\delta)$ makes a perfect with $\wls(\alpha)$ and
also another stable leaf which has half leaves
in $\partial C_1$ and $\partial C_2$.
Let $g$ be a non trivial deck translation 
preserving $\delta$ 
and preserving orientations. So $g$ preserves $\gamma$ as well.

We want to show that $\gamma$ is equal to $\alpha$.
Since both $\gamma, \delta$ are preserved by $g$ and
$\wlu(\gamma)$ makes a perfect fit with $\wls(\delta)$,
it follows that $\wls(\gamma)$ also makes
a perfect fit with $\wlu(\delta)$. 
We also know that $\wlu(\delta)$ has a half
leaf which is in the boundary of both $C_0$ and $C_1$.
In particular $\wlu(\delta)$ makes a perfect fit
with $\wls(\alpha)$ (as $C_0$ is a lozenge). 
So both $\wls(\alpha)$ and $\wls(\gamma)$ make a perfect
fit with $\wlu(\delta)$, which implies that $\alpha = \gamma$.
Once one corner is periodic, the previous lemma shows
that all corners are periodic.
\end{proof}

Notice that it was crucial in the last lemma that
$C_0$ is a lozenge, and there is more than one
lozenge in the block considered.

We first analyze two configurations, which will come
up in some arguments.

\vskip .05in
\noindent
{\bf {Configuration 0 $-$}} Let $W$ be a wall 
with associated bi-infinite block $\cC$. Let $\alpha$
a hinge of $W$. Let $Z \in \wls$ intersecting
$\wlu(\alpha)$, let $S \in \wlu$ intersecting 
$\wls(\alpha)$, and so that $Z, S$ make a perfect fit.
Let $Q$ be the quadrant at $\alpha$ so that contains
the perfect fit between half leaves of $Z, S$.
Then $W$ cannot enter $Q$.
\vskip .05in

The proof is by contradiction. Suppose this occurs.
Let $P$ be the element of $\cC$ with a corner
in $\alpha$ and contained in $Q$.
Suppose first that the element $P$ is a lozenge.
Then it has a side contained in a stable leaf $T_0$
making a perfect fit with $\wlu(\alpha)$ and
a side in a leaf $T_1$ making a perfect 
fit with $\wls(\alpha)$. Both $T_0, T_1$ are contained
in $Q$. The leaves $T_0, T_1$ have to intersect in the other corner
of the lozenge, but $Z$ separates $T_0$ from $T_1$ showing
this situation is impossible.
This contradiction is also the content of the ``Non corner
criterion", which is \cite[Lemma 2.29]{BFrM}.

\begin{figure}[ht]
\begin{center}
\includegraphics[scale=1.20]{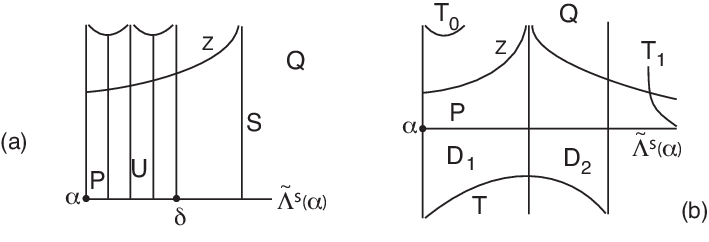}
\begin{picture}(0,0)
\end{picture}
\end{center}
\vspace{0.0cm}
\caption{(a) Figure used for Configuration $0$,
(b) Figure used to start the analysis of Configuration $1$.}
\label{figure10}
\end{figure}

Suppose now that $P$ is a $(1,3)$ ideal quadrilateral.
The situation is fairly symmetric here in terms of
stable and unstable 
so we assume that the pair of
adjacent ideal quadrilaterals, one of which is $P$,
intersects a common stable leaf.
Let $U \not = \wlu(\alpha)$ be the 
unstable leaf which also has a half
leaf in the boundary of the union of the pair of 
$(1,3)$ ideal quadrilaterals, one of which is
$P$. We refer to Figure \ref{figure10} (a).
The important thing to notice is that $U$ intersects 
$\wls(\alpha)$ and the intersection is a corner
of the union of the pair of $(1,3)$ ideal quadrilaterals.

Iterate until the last pair of ideal quadrilaterals in the
basic block of $\cC$ containing $P$,
with a last corner denoted by $\delta$ which is still
in $\wls(\alpha)$. A priori this could be
the last corner of the basic block of $\cC$ containing
$P$ (that is, $\delta$ would be a hinge of the 
basic block), or there could be more elements in the
basic block, the first of which necessarily would be a lozenge.
Here is where we are going to use the fact that
$\alpha$ is a hinge, so it is the first corner in
the basic block. Since we only had pairs of quadrilaterals,
we have an even number of elements, so eventually
we will hit a lozenge.
This lozenge
has a corner in $\delta$ in $\wls(\alpha)$ and is contained in $Q$.
The analysis of the lozenge case shows that this
is impossible.

This finishes the analysis of Configuration 0.

\vskip .1in
\noindent
{\bf {Configuration 1 $-$}} Suppose that $\alpha$ is a corner of 
a wall $W$ and $\alpha$ is in the boundary of a lozenge $D_1$,
part of two
adjacent lozenges $D_1, D_2$ and so that either $\wls(\alpha)$
or $\wlu(\alpha)$ 
intersects the interior of $D_1$ and $D_2$.
Let $Q$ be the quadrant at $\alpha$ which contains
half leaves of two distinct stable or unstable leaves which are
in the boundary of $D_1 \cup D_2$.
Then $W$ cannot enter the quadrant $P$.
\vskip .1in

We refer to Figure \ref{figure10} (b).
Let $\cC$ be the bi-infinite block associated with 
$W$. We assume by way of contradiction this is not
true and let $P$ be the element
of $\cC$ with a corner 
in $\alpha$ and contained in the quadrant $Q$.
Suppose first that $P$ is a lozenge. 
In this case without loss of generality 
we can assume assume that $\wls(\alpha)$
intersects the interiors of $D_1, D_2$.
Let $Z$ be the stable leaf which contains a half leaf
in boundary of $D_1$ and $Z$ intersecting $\wlu(\alpha)$.
Let $T$ be the stable leaf which is entirely
contained in $\partial(D_1 \cup D_2)$.
Two of the sides
of $P$ 
are half leaves in $\wls(\alpha)$ and $\wlu(\alpha)$.
The other two sides are in leaves $T_0$ that make a perfect
fit with $\wlu(\alpha)$ and $T_1$ which makes a perfect fit 
with $\wls(\alpha)$. The leaf $Z$ separates $T_0$ from
$T_1$, so $T_0, T_1$ could not intersect, contradicting
the lozenge possibility.

Let $Q'$ be the quadrant at $\alpha$ containing $T$.

The other possibility is that $P$ is a $(1,3)$ ideal
quadrilateral. 
This is part of a pair of adjacent
ideal quadrilaterals. Suppose first this pair intersects
a common unstable leaf as in Figure \ref{figure11} (a).
Then the other corner of the union of the two adjacent
$(1,3)$ ideal quadrilaterals is still in 
$\wlu(\alpha)$ as in Figure \ref{figure11} (a). 
If there are more pairs of
adjacent quadrilaterlas, this last property keeps occurring,
until the last element in the basic block of the basic 
block of $\C$ containing with $P$ which is 
either a lozenge or a $(1,3)$ ideal quadrilateral.
If it is a lozenge, then it still 
has a corner in $\wlu(\alpha)$, and 
we obtain a contradiction as in the case of initial
element is a lozenge. 
Otherwise the last element is a $(1,3)$ ideal quadrilateral,
but the basic block contains a lozenge, denoted by $C$, which has
to be contained in the quadrant $Q'$.
To get to $C$ one may need to go past some pairs of $(1,3)$
quadrilaterals contained in $Q'$, but the important point
is that they all have a corner in $\wlu(\alpha)$ (in
the half leaf $\wlu(\alpha) \cap \partial Q'$).
Therefore $C$ has a corner, denoted by $\delta$ which is 
in this half leaf of $\wlu(\alpha)$.
Then one stable side of $C$ is contained
in $T$ and intersects an unstable leaf which
makes a perfect fit with $\wls(\alpha)$. 
Again this is impossible.

\begin{figure}[ht]
\begin{center}
\includegraphics[scale=1.00]{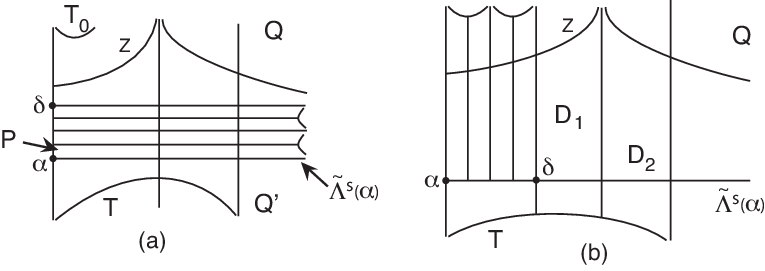}
\begin{picture}(0,0)
\end{picture}
\end{center}
\vspace{0.0cm}
\caption{Possibilities associated with Configuration 1.}
\label{figure11}
\end{figure}

The last situation is that the two adjacent $(1,3)$
ideal quadrilaterals 
(one of which is $P$) intersect a common stable leaf.
This is depicted in Figure \ref{figure11} (b). 
First notice that this implies that
the block is a stable adjacent block, and
in particular $\alpha$ is a hinge of this block.
Then we have a finite collection
of such pairs in $Q$, the corners of which are always in $\wls(\alpha)$. 
Let $\delta$ be the last such corner.
So we still have that $\wlu(\delta)$ intersects $Z$. 
The next element in the block is a lozenge, which
we denote by $C$. This in the block, which is still
contained in the quadrant $Q$.
The analysis of the lozenge case shows that this is impossible.

This finishes the analysis of Configuration 1.
\vskip .1in

We stress that we have to assume that $\alpha$ is a hinge
in Configuration 0, but $\alpha$ only needs to be a corner
in the analysis of Configuration 1.

\begin{lemma} \label{sameband}
Assume that $M$ is atoroidal.
Suppose that $W$ is a wall and $g$ is a deck transformation
such that both $W$ and $g(W)$ contain the same orbit 
$\eta$ of $\wwp$ and both $W$ and $g(W)$ enter
the same quadrant $Q$ associated with $\eta$.
Suppose that $g$ preserves the transverse orientations
to $\wls, \wlu$.
Then the subbands of $W, g(W)$ with boundary in
$\eta$ and entering $Q$ are the same.
\end{lemma} 

\begin{proof}
Let $\C$ (resp. $\E$)  be the bi-infinite block associated with
$W$ (resp. $g(W))$. 

The proof is long because there are many possibilities.

\vskip .1in
\noindent
{\bf {Case 0 $-$   There is a corner of $\C$ of which is periodic.}}

Equivalently there is a corner of $\cE$ which is periodic.
By Lemma \ref{morethanone} all corners of $\cC$ and of $\cE$
are periodic. 
To deal with this possibility we will have to use that $M$
is atoroidal.

By hypothesis $W, g(W)$ share the corner $\eta$.
Let $\beta = g^{-1}(\eta)$ which is a corner 
in $W$ as well.
Hence $\beta$ and $g(\beta)$ are corner orbits
of $W$.
Suppose for a moment that 
$g(\beta) = \beta$. Since $g$ preserves transverse 
orientations to $\wls, \wlu$, then $g$ preserves
any lozenge in a chain with a corner $\beta$. 
Therefore $g$ preserves all lozenges associated
with $W$. It follows that $\cE = g(\C) = \C$.
Since $W$ is made up of bands this
implies that $g(W)$ is equal to $W$. This
finishes the proof in this case.

Assume therefore that $g(\beta)$ is not $\beta$.
Since $\beta$ is periodic it has a non trivial
isotropy group, let $h$ be a generator of the subgroup
of this isotropy group which preserves all transverse
orientations.
Since $h$ preserves transverse orientations to $\wls, \wlu$
and $\C$ is a chain of lozenges, it follows that $h$
preserves all lozenges of $\C$ and hence all corners of $\C$.
In addition $h$ is a generator of the subgroup of
the isotropy group of 
any of these corners which preserves transverse
orientations: this is obtained by switching the corners.
Then since $g(\beta)$ is a corner of $L$, it is preserved
by $h$ and therefore:,

$$g^{-1} h g (\beta) \ = \ \beta, \ \ \ \
{\rm which \ implies \ that} \ \ \ \ g^{-1} h g \ = \
h^i$$

\noindent
for some non zero integer $i$. 
In the same way starting with the periodic orbit $g(\beta)$
which is a corner of $L$,  one
obtains that 

$$g h g^{-1} (g(\beta)) \ = \  g(\beta), \ \ 
{\rm so} \ \ g h g^{-1} \ = \ h^j$$

\noindent
for some $j$ non zero integer. These two equations imply
$h^{ij} = h$. Hence either $i = j = 1$ or $i = j = -1$.
In both cases we have that 
$g^2 h g^{-2} = h$. So in either case 
$g^2, h$ generate a $\mathbb Z^2$ subgroup of $\pi_1(M)$.
Notice that $\pi_1(M)$ is torsion free.

A $\mathbb Z^2$ subgroup of $\pi_1(M)$ contradicts 
$M$ being atoroidal.
This finishes the analysis in this case.

\vskip .1in
\noindent
{\bf {Case 1 - The elements of both $\C$ and $\E$ adjoining
$\eta$ in the quadrant $Q$ are lozenges.}}

In this case since the lozenges share a corner and intersect
in the interior, it follows that these lozenges are the
same.  Let $C$ be this lozenge.
It does not immediately follow that the subbands of $W$ and $g(W)$
are the same because one or both the basic blocks in $\C$, $\E$
could be a stable adjacent block with more than one element.
Let $\alpha$ be the other corner of $C$.

\vskip .1in
\noindent
{\bf {Case 1.1 - One of the basic blocks in $\C$ or $\E$ containing
$C$ 
is a stable adjacent block with more than
one element.}}

Without loss of generality assume that the basic block of 
$\C$ satisfies this.
In particular $\eta$ is a hinge in $W$.
Let $\cE_0$ be the basic block
of $\cE$ containing $C$, and likewise define the basic block
$\cC_0$ of $\cC$.
The analysis of this case will be split further into subcases 
which are:

$-$ the basic block $\cE_0$ in $\cE$ is a chain of at least 2 lozenges intersecting
a common stable leaf,

$-$ $\cE_0$ is is a union of one lozenge and some
$(1,3)$ quadrilaterals intersecting a common stable leaf,

$-$ $\cE_0$ is a single lozenge,

$-$ $\cE_0$ is an unstable adjacent block with more than
one element.

\vskip .05in
Suppose first that $\cE_0$
is also a stable adjacent block with more than one
element. 
Since the first element of the basic block  $\C_0$
is a lozenge with a corner $\eta$,  and the basic block 
$\cE_0$ is a stable adjacent block, then 
the other elements of $\cE_0$ are either $(1,3)$ ideal quadrilaterals 
or lozenges, all intersecting a common stable leaf.
Hence they all have a side
contained in a stable leaf non separated from
$\wls(\eta)$, and also a stable side either
contained in $\wls(\alpha)$,
or in stable leaves non separated from $\wls(\alpha)$
(case that additional elements are lozenges).
It follows that the elements are the same for
$\C$ and $\E$. 
If the blocks $\cC_0, \cE_0$ are the same, then the subbands
are the same and we finish this part of the analysis. 
If the blocks are not the same, then one of them has more 
elements. Notice that the intersection of these basic blocks is 
one of the basic blocks,
which we assume is the block in $\C$ without
loss of generality. 

\vskip .05in
To analyze this, assume 
first that the additional elements in the basic block $\cE_0$ 
are lozenges. 
Lemma \ref{morethanone} implies that all the corners
of $\C$ are periodic. This is dealt with in Case 0.
It follows that we cannot have that the additional
elements in the basic block of $\E$ are lozenges.

The other situation to be analyzed is that the
additional elements in $\E_0$ are all pairs of 
adjacent $(1,3)$ ideal quadrilaterals.
In Figure \ref{figure12} we depict
an example so the block $\C_0$ has 3 elements and the
block $\E_0$ has 5 elements. The last two elements
in $\E_0$ are adjacent $(1,3)$ ideal
quadrilaterals $C_1, C_2$. 
Let $D_1, D_2$ be the two adjacent lozenges containing
$C_1, C_2$ respectively.
Let $\tau$ be the corner of $C_1$ which is also 
a corner of the block $\C_0$. Hence $\tau$ is a hinge
of $\C$. Then we are in the situation of Configuration 1.
It was shown in the analysis of Configuration 1 that
$W$ cannot enter the quadrant of $\tau$ containing 
the pair of non separated stable leaves $Z_1, Z_2$ with
sides in $\partial(D_1 \cup D_2)$. But since
$\tau$ is a hinge of $\C$, then $W$ has to enter that
quadrant of $\tau$. This contradiction shows that this
case cannot happen.

\begin{figure}[ht]
\begin{center}
\includegraphics[scale=0.90]{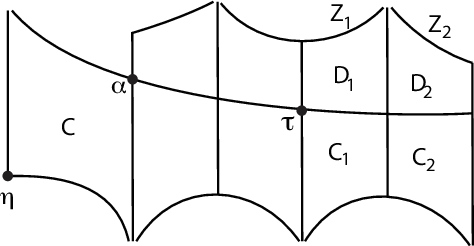}
\begin{picture}(0,0)
\end{picture}
\end{center}
\vspace{0.0cm}
\caption{Figure for Case 1.1.
$C_1, C_2$ are the first two elements in $\E$ which
are  not elements of $\C$.}
\label{figure12}
\end{figure}

We conclude that if both elements adjoining $\eta$ are lozenges
and the blocks $\C_0$ and $\E_0$ are stable adjacent blocks, then the blocks
are the same, so the subbands of $W, g(W)$ adjoining
$\eta$ are the same, and this obtains the lemma in this case.

\vskip .05in
We continue the analysis of Case 1.1. The next situation to 
consider is that the basic block $\cC_0$ 
(it contains $C$ as an element) is a stable adjacent
block, and $\E_0$ has a basic block starting with $C$ which is 
either a single lozenge 
or an unstable adjacent block with more than one
element. We analyze each in turn.

Suppose first that $\E_0$ ends with $C$,
that is, $\alpha$ is a hinge of $\E$.
If $\C_0$ has only lozenges, then
all corners are periodic and $W = g(W)$ by Case 0.
If all other elements
of the basic block $\C_0$ are $(1,3)$ ideal quadrilaterals
do the following: analyze the hinge $\alpha$ of $\E$ and the
next element of $\E$ beyond $\alpha$. The analysis of Configuration
1 produces a contradiction, because $\alpha$ is a hinge of $\E$,
and so $\E$ enters the quadrant at $\alpha$ opposite to $C$.

Finally suppose that the basic block $\E_0$ containing
$C$ does not end in $C$. Then $\alpha$ is a corner
orbit of $\E$ which is not a hinge, and the basic block $\E_0$
is an unstable adjacent block with more than
one element. 
If the additional elements in $\E_0$ are
lozenges, then all corners are periodic and we finish
by Case 0.
The other possibility is that the remaining elements
are pairs of adjacent $(1,3)$ quadrilaterals all intersecting
a common unstable leaf, which also intersects $C$.
The additional corners of the union of these pairs
are always in $\wlu(\alpha)$. Let $\delta$ be the last
such corner. Now $\delta$ is a hinge of $\E$, and we are 
in the situation just analyzed,
which produced a contradiction.

%

This finishes the proof in Case 1.1.

\vskip .1in
\noindent
{\bf {Case 1.2 - Suppose that none of the basic blocks $\cC_0,
\cE_0$ in $\C, \E$ 
with a corner in $\eta$ and containing $C$ 
is a stable adjacent block with more
than one element.}}

In this case the basic blocks are either a single lozenge or
an unstable adjacent subblock with more than one
element. In either case the subband associated with it
is the band associated with the lozenge $C$.
This is one aspect where there is an asymmetry between
the stable and unstable foliations. So the bands 
$W$ and $g(W)$ agree in this subband with boundary $\eta$.

This finishes the analysis of Case 1.

\vskip .2in
\noindent
{\bf {Case 2 - Up to switching $\cC$ and $\cE$, the element in $\C$ with
a corner $\eta$ and contained in $Q$ is a lozenge
and the corresponding element  in $\E$ is
a $(1,3)$ ideal quadrilateral.}} 

Let $C$ be the lozenge in $\C$ with a corner in $\eta$
and contained in the quadrant $Q$.
Let $E_1, E_2$ the adjacent $(1,3)$ quadrilaterals, which
are elements of $\E$ and (say) $E_1$ has a corner in $\eta$
and is contained in $Q$.
We first claim that $E_2$ is also contained in $Q$.
Otherwise $\eta$ is an interior point in the compact side
of the $(2,2)$ ideal quadrilateral which is the
union of $E_1, E_2$. But in that case the orbit
$\eta$ is not contained in the band associated
with the basic block containing $E_1, E_2$, contrary to
assumption.

It follows that the boundary of the union of $E_1, E_2$ has a corner in $\eta$
and contains one side of $C$. Suppose without loss of generality
that it is an unstable side of $C$, which is then a half leaf of
$\wlu(\eta)$. Then the stable leaf $A$ which is contained in the
boundary of $E_1 \cup E_2$ makes a perfect fit with $\wlu(\eta)$
and hence it contains a stable side of the lozenge $C$. 
But $A$ is supposed to make a perfect fit with another unstable
leaf which intersects $\wls(\eta)$ (besides $\wlu(\eta)$).
This is impossible since $C$ is a lozenge.

This shows that Case 2 cannot happen.
Finally we have:

\vskip .1in
\noindent
{\bf {Case 3 - The elements in the blocks $\C, \E$ adjoining
$\eta$ are both $(1,3)$ ideal quadrilaterals}}

The $(1,3)$ ideal quadrilaterals are part of a pair of adjacent
ones forming a $(2,2)$ ideal quadrilateral. 
Suppose first that the pairs
for both $\C$ and $\E$ intersect either a common stable leaf or
a common unstable leaf.
We initially observe that the $(2,2)$ ideal quadrilaterals for both
are exactly the same: 
this is because $\eta$ is a corner of both $W$ and $g(W)$, hence
$\eta$ has to be a corner of the $(2,2)$ ideal 
quadrilaterals $-$ as opposed to $\eta$ being
a common corner of the two $(1,3)$ quadrilaterals
whose union is the $(2,2)$ quadrilateral.
It follows that the $(2,2)$ quadrilaterals in question
are the same for $\cC$ and $\cE$.

If there is an unstable leaf intersecting
both $C_1, C_2$, then this is part of unstable
adjacent blocks in both $\cE, \cC$ and the associated
subband with this $(2,2)$ quadrilateral is the same and 
the result is proved.

Next suppose that $C_1, C_2$ intersect
a common stable leaf. Since the orbit $\eta$
is contained in $\cC, \cE$ and both blocks in $\C$ and $\E$ 
are stable adjacent blocks, 
Following along the block in $\cC$ and $\E$ we have
a number of $(2,2)$ and then a lozenge. So the only
case to analyze is what happens when we hit a lozenge in
say $\cC$, and a $(2,2)$ ideal quadrilateral in $\cE$.
This is impossible.

The remaining case is when the pair of adjacent $(1,3)$
ideal quadrilaterals for one of $\C$ or $\E$ intersects
a common stable leaf and the other intersects a common
unstable leaf.
Again without loss of generality assume that the block
in $\C$ intersects a common unstable leaf.
Hence the basic block in $\C$ is an unstable
adjacent block with more than one element, and 
likewise the basic block in $\E$ is a stable adjacent
block.
Let $C_1, C_2$ be the initial elements of the
basic block in $\C$ and $P_1, P_2$ the corresponding elements in the 
basic block in $\E$, see Figure \ref{figure13} a.
Let $U$ be the unstable leaf entirely
contained in $\partial(C_1 \cup C_2)$,
and let $Z$ be the stable leaf entirely contained in 
$\partial(E_1 \cup  E_2)$.
By the description of the stable adjacent blocks, it follows
that $U$ should intersect a stable leaf non separated from
$Z$. This is impossible, showing this case cannot happen.

This finishes the proof of Lemma \ref{sameband}.
\end{proof}

\begin{remark} \label{atoroidallittle}
In the previous result the atoroidal hypothesis is only
used when the corners of $L$ are periodic.
\end{remark}

\begin{lemma} \label{eitherperiodic}
Supppose that $W$ is a wall, and $g(W)$ is a deck translate
sharing a corner $\eta$ with $W$. 
Suppose that $g$ preserves the transversal orientations
to $\wls, \wlu$.
Suppose that the 
pair of quadrants
at $\eta$ that $W$ enters are not exactly the
same as the pair of quadrants at $\eta$ that $g(W)$
enters. Then all corners of $W$ are periodic, and so are
all corners of $g(W)$.
\end{lemma}

\begin{proof}
Let $\C, \E$ be the bi-infinite blocks associated
with $W$ and $g(W)$ respectively.
Consider the quadrants of $\eta$. There are two possibilities:
either $W$ and $g(W)$ enter a common quadrant or not.

\begin{figure}[ht]
\begin{center}
\includegraphics[scale=1.00]{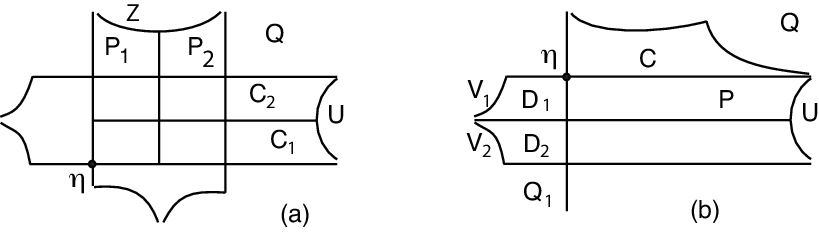}
\begin{picture}(0,0)
\end{picture}
\end{center}
\vspace{0.0cm}
\caption{a. Figure for Case 3 of Lemma \ref{sameband},
b. Figure for Case 1.1 of Lemma \ref{eitherperiodic}.}
\label{figure13}
\end{figure}

Notice that we do not know a priori whether $\eta$ is a
hinge of $W$ or not. 

\vskip .1in
\noindent
{\bf {Case 1 - 
Suppose that $W$ and $g(W)$ enter the same quadrant $Q$
of $\eta$.}}

By the previous lemma, using also Remark \ref{atoroidallittle},
either the corners of $W$ and
$g(W)$ are periodic $-$ in which case the lemma is proved;
or  the bands $W$ and $g(W)$ agree
on the subband with a boundary in $\eta$ and entering $Q$.
So we assume the second option occurs and we analyze that.
There is an element $C$ of $\C$ with a corner $\eta$ and
contained in $Q$. The element $C$ is also an element of $\E$
because the bands agree.

In particular since the quadrants of $\eta$ that $W$ and $g(W)$
are not the same, it follows that one of $\C$ or $\E$
has adjacent elements with common corner $\eta$.
Without loss of generality assume that $W$ satisfies this 
property ($g(W)$ may also satisfy this property with
adjoining element adjacent along another side of $C$ 
or $\eta$ may be a hinge of $g(W)$).

\vskip .1in
\noindent
{\bf {Case 1.1 - $C$ is a lozenge.}}

Suppose first that the adjoining element $P$ of $\C$ is adjacent to
$C$ along $\wlu(\eta)$. Then $P, C$ are part of the same
adjacent block of $\C$ which is a stable adjacent basic block.
Then the element $P$ is forced to be a $(1,3)$ ideal
quadrilateral. In this case the band associated with this
subblock has boundary orbits one which is the other corner
of $C$ and one which is on the component of $\mt - \wlu(\eta)$
not containing $C$. Since the basic block is a stable
adjacent block, this would imply that
$\eta$ is not a corner orbit of $L$,
contradiction.

Suppose on the other hand that the adjoining element $P$ in $\C$ is
adjacent to $C$ along $\wls(\eta)$. Then the adjacent basic block 
of $\C$ containing $C$ and $P$ is an unstable adjacent block. 
By Lemma \ref{morethanone} we can assume that $P$ is an
$(1,3)$ ideal quadrilateral.
The two adjacent $(1,3)$ ideal quadrilaterals of $\cC$ (one of
which is $P$) are contained
in adjacent lozenges $D_1, D_2$. The leaf $\wlu(\eta)$ intersects
the interior of $D_1$ and $D_2$. There is an unstable leaf $U$
entirely contained in $\partial(D_1 \cup D_2)$. Let $V_1, V_2$
be the other unstable leaves containing half leaves in the
boundary of $D_1, D_2$ respectively. Then $V_1$ intersects
$\wls(\eta)$. Let $Q_1$ be the quadrant at $\eta$ that contains
$V_2$, see Figure \ref{figure13} (b). Consider now the other element $Y$ of 
$\E$ with a corner in $\eta$ as well. If $Y$ were adjacent
to the lozenge $C$ along $\wlu(\eta)$, then the basic block
in $\E$ would be a stable adjacent block and $\eta$ would not
be contained in the wall $g(W)$, contradiction to
assumption.
It follows that the element $Y$ of $\E$ is contained in the 
quadrant $Q_1$. We are now exactly in the situation described
in Configuration 1.
Here the quadrant
in question is $Q_1$, the lozenges are $D_1, D_2$.
The analysis of Configuration 1 shows that this is 
impossible, so this cannot occur.
This finishes the case that $C$ is a lozenge.

\vskip .1in
\noindent
{\bf {Case 1.2 - $C$ is a $(1,3)$ ideal quadrilateral.}}

\begin{figure}[ht]
\begin{center}
\includegraphics[scale=1.20]{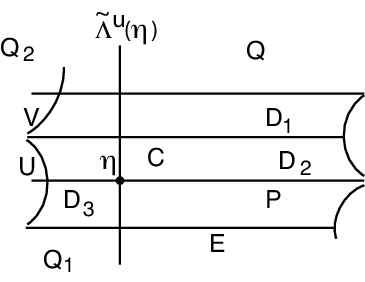}
\begin{picture}(0,0)
\end{picture}
\end{center}
\vspace{0.0cm}
\caption{Figure for Case 1.2.}
\label{figure14}
\end{figure}

Let $P$ be the element of $\C$ which is adjacent to $C$ 
with common corner $\eta$. 
If $P$ is adjacent to $C$ along $\wlu(\eta)$ then 
$C,P$  are part of a stable adjacent block. As in Case 1.1, \ 
$\eta$ would not be a corner orbit of $W$, contradiction.
It follows that $C,P$ are adjacent along $\wls(\eta)$
and they are part of an unstable adjacent basic block
which is contained in $\C$.
Then $C$ is part of a lozenge $D_2$, and there is another
lozenge $D_1$ adjacent to $D_2$ along a stable side,
and so that $D_2$ separates $D_1$ from
$\wls(\eta)$, see Figure \ref{figure14}.
This is because
$C$ is part of a pair of adjancent ideal quadrilaterals, and
$\eta$ is a corner orbit of the union.
In addition $P$ is contained in another lozenge
$D_3$ which is adjacent to $D_2$ along $\wls(\eta)$. 
Let $U$ be the unstable leaf entirely contained
in $\partial(D_2 \cup D_3)$ and let $V$ be the unstable
leaf with a half leaf in $\partial D_1$ and not contained
in $Q$. 
Let $E \not = \wls(\eta)$ be the other stable leaf which 
contains a side of $D_3$.
Let $Q_2$ be the quadrant at $\eta$ containing 
$V$. Let $Q_1$ be the quadrant at $\eta$ opposite to $Q$.
Consider the other element of $\E$ with corner $\eta$:
if this other element is in the quadrant $Q_2$, then we
are exactly in the situation of Configuration 1, where
the quadrant is $Q_2$ and the non separated leaves
are $U, V$. 
This is impossible by the analysis of Configuration 1.
If the other element of $\E$ with corner $\eta$ is
in the quadrant $Q_1$, then $\eta$ is a hinge of
$\E$. Here we are in the situation analyzed in
Configuration 0, where we use the leaves $U, E$ which make
a perfect fit and intersect $\wls(\eta), \wlu(\eta)$ respectively.
The analysis of Configuration 0 shows that this is impossible.

This finishes the analysis of Case 1.
Notice that there was one situation possible in this case,
which was that the adjoining elements to $\eta$ in $\C$ are
both lozenges and all corners of $L$ are periodic. This
is the conclusion we wanted.

\vskip .1in
\noindent
{\bf {Case 2 - $W$ and $g(W)$ do non enter the same
quadrant of $\eta$.}}

Either $W$ (and then also $g(W)$) enters adjacent quadrants
of $\eta$ or $W$ (and hence also $g(W)$) enters opposite
quadrants of $\eta$.

Suppose first that $W$ enters adjacent quadrants of $\eta$.
We initially consider the case  that
the elements in $\C$ with corner $\eta$ are adjacent
along $\wls(\eta)$. Then the block containing them  is
an unstable adjacent block. If all elements
are lozenges, then as in 
Case 1 this leads to all corners of $W$
being periodic and this completes the proof in this case.
If at least one of them is part of a pair of 
adjacent $(1,3)$ ideal quadrilaterals, then they 
are contained in adjacent lozenges $D_1, D_2$. 
Let $Q$ be the quadrant at $\eta$ which intersects
two non separated stable leaves
which contain sides in $D_1, D_2$
respectively. This quadrant $Q$ intersects $g(W)$ by
hypothesis, and contains $U_2$ and a half leaf of 
$U_1$ ($U_1, U_2$ are the unstable leaves with half leaves
in $\partial D_1, \partial D_2$ respectively).
Now we are in the situation of Configuration
1, whose analysis shows this is impossible.

If the elements in $\C$ with corner $\eta$ are adjacent along
$\wlu(\eta)$ then the associated block in $\C$ is a stable
adjacent block. 
Then $\eta$ is not a corner orbit of $\C$, contradiction. 

Finally suppose that $W$ (and hence $g(W)$) do not
enter adjacent quadrants at $\eta$. It follows that
$\eta$ is a hinge of both $W$ and $g(W)$.
We will show this leads to a contradiction.
Choose a transverse orientation to $\wls$. 
Now put a transverse orientation to $\wlu$ so that
when one crosses a hinge in $W$ in the positive transverse
orientation to $\wls$, then one also crosses the hinge
in the positive transverse orientation to $\wlu$.
This is possible because they are hinges.

Now consider the situation from the point of view of $g(W)$
when one crosses $\eta$ in the positive transverse orientation
to $\wls$. Since $\C, \E$ intersect disjoint quadrants of $\eta$
it follows that $g(W)$ crosses $\eta$ going in the negative
transverse orientation to $\wlu$. 
Now we use that $g$ preserves
transversal orientations to $\wls, \wlu$.
Translating back by $g^{-1}$ there is a hinge in
$W$ so that crossing it in the positive transverse orientation
to $\wls$ is crossing in the negative transverse orientation
to $\wlu$. This is contradiction. 
This shows this last case cannot happen.

This finishes the proof of lemma \ref{eitherperiodic}.
\end{proof}

Now we use the previous lemmas to obtain an important conclusion
when $M$ is atoroidal.

\begin{proposition} \label{identical}
Suppose that $W$ is a wall. Suppose that $M$ is atoroidal.
Let $g$ be a non trivial deck transformation so that 
$g(W)$ and $W$ intersect, but not transversely.
Suppose that $g$ preserves transversal orientations to
$\wls, \wlu$.
Then $g(W) = W$. If a corner of $W$ is periodic,
then $g$ preserves
all corners of $W$.
\end{proposition}

\begin{proof}
The proof is by contradiction. Let  $W, g(W)$ which
intersect, but not transversely. 
We assume that $W$ is not equal to $g(W)$
and we obtain a contradiction.

By Lemma \ref{commoncorner} $W, g(W)$ share a corner $\eta_0$.
By Lemma \ref{sameband} 
if $W, g(W)$ enter the same quadrant $Q$ of 
$\eta_0$, 
then $W, g(W)$ have identical subbands
with a boundary in $\eta_0$ and entering $Q$.
Let $\eta_1$ be the other boundary component of this
subband, which is a corner of both $W, g(W)$.
This can be iterated, and in the other direction too.
Since $W, g(W)$ are distinct, one eventually
reaches a common corner $\eta_2$ so that $W, g(W)$
do not enter exactly the same pair of quadrants of $\eta_2$.
By the previous lemma, this shows that all corners of 
$W$ are periodic and all elements of $\C$ are lozenges.
So in any case all corners of $W$ and $g(W)$ are
periodic.

By hypothesis $W, g(W)$ share a corner.
Now use the argument contained in the analysis 
of Case 0 of 
Lemma \ref{sameband}: that analysis produces
a $\mathbb{Z}^2$ subgroup of $\pi_1(M)$,
contradicting that $M$ is atoroidal.

We already proved that $g(W) = W$.
Finally if a corner of $\C$ is periodic, then again
the analysis in Case 0 of 
Lemma \ref{sameband} implies that $g$ preserves
all corners of $\C$.
This finishes the proof of the proposition.
\end{proof}

\section{The case with no transverse intersections}
\label{notransverse}

We now start the actual proof of the Main Theorem.
In this section we consider the case that there is
a wall $W$ so that no deck translate of $W$ intersects
$W$ transversely.
We are only considering a wall $W$ that is obtained as limits
of partial walls associated with strings of lozenges. 
In Section \ref{construction} we explained that there are
such constructions.
With this hypothesis, we know from
Proposition \ref{identical} that since $M$
is atoroidal, then if $g(W)$ is
a deck translate intersecting $W$ then $g(W) = W$.
As a consequence we have the following result.
Suppose that a $2$-dimensional foliation $\cG$ of $M$ with
hyperbolic leaves is fixed.
A {\em leafwise geodesic lamination} is a $2$-dimensional
lamination in $M$, which is 
transverse to $\cG$,  and intersects the leaves of $\cG$
in geodesics. In this article we only consider the foliation
$\cG = \ls$ for leafwise geodesic laminations.
Then we have the following result:

\begin{proposition} \label{prop-gettinglamination}
Suppose that $M$ is atoroidal and $\ls, \lu$ are
transversely orientable.
Suppose that $\ls$ has hyperbolic leaves.
Let $W$ be a wall 
obtained as a limit of partial
walls associated to strings of lozenges. Suppose that
that $\pi(W)$ has
no transverse self intersections, then the closure of $\pi(W)$
is a leafwise geodesic lamination in $M$.
All leaves of the lifted lamination are obtained as limits
of partial walls associated with strings of lozenges.
\end{proposition}

\begin{proof}
Following the notation of Section \ref{construction} 
there are partial walls 
$$O_i \ \ = \ \ \cup \ \{ B_i, \  -n+1 \leq j \leq i \}$$
associated to strings of lozenges, so that $O_i$ converges to $W$.
We explained in previous sections that this is equivalent
to the hinges of $O_i$ converge to the hinges of $W$.
The closure of $\pi(W)$ lifts to the closure of the union
of deck translates of $W$ in $\mt$.

Let $p_n \in g_n(W)$ with $p_n$ converging to $p$,
here $g_n \in \pi_1(M)$. We will show that $g_n(W)$ converges
to a wall $H$, which is also obtained as the limit of the 
process of Section \ref{construction}.

Let $q_n = g_n^{-1}(p_n)$ which are in $W$. 
Since $O_i$ converges to $W$ as $i \to \infty$,
then for each $n > 0$, and
for each $2n$ blocks of $W$ around $q_n$ we can 
find partial walls in $O_i$  (for $i >> n$) 
which are less than $1/n$ from 
these blocks in bigger and bigger sets.
The images of these by $g_n$ have points
closer and closer to $p$. By the diagonal process explained
in Section \ref{construction} one can extract a subsequence
of partial walls that converges, and to a wall, which we
denote by $H$.
In particular $H$ is also obtained as a limit of partial
walls associated with string of lozenges.

We claim that in fact $H$ is the limit of $g_n(W)$ $-$
no need of subsequences.
Since $p_n$ converges to $p$, then up to a diagonal 
process we get that a subsequence of $g_n(W)$ converges, let $H'$
be the limit.
As explained above, the limit is a wall, also obtained by the
limiting process as in Section 4, so by the same argument 
as above we get that $H' = H$. Since any subsequence of $g_n(W)$
has a subsequence which converges to $H$, it now follows that
$g_n(W)$ converges to $H$.

Any limits will not have pairwise transverse intersection.
This proves that the projection to $M$ is a lamination.
It is obviously a leafwise geodesic lamination.
\end{proof}

For future reference we explicitly state the following result
which also works when $\pi(W)$ has transverse self intersections:

\begin{corollary} Let $W$ be a wall that is obtained as a
limit of partial walls associated with strings of lozenges.
Then any limit of $g_n(W)$ where $g_n$ are in $\pi_1(M)$ is 
also a wall which is a limit of partial walls associated
with strings of lozenges.
\end{corollary}

We fix such a wall $W$ and let $\cW$ be the leafwise
geodesic lamination in $M$ obtained by completing $\pi(L)$.
We have the following extremely important property.

\begin{proposition} \label{asymptotic1}
There is $\eps > 0$ so that if $E$ is a leaf of
$\wls$, $r_1, r_2$ are boundary leaves of a complementary
region of the lamination $\wl \cap E$, and there are 
$x_i \in r_i$ with $d_E(x_1,x_2) < \eps$, then 
the geodesics $r_1, r_2$ in $E$ are asymptotic in $E$.
\end{proposition}

\begin{proof}
This will use properties of the lamination $\lam$ and of
Anosov flows.

Suppose then that there are leaves
$E_i$ of $\wls$ and boundary leaves $\ell_i, \alpha_i$ of complementary
regions of $\wl \cap E_i$ with points that are less than
$\eps_i$ apart in $E_i$, with $\eps_i \rightarrow 0$,
as $i \to \infty$.
We want to show that for $i$ big enough, 
$\ell_i, \alpha_i$ are asymptotic in $E_i$.
There is a constant $a_1 > 0$ satisfying the following:
Choose $p_i, q_i$ in $\ell_i, \alpha_i$ respectively which are at
a distance $1$ from each other in $E_i$ and so
that 

$$d_{E_i}(p_i,\alpha_i) \  > \  a_1, \ \ \ 
d_{E_i}(q_i,\ell_i) \  > \  a_1.$$

\noindent
The distance of $1$ is in the leaf $E_i$, but the two
distances above are from a point to the other geodesic.
The reason is as follows: any two disjoint geodesics in the 
hyperbolic plane diverge from each other unboundedly in at
least one direction. Since the geodesics in question have points
which are $<< 1$ from each other, start at these points
and move along both geodesics in a direction where they 
diverge from each other, until hitting distance $1$ between
the geodesics.

 Choose rays $r_i$
in $\ell_i$, $s_i$ in $\alpha_i$ starting in $p_i, q_i$ respectively
and which get $\eps_i$ close to each other in $E_i$ at some
point. Using that $M$ is compact, then 
 up to a subsequence and deck transformations assume that 
$p_i, q_i$ converge respectively to points $p, q$ both
in a  leaf of $\wls$ denoted by $E$.
Let $W_i, H_i$ be the leaves of $\wl$ containing $p_i, q_i$
respectively.
These are walls.

Since $\lam$ is a lamination, then $p, q$ are in leaves
of $\wl$, so let $W, H$ be the leaves of $\wl$ containing
$p, q$ respectively. In particular $W, H$ are walls.
Let $\ell, \alpha$ be the intersections $W \cap E, H \cap E$
respectively.
First of all notice that $\ell, \alpha$ are not the same 
since $d_{E_i}(p_i,\alpha_i) > a_4$ for some fixed
$a_4 > 0$, and vice versa for $q_i$.
In addition $\ell_i, \alpha_i$ have points far away
from $p_i, q_i$ respectively, which are $\eps_i$ close to each
other with $\eps_i \to 0$.

\begin{claim}
For big enough $i$, we have that $\ell_i \subset W,
\alpha_i \subset H$.
\end{claim}

\begin{proof}
Choose laminated boxes (that is lamination charts) $B_0, B_1$ of $\wl$
near $p, q$ respectively. The leaf of $\wl$ through $p_i$
intersects $B_0$ for $i$ big enough and the intersection
is a disk, and similarly for $q_i$ in $B_1$. Given any such
$i$, the intersections of these disks with $E_i$ are in
the boundary of a complementary region of $E_i \cap \wl$.
Up to subsequence it follows
that for all leaves $F$ of $\wls$ intersecting $B_0, B_1$, 
the intersections $F \cap W_i, F \cap H_i$ are boundary
leaves of a fixed  complementary region of $\wl \cap F$.
Hence one obtains that after removing finitely many 
elements, then
$p_i, p_j$ are in the same local leaf of $\wl$.
Similarly for $q_i, q_j$. 
The
claim follows.
\end{proof}


\begin{claim}
The leaves $\ell, \alpha$ of $\wl \cap E$
are asymptotic.
\end{claim}

\begin{proof}
This is fairly simple. If $\ell, \alpha$ are not asymptotic, then 
they have a common perpendicular in $E$: the distance between points
in $\ell, \alpha$ decreases until both points are in the common 
perpendicular then it increases on the other side of these
perpendicular points. So from $p, q$ and in a given
direction, the distance between points in $\ell, \alpha$ either
always increases, or decreases for a fixed finite length
then increases monotonically.
Let $a_2$ be the minimum of the distance between points
in $\ell, \alpha$. Then for big enough $i$ one sees the same
behavior between $\ell_i, \alpha_i$: the distance between decreases
to close to $a_2$ then increases. This means that the
distance cannot get close to $\eps_i$ if $\eps_i$ is 
smaller than $a_2/2$ and $i$ is big enough. 
This contradiction proves the claim.
\end{proof}

\noindent
{\bf {Continuation of the proof of Proposition \ref{asymptotic1}.}}

Recall that $W, H$ are walls and $E$ is a leaf of $\wls$
and $\ell = W \cap E, \ \alpha = H \cap E$.
We will now prove that $\ell_i, \alpha_i$ are asymptotic in $E_i$ 
for $i$ big enough.

Consider first the case that the common ideal point of $\ell, \alpha$
in $S^1(E)$ is not the common forward ideal point of flow lines in $E$.
Then it is a negative ideal point in $S^1(E)$ and for
$i$ big enough, $E_i$ is asymptotic to $E$ in that direction
of $E$.
This follows because transversely to $\wls$ there is the 
unstable foliation $\wlu$.
Then $E_i$ is asymptotic to $E$ in this direction, and
there is a curve in $E_i$ asymptotic to $E$ in this direction.
But $E_i \cap W = \ell_i, \ E \cap W = \ell$.
In particular $\ell_i$
is asymptotic in $\mt$ to $\ell$, and likewise
$\alpha_i$ is asymptotic to $\alpha$. Since $\ell, \alpha$ are asymptotic
in this direction,
then so are $\ell_i, \alpha_i$. Using the local product structure
of foliations (applied to $\wls$) it follows that $\ell_i, \alpha_i$
are asymptotic in $E_i$ as well, as we wanted to prove.

\vskip .1in
The other case is that the common ideal point of $\ell, \alpha$ in 
$S^1(E)$ is the common forward ideal point of all flow lines
of $\wwp$ in $E$. Let $x$  be this point.
This case is much more complex.
In this case the rays of $\ell_i$ are not asymptotic to the ray
of $\ell$ in the $x$ direction, 
since in the positive flow direction of $\ell$ the nearby
leaves of $\wls$ expand away from $E$.
Since $\ell, \alpha$ have ideal point $x$, it follows that $\ell, \alpha$ 
are flow lines of $\wwp$ contained in $E$. In other words
$\ell, \alpha$ are corner orbits of the walls $W, H$.
Up to a subsequence assume every $E_i$ is in the same component of
$\mt - E$.

Let $\C, \E$ be the bi-infinite blocks associated with the
walls $W, H$.
Let $C, C^*$ be the corresponding elements of $\C, \E$
which have a corner in $\ell, \alpha$ respectively
and which intersect $E_i$.
We now prove a simple result that will be very useful for us.

\begin{lemma} \label{boundary}
Let $H$ be a wall and let $\delta$ be a leaf of the
foliation $\fol$ induced by $H \cap \wls$ in $H$, and
so that $\delta$ is not a flow line. 
Let $\HH$ be the bi-infinite block associated with $H$.
Let $E$ in $\wls$
so that $\delta = H \cap E$.  Let $U, V$ in $\wlu$ so that
the ideal points of $\delta$ in $S^1(E)$ are the negative
ideal points of $U \cap E$ and $V \cap E$. 
Let $L$ in $\wls$ also intersecting $H$. Let $I$ be the interval
of leaves of $\wls$ intersecting $H$ between 
$L$ and $E$, including $L, E$. Suppose that for any
leaf $F$ in $I$ then the ideal points  of $F \cap H$ are
both negative ideal points in $S^1(F)$. Then for any
$F$ in $I$ the ideal points of $F \cap H$ are the negative
ideal points of $F \cap U$ and $F \cap V$.
In other words the leaves $U, V$ are constant until
$H$ hits a corner orbit, or equivalently  until the
first time $H$ does not intersect one of $U$ or $V$.
Finally the leaves $U, V$ are constant for any $F$ in an
element of $\HH$.
\end{lemma}

\begin{proof}
The leaf $\delta$ is in a basic block of $\HH$, which we denote by
$\cB$. If the basic block is a lozenge, there are unique
unstable leaves $U, V$ containing half leaves in the boundary
of $\cB$. The corners of the lozenge are leaves of $H \cap \wls$
which are tangent to the flow. Hence any leaf $L$ as above
has to intersect the lozenge. It follows that 
for any $F$ between $E$ and $L$, the ideal points
of $Z \cap H$ are the negative ideal points of $U \cap Z$ and
$V \cap Z$ as desired. 
For any $T \in \wls$ intersecting $H$ outside the lozenge, it does not
intersect at least one of $U, V$, and if it is a leaf
containing a side of the lozenge, it contains a corner
of the lozenge, which is say in $U$ and then 
$T$ does not intersect $V$.

If the basic block $\cB$ is a stable adjancent block, then again there are
only two unstable leaves $U, V$ with half leaves contained in
$\partial \cB$. The corners of $\cB$ are contained in $H$, 
a similar analysis as in the lozenge case yields the result.

Finally suppose the block $\cB$ is an unstable adjacent block. 
Let $P_1$ be the element of $P$ containing $\delta$.
The element $P_1$ could be either a lozenge or one of a pair
of stable adjacent $(1,3)$ ideal quadrilaterals. If it is
a lozenge, the corners of the lozenge are contained in $H$
and the result follows as above. Otherwise let $P_2$ be the adjacent
$(1,3)$ ideal quadrilateral forming a pair with $P_1$,
and so that there 
is a unique unstable leaf $U$ entirely contained in
$\partial (P_1 \cup P_2)$. In addition there is a unique
unstable leaf $V$ so that $V \cap 
\partial (P_1 \cup P_2)$ is a non empty flow band 
(over a bounded interval) in $V$.
The two corners of $P_1 \cup P_2$ are orbits contained in $H$.
For any $F \in \wls$ intersecting $P_1 \cup P_2$, including
the stable leaf with a flow band which is a side of both
$P_1$ and $P_2$ (or equivalently the stable leaf
separating $P_1$ from $P_2$),
then $F$ intersects $H$ and
the ideal points of $Z \cap H$ are 
the negative ideal points of $H \cap U$ and $H \cap V$.
At the corners of $P_1 \cup P_2$, the wall $H$ intersects $V$
but not $U$ and any other stable leaf does not intersect
at least one of $U$ or $V$,
proving the result in this case as well.
\end{proof}

\noindent
{\bf {Continuation of the proof of Proposition \ref{asymptotic1}.}}

By the lemma, there are 
unstable leaves $U_0, U_1$ so that 
the leaves of $W \cap L$ for any stable leaf $L$ intersecting the
element $C$ have ideal points which are the negative 
ideal points of $U_0 \cap L$ and $U_1 \cap L$.
Since $W \cap E = \ell$ is a flow line of $\wwp$ it follows
that up to switching $U_0$ with $U_1$ we have that  $U_0 = \wlu(\ell)$.
Then it follows 
that $U_1$ does not intersect $E$,
but for any $L$ intersecting
$C$, then  $L$ intersets $U_1$. It follows that $U_1$ makes a perfect
fit with a stable leaf non separated from $E$.
There is then a unique unstable leaf $U_2$ which makes a perfect
fit with $E$ and is either equal to $U_1$ or separates $E$
from $U_1$. Recall that $E = \wls(\ell) = \wls(\alpha)$.

In the same way there are unstable leaves $V_0, V_1$
satisfying the following:
the ideal points of $H \cap L$ with $L$ leaf of $\wls$ intersecting
the element
$C^*$, are given by the negative ideal points of 
$V_0 \cap L$ and $V_1 \cap L$. 
Without loss of generality
assume that $V_0 = \wlu(\alpha)$. Then as in the case of $W$,
it follows that $V_1$ makes a perfect fit with a stable
leaf non separated from $E$.
Let $V_2$ making a perfect fit with $E$ and either equal to
$V_1$ or separating it from $E$.

By assumption $\ell, \alpha$ are distinct flow lines
in $E$, hence $U_0, V_0$ are distinct leaves.
If $U_1 = V_1$ then the intersections of $W, H$ with
stable $L$ intersecting $C$ (or $D$)  near $E$ are 
asymptotic, and so $\ell_i, \alpha_i$ 
are asymptotic for $i$ big enough, as we wanted to prove.
So we assume that $U_1, V_1$ are distinct unstable leaves.
There are two possibilities here:

\vskip .05in
\noindent
{\bf {Situation 1 $-$ 
None of $U_1, V_1$ separates the other
from $E$.}}

 This is equivalent to $U_1, U_0, V_0, V_1$ all 
intersecting a stable leaf $L$ ($L$ intersects $C$) and
so that in $L$ the flow lines intersection of the 
above leaves with $L$  are linearly ordered, 
in the order
above $U_1, U_0, V_0, V_1$ $-$ this is because
the bands $W, H$ do not intersect.
In this case the leaves $U_2, V_2$ are distinct leaves,
contained in the same component of $\mt - E$, and both 
making a perfect fit with $E$. Then 
Theorem \ref{periodicdouble} implies
that $E, U_2, V_2$ are all periodic, and left invariant by
the same non trivial deck transformation $g$.
Let $\delta$ be the periodic orbit in $E$.
We assume that $g$ acts on $\delta$ moving points
backwards. 
Then $U_2, V_2$ are connected by a pair of adjacent lozenges,
denoted by $D_1, D_2$. 
The leaf $E$ is entirely contained in $\partial(D_1 \cup D_2)$.
Let $S,S'$ be the other stable leaves which have half
leaves in the boundary of $D_1, D_2$.

Since $\ell, \alpha$ are distinct
orbits in $E$, at least one of them is not the periodic orbit
in $E$. Without loss of generality assume that $\ell$ is such
an 
orbit. We depict this in Figure \ref{figure15} (a).
Up to switching $S, S'$ assume that $\wlu(\ell)$ ($= U_0$)
 intersects $S$.
We consider the two quadrants at $\ell$ which contain
$U_2, V_2$: let $Q$ be the one that contains 
$\delta$ and let $Q'$ be the other quadrant that contains
one of $U_2, V_2$. Suppose that $W$ enters the quadrant
$Q$, in which case $U_2$ intersects $S'$ $-$ notice this
is not depicted in Figure \ref{figure15} (a).
Then we are in the situation of Configuration 1,
which showed that this is impossible.
Since $W$ enters $Q \cup Q'$ as it intersects $E_i$,
 which in turn
intersect $D_1 \cup D_2$, it follows that $W$ enters
$Q'$. 
This is what is depicted in Figure \ref{figure15} (a). 
Here $U_2$ is contained in $Q'$.

\begin{figure}[ht]
\begin{center}
\includegraphics[scale=1.00]{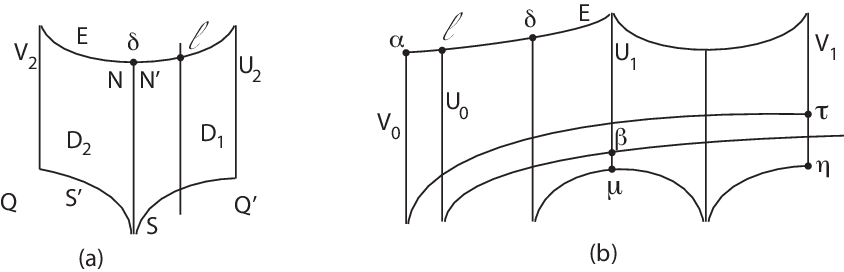}
\begin{picture}(0,0)
\end{picture}
\end{center}
\vspace{0.0cm}
\caption{a. Figure for Situation 1, b. Figure for 
Situation 2.}
\label{figure15}
\end{figure}

Let $A$ be the component of $\mt - E$ which contains $U_2, V_2$.
Let $N, N'$ be the quadrants of $\delta$ contained in $A$
and so that $Q$ contains the quadrant $N$ at $\delta$.
Notice that $N, N'$ are quadrants at $\delta$ and
$Q, Q'$ are quadrants at $\ell$. In particular
$Q' \subset N'$ and not equal, and similarly $N \subset Q$ and
not equal.
Now iterate $W$ by $g^n, \ n \to \infty$.
Then $g^n(\ell)$ converges to $\delta$, so 
$g^n(W)$ converges to a wall $W_0$ which contains $\delta$
and whose intersection with the component $A$ of $\mt - E$
is contained in $N'$.

Now do the same with the orbit $\alpha$ which is the 
intersection of $H$ with $E$. 
Using the setup of Configuration 1, we obtain that
$\wlu(\alpha)$ cannot intersect $S$. So either
$\alpha = \delta$ or $\wlu(\alpha)$ intersects
$S'$. In any case $H \cap A$ is contained
in the quadrant $N$ at $\delta$. Again apply $g^n$:
$g^n(\alpha)$ converges to $\delta$ as $n \to \infty$, and 
$g^n(H)$ converges to a wall $H_0$ containing
$\delta$ and so that $H_0 \cap A \subset N$.

We conclude that $W_0, H_0$ both contain $\delta$ but are
distinct. This is impossible, by Lemma \ref{identical}.
This shows that Situation 1 cannot happen.

\vskip .05in
\noindent
{\bf {Situation 2 $-$
One of $U_1, V_1$ separates the other from $E$.}}

In particular $U_2 = V_2$.
Without loss of generality we assume that $U_1$ separates
$E$ from $V_1$. We refer to Figure \ref{figure15}  (b).
This shows that the basic block in $\E$ containing $C^*$ is a
stable adjacent block. 
The orbit $\alpha$ is a corner orbit of this block,
let $\tau$ be the other corner orbit.
Notice that $\tau$ is contained in $V_1$. 

In the same way $\ell$ is a corner orbit of the block
in $\C$ containing $C$. 
For simplicity in the figure
we denote the case that $U_1$ makes a perfect fit with $E$.
Let $\delta$ be the periodic orbit in $E$, 
$\eta$ the periodic orbit in $V_1$,
and $\mu$ the periodic orbit in $U_1$.
As in the previous case let $g$ be a generator
of the isotropy group of $E,U_1,V_1$, and assume
as in Situation 1 that $g$
acts pushing points backwards along $\delta$.

We first consider the iterates $g^n(H)$ of the wall 
$H$, and let $n \to \infty$.
Then as in the previous case $g^n(\alpha)$ converges to
$\delta$ as $n \to \infty$. Hence $g^n(H)$
converges to a wall $H_0$ which contains $\delta$.
Notice that $g$ acts in $\eta$ pushing points
forward: this is because $\pi(\delta)$ is freely
homotopic to the inverse of $\pi(\eta)$. Since
$V_1$ is an unstable leaf, then $g$ acts
on the orbit space of $V_1$ as a contraction.
Hence $g^n(\tau)$ converges to $\eta$.
The images under $g^n$ of the band in $H$ from
$\alpha$ to $\tau$ are all isometric, hence they converge
to the band with corners $\delta, \eta$.
It follows that $H_0$ is a wall containing
both $\delta$ and $\eta$.

On the other hand applying $g^n$ to $W$ one obtains:
$g^n(\ell)$ also converges to $\delta$ as $n \to \infty$,
and $g^n(W)$ converges to a wall $W_0$ containing $\delta$.
In addition either $W$ does not intersect $U_1$ or 
intersects it in an orbit $\beta$. 
Hence the limit $W_0$ either does not intersect $U_1$
or intersects it in the orbit $\mu$.
Regardless $W_0$ is distinct from $H_0$, but both are walls
containing $\delta$. Again this is a contradiction
to Proposition \ref{identical}.
This shows that Situation 2 cannot happen either.

\vskip .05in
This contradiction shows that the assumption $U_1, V_1$
are distinct is impossible. Hence $U_1 = V_1$ and
it follows that $\ell_i, g_i$ have
to be asymptotic in $E_i$ for $i$ big enough.
This finishes the proof of Proposition \ref{asymptotic1}.
\end{proof}

\subsection{Ruling out the case that $\lam$ is a foliation}

Conceivably, $\lam$ could be a foliation, we deal with this
possibility first.

\begin{proposition} $\lam$ cannot be a foliation.
\end{proposition}

\begin{proof}
Suppose by way of contradiction that $\lam$ is a foliation.
Then $\lam$ is a foliation transverse to $\ls$ and intersecting
its leaves
in geodesics. This was analyzed in detail by Calegari
in \cite{Cal1}, section 4. He proved that there is what is called
a {\em funnel leaf} $E$ in $\wls$. This means that
all the leaves of $\wl \cap E$ (which are all geodesics
in $E$) have a common ideal point in $S^1(E)$, called
the funnel point of $\wl \cap E$.

We now have two funnel foliations in $E$: one which is
given by leaves of $\wl \cap E$ and
is denoted by $\fol_{\lam}$  and one which is given
by the flow lines of $\wwp$ in $E$ denoted by $\fol_{\wwp}$.
Both of these foliation are by geodesics.

Let $x_1$ be the funnel point of $\fol_{\lam}$ and $x_2$ 
be the funnel point of $\fol_{\wwp}$ in $E$.
They may or may not be the
same point (we want to show they are the same point).
Let now $x$ be a point in $S^1(E)$ distinct from both
$x_1, x_2$ and $I$ an open interval in $S^1(E)$ containing
$x$ and whose closure is disjoint from $x_1, x_2$. 
The geodesics with one ideal point in $x_1$ and another
in $y \in I$ are asymptotic in the $y$ direction
to the geodesics with one ideal point in $x_2$ and
the other in $y$. The first is a leaf of $\fol_{\lam}$ and
the second is a leaf of $\fol_{\wwp}$. This happens
not at a single point $y$ but at every point $y$ in $I$.
It follows that there are disks $D_i$
in $E$
of radii converging to infinity where the directions
of $\fol_{\lam}$ and $\fol_{\wwp}$ in $D_i$ are less than $1/i$ from
each other throughout $D_i$.
Up to a subsequence and deck
translations we obtain a leaf $F$ of $\wls$ so that
the foliations $\wl \cap F$ and $\wwp|_{F}$ are the
same. 

Since $M$ is atoroidal, $\Phi$ is transitive.
It now follows that for every leaf $L$ of $\wls$, the
foliations $\wl \cap L$ and $\wwp|_{L}$ are the same.
In particular this shows that if $W$ is a leaf of $\wl$,
then the foliation induced by $\wls \cap W$ has all leaves
which are flow lines.

This is a contradiction to the properties of walls.
This contradiction shows that $\lam$ cannot be a foliation
and finishes the proof of the proposition.
\end{proof}

\subsection{The complementary regions of $\lam$}

Here we prove that the complementary regions of $\lam$ are particularly
simple.

\begin{proposition} \label{complementary}
Suppose that $M$ is atoroidal.
Let $E$ be a leaf of $\wls$. The 
complementary regions of $\wl \cap E$ are finite sided ideal
polygons in $E$. 
The complementary regions of $\lam$ are open solid tori
or open solid Klein bottles,
which are obtained as suspensions of homeomorphisms of
finite sided ideal polygons.
\end{proposition}

\begin{proof}
Taking a double cover if necessary we assume that $M$ is
orientable. This does not affect the conclusion of the
Proposition.
First of all $\lam$ is not a foliation, hence there
are complementary regions, which obviously are open.
Any leaf $L$ of $\ls$ is dense in $M$, because $M$ is
atoroidal, hence $L$ intersects a complementary
region. 
In addition every leaf of $\wl$ is a plane, and therefore
the leaves of $\lam$ are incompressible in $M$.

Let $\eps_0$ be the constant given by Proposition \ref{asymptotic1}.
Choose $\eps < \eps_0/2$. 
Any leaf $F$ of $\wls$ is isometric to $\mathbb H^2$.
Any complementary region of $\wl \cap F$ contains
at least an ideal triangle.
The $\eps$ is also chosen much smaller than the inradius
of an ideal hyperbolic triangle.

A lot of the arguments are done in $M$.
Consider the set $B_0$ of point $x$ in $M - \lam$ so that $x$ is
in a complementary region of $F \cap \lam$ ($F$ leaf of $\ls$)
and the distance along $F$ from $x$ to a boundary leaf of $F \cap \lam$
is exactly $\eps$. 
This set is non empty.

We claim that each connected component of $B$ of $B_0$ is a 
topological closed surface
which is transverse to $\ls$ and intersects it in a one
dimensional foliation in $B$.
It will have corners as shown in Figure \ref{figure16}.
Parts of this analysis are similar 
to what  was done by Thurston in \cite{Th3}.
Let $x \in B$. If $x$ is $\eps$ distant from exactly one
boundary leaf of $F \cap \lam$, then $B \cap F$ is locally 
one dimensional near $x$. 
Since $\lam$ is transverse to $\ls$, the same happens for
leaves $F'$ of $\ls$ near $F$ (along $x$)  and with the corresponding
complementary regions. Therefore $B$ is two dimensional
near $x$ and $\lam \cap \ls$ induces a one dimensional
foliation in $B$ near $x$. 

\begin{figure}[ht]
\begin{center}
\includegraphics[scale=0.80]{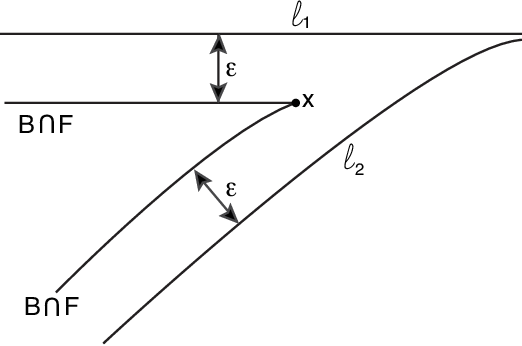}
\begin{picture}(0,0)
\end{picture}
\end{center}
\vspace{0.0cm}
\caption{$\ell_1, \ell_2$ are boundary leaves
in $F \cap \lam$. $x$ is a corner point of $B \cap F$.}
\label{figure16}
\end{figure}

Suppose now that $x$ is a point in $B \cap F$ which is 
$\eps$ away from at least two boundary leaves $l_1, l_2$
of $F \cap \lam$. Then there are exactly two such
leaves in the boundary of the same complementary
region of $F \cap \lam$. By Proposition \ref{asymptotic1} the leaves
$l_1, l_2$ are asymptotic in one direction and hence
they diverge from each other in the other
direction, see Figure \ref{figure16}. This implies that locally
near $x$ the set $B \cap F$ is also one dimensional and
forms a ``corner" at $x$, see Figure \ref{figure16}.
In particular moving in different directions in $l_1, l_2$
from points $\eps$ close to $x$, the leaves
$l_1, l_2$ get closer than $\eps$ to each other in one
direction and farther than $\eps$ 
 from each other in the opposite
direction. Therefore
one can also see this behavior in 
leaves $F'$ of $\ls$ near $F$ along $x$. In other words
$B \cap F'$ is one dimensional locally and has a corner 
$x'$ which is near $x$ in $M$. 
This analysis shows that $B$ is locally two dimensional,
that is, it is  a surface. In addition
$\ls \cap B$ induces
a one dimensional foliation in $B$. 

\vskip .1in
An important property is that the surface $B$ cannot
accumulate anywhere in $M$: 
first of all notice that $B_0$ is a closed subset of $M$,
because it is defined by a closed condition.
The way $\eps$ is chosen, a point can be $\eps$ from at most
two boundary components.
If a point $x$ is $\eps$ from a boundary component, then
either \ 1) it is not $\eps$ from another boundary component
and we constructed a small arc of $B_0$ in that leaf, and
this is the only intersection of $B_0$ with a small
neigbhborhood of $x$ in its $\ls$ leaf, or \ 2) $x$ is
$\eps$ away from exactly two boundary components, $x$ is a
corner, and a curve neighborhood containing the corner
is the only intersection of $B_0$ with a small neighborhood
of $x$ in its $\ls$ leaf.

We conclude that $B$ is compact. Also it has no
boundary, so it is a closed surface.
In addition $B$ has 
an induced one dimensional foliation given by $\ls \cap B$.
Since $B$ is obviously two
sided and $M$ is orientable, it follows that $B$ is a
two dimensional torus. 
Notice also that the set of ``corner" points in $B$ is locally
one dimensional by the above arguments. The arguments above imply
that any point in $B$, be it a regular point or a corner
point, has a neigbhorhood either with no corner points, or
a neighborhood with a single corner curve. It follows that 
corner curves cannot accumulate anywhere in $B$, 
hence the set of corner curves is a finite set of closed
curves in $B$.
In addition any such corner curve is transverse to $\ls$, 
so it is $\pi_1$ injective in $M$, and hence in $B$
as well.

A point in $B$ is $\eps$ near one or two leaves of $\ls$: 
these leaves are tracked by the point. One can only change
the tracking leaf if one crosses a corner curve.
Hence there are only finitely many tracking leaves in
the boundary. But also every leaf in the boundary is tracked,
so the boundary of the complementary region is made up of
finitely many leaves of $\ls$.

If $B$ has no corner points 
then $B$ will $\eps$ track a single leaf $L$ of $\ls$. 
It follows that $L$ is also a torus, and it is isotopic
to $B$. But as $M$ is atoroidal, this would
imply that $L$ is compressible, 
contradicting the fact that leaves of $\lam$ are
incompressible. 

Since there are corner curves, we pick one and denote it by $\mu$,
which we stress is a closed curve.

Again because $M$ is atoroidal, and $B$ is an embedded torus in $M$,
then either $B$ bounds a solid torus, or $B$ is contained in
a three ball. 
The second option is impossible because of the following.
The curve $\mu$ (the corner curve above) is closed
and it is transverse to $\ls$. By Novikov's theorem
$\mu$ cannot be null homotopic, hence cannot be contained
in a closed ball.
Hence $B$ bounds a solid torus $V$ in $M$.

%
%

\vskip .05in
The non empty set of corner curves in $B$ forms a finite set of closed
curves which are not null homotopic in $B$, and are 
pairwise isotopic in $B$ because they are disjoint, not null homotopic,
and $B$ is a torus..
Let $\U$ be the one dimensional foliation in $B$ obtained
by intersecting $B$ with leaves of $\ls$. We want to show
that this is a foliation by closed curves in $B$, each of
which bounds a disk in $V$ and bounds a disk in the
respective leaf of $\ls$ containing it.

The corner curves in $B$ cut up $B$ into finitely many
annuli, and the boundaries are transverse to $\U$.
Suppose that there 
is one such
annulus,
denoted by $A$ so that the foliation $\U$ in $A$ is not a product.
There is a leaf of $\U$ intersecting one boundary component
of $A$ but not the other. This leaf limits to a closed
leaf of $\U$ in the interior of $A$. 
The annulus $A$ is $\eps$ close to a single leaf $Z$ of
$\lam$. The foliation $\U$ in $A$ pushes along leaves
of $\ls$ to a foliation of an annulus in $Z$.
This foliation is part of the foliation $\ls \cap Z$.
Summarizing: in $Z$ the foliation $\ls \cap Z$ has 
a closed leaf $\gamma$ and also has a leaf $\eta$
which has a point $\eps$ near a corner, hence $\eta$ is asymptotic
to another leaf of $\lam$ $-$ by Proposition \ref{asymptotic1}.
Lift $Z$ to a leaf $L$ of $\wl$, and lift $\gamma, \eta$
to $\widetilde \gamma, \widetilde \eta$.
Notice that $Z$ is an annulus. In addition the set of curves
$\widetilde \eta$ as above fills one complementary component
of $\widetilde \gamma$ in $L$.
But in a wall this is impossible: there are many corner leaves
(of the foliation $\wl \cap L$)
in any component of $L - \widetilde \gamma$,
and they are not asymptotic to $\widetilde \gamma$.

We have shown that the foliation $\U$ restricted to any annulus is a product.
Suppose that a leaf of $\U$ has a return to $\mu$ which is
not the initial point. Lift to $\mt$: any lift of $B$ is
an infinite cylinder and there is a unique lift $\widetilde
\mu$ of $\mu$ to this cylinder. The above implies that
the lift of the leaf of $\U$ to this cylinder intersects
$\widetilde \mu$ in more than one point. But $\widetilde
\mu$ is a transversal to $\wls$, so again Novikov's theorem
implies that this is impossible.

This now proves that all leaves of $\U$ are closed and in fact
they lift to closed curves in $\mt$. It follows that the
leaves in $\U$ 
bound disks in the respective leaves of $\ls$. The
boundary does not intersect $\lam$, so the disks do
not intersect $\lam$ either and are contained in 
complementary regions of $\lam$.
The boundary of these disks are the curves in $\U$:
they are compact, polygonal curves (convex
when lifted to the universal cover). To get the
complementary region in the leaf of $\ls$ we just at
the points $< \eps$ from the boundary: these
are finitely many relatively compact strips, plus
finitely many cusps. It follows that the complementary
region is the interior of a finite sided ideal polygon.

This finishes the proof of the Proposition \ref{complementary}.
\end{proof}

We are now ready to prove one case of our Main theorem.

\begin{theorem} \label{nontransverse}
Let $M$ be atoroidal.
Suppose that for some wall $W$ obtained as the limit
of partial walls associated to strings of lozenges,
every deck translate $g(W)$
does not interesect $L$ transversely. Then $\Phi$ is
$\rrrr$-covered.
\end{theorem}

\begin{proof}
Up to a finite cover we may assume that $M$ is orientable
and $\ls, \lu$ are transversely orientable.
Recall that we constructed a leafwise geodesic lamination $\lam$
with leaves being projections
of walls to $M$. 
This was proved in Proposition \ref{prop-gettinglamination}.
 The previous propositions show that
$\lam$ is not a foliation,
and the complementary regions of $\lam$ intersect
leaves of $\ls$ with components which are finite sided
ideal polygons.

The proof will be by contradiction. Assume that $\Phi$ is
not $\rrrr$-covered, and choose distinct leaves $Y_1, Y_2$ of
$\wls$ which are non separated from each other and with half
leaves part of the boundary of adjacent lozenges 
$D_1, D_2$ as in
Figure \ref{figure17} (a).
Let $\alpha_i$ be the periodic orbits in $Y_i$.
Let $U_i$ be the unstable leaf of $\alpha_i$.
Let $Y$ be the leaf of $\wls$ entirely contained in
$\partial(D_1 \cup D_2)$. Let $\alpha$ be the
periodic orbit in $Y$ and let $U = \wlu(\alpha)$.
Let $g$ be a generator of the isotropy group  of
$\alpha$, which is also a generator of the isotropy group 
of $\alpha_1, \alpha_2$.

\begin{figure}[ht]
\begin{center}
\includegraphics[scale=1.10]{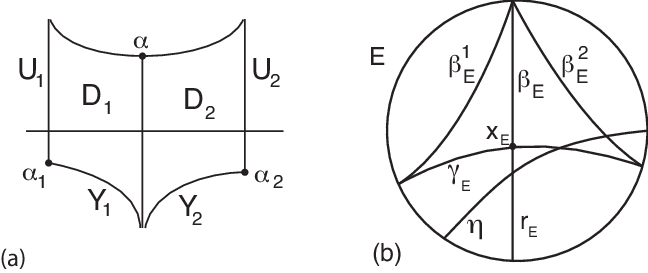}
\begin{picture}(0,0)
\end{picture}
\end{center}
\vspace{0.0cm}
\caption{a. Non separated stable leaves, b. The segment
$r_E$ and the case that $\eta$ intersects it 
transversely.}
\label{figure17}
\end{figure}

We will consider the set $\II$ of stable leaves $E$ which
intersects $D_1$. 
Each such leaf $E$ intersects $U_1, U_2, U$
in flow lines which are geodesics in $E$.
Let these intersections be $\beta^1_E, \beta^2_E, \beta_E$
respectively.
Consider the geodesic $\gamma_E$ in $E$ with ideal points
the negative limit points of $\beta^1_E, \beta^2_E$.

Let $x_E$ be the intersection point of $g_E$ and $\beta_E$,
see Figure \ref{figure17} (b).
Let $r_E$ be the ray of $\beta_E$ starting
at $x_E$ and with ideal point the negative ideal point of
$\beta_E$.
We define 

$$A \ \ = \ \ \bigcup_{E \in \II} \ r_E$$

\noindent
Notice that $A$ is a subset of $U = \wlu(\alpha)$.
The points $x_E$ vary continuously with $E$ in $\II$ because
the geodesics $\gamma_E$ vary continuously as they are asymptotic to
each other (for different $E$) in both directions.
Similarly $\beta_E$ also varies continuously with $E$ as they
are all flow lines in $U$. Hence the boundary component
of $\partial A$ made up by the points $r_E$
is a continuous curve in $U$. The set $A$ is also $g$ invariant,
since all objects are $g$ invariant: $U_1, U_2, U, \II$ and
so on.
Hence the projection $\pi(A)$ is a subannulus of the annulus
$\pi(U)$ in $M$.

The main property we need is the technical result
below. By a transverse intersection of $A$ with a leaf
$H$ of $\wl$ we mean a point $p$ in $H \cap A$ so 
that if $E = \wls(p)$, then $\eta = E \cap H$ 
intersects $A$ transversely.

\begin{lemma} \label{nocross}
The set $A$ has no transverse intersection with the leaves
of $\wl$.
\end{lemma}

\begin{proof}
We do the proof of the lemma by contradiction.
Suppose the lemma is false.
Then there is $E \in \II$ so that $r_E$ 
intersects $\eta := H \cap E$ transversely.

We refer to Figure \ref{figure17} (b).
Let $J_i, i = 1, 2$ be the open interval in $S^1(E)$ between the ideal
point of $r_E$ and the negative ideal point of $\beta^i_E$
which does not contain the positive ideal point of all flow lines
in $E$.
Since $\eta$ intersects $r_E$ 
transversely, it follows that $\eta$ has at least one
ideal point in $J_1 \cup J_2$. We assume without loss of
generality that $\eta$ has an ideal point in $J_1$.
This is depicted in Figure \ref{figure17} (b).

The main idea to obtain a contradiction is that the pair
of adjacent lozenges $D_1, D_2$ prevents the lozenges
or ideal quadrilaterals in the bi-infinite block associated
with $H$ from closing up. There are various possibilities
depending on whether the elements of the bi-infinite block
are lozenges or $(1,3)$ ideal quadrilaterals.

Let $\C$ be the bi-infinite block associated with $H$.
Let $\cB$ be the basic block of $\C$ containing $\eta$.
The geodesic $\eta$ is transverse to the flow in $E$.
The ray of $\eta$ which has ideal point in $J_1$ is
asymptotic to a flow line $\nu$ in $E$
which has negative ideal point in $J_1$.
Let $V = \wlu(\nu)$. 
Part of the boundary of $\cB$ is contained in $V$: it 
could be the full leaf $V$ (for example when there are adjacent
$(1,3)$ ideal quadrilaterals in $\cB$), or a half leaf of $V$ 
(for example when $\cB$ is a 
single lozenge with an unstable side contained in
$V$), or a bounded flow band (two or more adjacent
$(1,3)$ ideal quadrilaterals).

There are three possibilities for $\cB$: a single lozenge,
a stable adjacent block, an unstable adjacent block.

\vskip .05in
\noindent
{\bf {Case 1 - $\cB$ is a single lozenge.}}

Denote the lozenge by $C$.
The lozenge $C$ has one unstable side contained in $V$. 
It has a stable side, contained in a 
leaf $F$ which makes a perfect fit 
$V$. 
Notice that $C$ has to intersect $D_1$ because
$\eta$ intersects $D_1$.
There are two possibilities for $F$
depending on which component of $\mt - Y$ it 
is contained in. In both cases it follows that
the other unstable side of $C$ is contained in a leaf
$V'$ which has to intersect $D_1$ as well.
See example in Figure \ref{figure18} (a),
where $F$ is contained
in the component of $\mt - Y$ which does contain
$D_1$.
This implies that $\eta$ is contained
in $D_1$. But then $\eta$ does not intersect $U = \wlu(\alpha)$,
contradiction.
We conclude that Case 1 cannot happen.


The analysis of Case 1 highlights how the adjacent pair of
lozenges $D_1, D_2$ traps the basic block $\cB$ in the 
case of $\cB$ being a lozenge.
A similar situation happens more generaly.


\begin{figure}[ht]
\begin{center}
\includegraphics[scale=1.10]{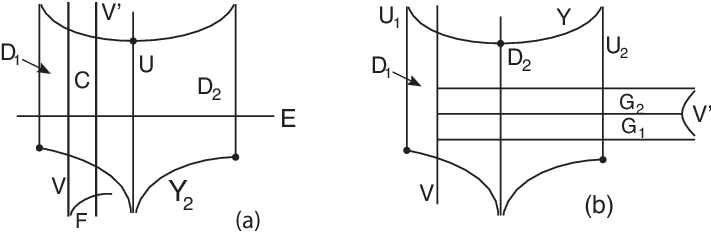}
\begin{picture}(0,0)
\end{picture}
\end{center}
\vspace{0.0cm}
\caption{a. Case when $\cB$ is a single
lozenge, b. Case when $\cB$ is an unstable adjacent block.}
\label{figure18}
\end{figure}

\vskip .05in
\noindent
{\bf {Case 2 - $\cB$ is an unstable adjacent block.}}

If the element of $\cB$ containing $\eta$ is a lozenge,
the proof of Case 1 shows that $V$ has to intersect
$D_1$, hence as in Case 1, $\eta$ does not
intersect $U$, contrary to assumption.
If the element of $\cB$ containing $\eta$ is a $(1,3)$ ideal
quadrilateral, then there is a pair 
of adjacent $(1,3)$ quadrilaterals $G_1, G_2$ in $\cB$ intersecting
a common unstable leaf.
In this case either the full leaf $V$ is contained
in $\partial(G_1 \cup G_2)$ or there is an unstable
leaf, denoted by $V'$, which is distinct from $V$ and which
is fully 
contained in $\partial(G_1 \cup G_2)$.
Suppose first that the full leaf $V$ is contained
in $\partial(G_1 \cup G_2)$.
Then $V'$ intersects $D_1$, and so $\eta$ does 
not intersect $U$, again a contradiction.

Suppose now that  another leaf
$V'$ is contained in $\partial(G_1 \cup G_2)$.
The only possibility
is that $V'$ is contained in the component of $\mt - U_2$
that does not contain $U$, see Figure \ref{figure18} (b).
In this case we do not yet obtain a contradiction, this is
because it allows $\eta$ to intersect $D_2$.
Now we move in the direction of $Y$.
As long as there are adjacent $(1,3)$ ideal quadrilaterals,
one unstable side is contained in $V$ and the other unstable
side is a leaf non separated from $V'$ and contained in the
component of $\mt - U_2$ not containing $U$. Eventually
one arrives either at a lozenge or at a hinge.
Supose first it is a lozenge, denote it by $C^*$. 
It has one unstable side contained in 
the leaf $V$
and the other unstable side contained in a leaf non
separated from $V^*$. 
This is disallowed by the analysis of Configuration 0.
If the wall hits a hinge, the hinge has to be in $V$.
But this is also disallowed by the analysis in Configuration 0.

This finishes the proof of Case 2.

\vskip .05in
\noindent
{\bf {Case 3 - $\cB$ is a stable adjacent block.}}

If the ``first" element of $\cB$ (the one
which has an unstable side 
contained in $V$) is a lozenge $C$,
then this was dealt with in Case 1.

Suppose now that the ``first" element of $\cB$ (the one which 
has an unstable side contained in $V$) is a $(1,3)$ ideal quadrilateral.
It is in a pair of $(1,3)$ ideal quadrilaterals, where all unstable
sides intersect $D_1$. This keeps on happening as long as there
are $(1,3)$ quadrilaterals, without producting intersection
of $\eta$ with $D_2$.
the final lozenge. Again we conclude that the final unstable
side intersects $D_1$, contradicting that $\eta$ intersects
$D_2$.
The case of lozenge is disallowed by the analysis of Configuration 0.
This shows that Case 3 cannot happen either.

This finishes the proof of Lemma \ref{nocross}.
\end{proof}

\noindent
{\bf {Completion of the proof of Theorem \ref{nontransverse}.}}

The previous lemma shows that $A$ does not intersect 
leaves of $\wl$ transversely. The projection $\pi(A)$ of $A$ 
is a subannulus of $\pi(U)$, see Figure \ref{figure19} (a).
We need to understand this subannulus better.
The surface $A$ is the union of $r_E$ where $E$ 
runs through all stable leaves intersecting 
both $U_1, U_2$.
Each $r_E$ is a negative ray in a flow line
in $U$. All negative rays in $U$ are asymptotic.
Recall that $\alpha = U \cap Y$ is periodic, invariant
under $g$. It follows that $\pi(\alpha)$ is a 
periodic orbit, contained in the annulus $\pi(U)$.
The projections $\pi(r_E)$ are all asymptotic to
$\pi(\alpha)$. 
On the other end of $r_E$ we have the point $x_E$
and the collection of these is a continuous curve, 
also invariant by $g$. Hence it projects to a
simple closed  curve $\delta$ in $\pi(U)$.

As a conclusion we obtain that $\pi(A)$ is a half
open annulus with compact closure in $\pi(U)$ with a boundary component
$\delta$ contained in $\pi(A)$ and the other 
boundary component $\pi(\alpha)$ disjoint from
$\pi(A)$. 

We proved in Lemma \ref{nocross} that $A$ does not intersect
any wall transversely, and so $\pi(A)$ does not intersect
any projection of a wall transversely. 
Since $\pi(\alpha)$ is in the boundary of $\pi(A)$ it now follows
that $\pi(\alpha)$ does not intersect any projection of a wall transversely.

This is the important property we want. Lifting to $\mt$
it follows that $\alpha$ does not intersect any wall
in $\wl$ transversely.

Consider the situation in $Y = \wls(\alpha)$. 
Then $Y$ intersects $\wl$ and
the complementary regions of $\wl \cap Y$ are finite sided
ideal polygons. 
Since $\alpha$ does not intersect $\wl \cap Y$ transversely
then either $\alpha$ is a leaf of $\wl \cap Y$ or $\alpha$
stays forever in a complementary region of $\wl \cap Y$.
We call the second type a diagonal curve.
We analyze each case separately, showing that each case is
impossible.
Notice that $\alpha$, $Y$, $\wl \cap Y$
are all invariant under the non trivial deck transformation 
$g$.

\begin{figure}[ht]
\begin{center}
\includegraphics[scale=1.00]{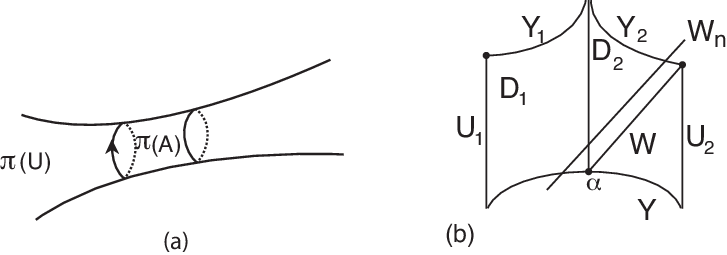}
\begin{picture}(0,0)
\end{picture}
\end{center}
\vspace{0.0cm}
\caption{a. The unstable leaf $\pi(U)$ and the 
annulus $\pi(A)$  with compact closure in $\pi(U)$, 
b. Walls $L_n$ converging to wall $L$.}
\label{figure19}
\end{figure}

\vskip .05in
\noindent
{\bf {Situation 1 $-$ There is a complementary region $X$ of
$\wl \cap Y$ so that $\alpha$ is either a boundary component
of $X$ or $\alpha$ is a diagonal of $X$.}}

In either case since $g(\alpha) = \alpha$ and 
$g(\wl \cap Y) = \wl \cap Y$, it follows that 
$g(X) = X$. But this is impossible since $\gamma$ has
only two fixed points in $S^1(Y)$, but $X$ has at least
$3$ ideal points. All ideal points of $X$ need to be
fixed by $g$. We conclude that this situation cannot happen.

\vskip .1in
\noindent
{\bf {Situation 2 $-$ $\alpha$ is a leaf of $\wl \cap Y$ 
which is not isolated on both sides.}}

Let $H \in \wl$ be the wall containing $\alpha$,
and let $\HH$ be the bi-infinite block associated with $H$.
Since $\alpha$ and hence $H$ is not isolated on both sides, 
there are sequences of walls $H_n, W_n$
in $\wl$ so that $H_n, W_n$ converge to $H$, 
and $H$ separates $H_n$ from $W_m$
for any $n, m$.

Recall that $\alpha$ is periodic, and since $\alpha \subset H$,
$\alpha$ is a corner orbit of $H$. Lemma \ref{cornerperio} implies
that all the corner orbits of $\HH$ are periodic 
and all elements of $\HH$ are lozenges. 
We consider the two elements of $\HH$ with a corner $\alpha$.
If these  two lozenges are (the previously defined
lozenges) $D_1, D_2$ (of which $\alpha$ is a corner of both),
then there is a basic block
of $\HH$ containing both of them, and this basic block
is a stable adjacent block. But since it is a stable
adjacent block, then the only corner orbits
of the basic block of $\HH$ which are corners
of $H$ are the hinges of the basic block,
and $\alpha$ could not be a corner orbit of $H$, contradiction.
The same proof disallows the case that none of the elements of
$\HH$ with a corner in $\alpha$ are $D_1$ or $D_2$. It follows
that one and only of $D_1, D_2$ is an element of $\HH$.

Without loss of generality we assume that this
element of $\HH$ is $D_2$. We refer to
Figure \ref{figure19} (b). Up to switching $H_n$ with $W_n$ assume that $W_n$
intersects $Y = \wls(\alpha)$ in the boundary
of $D_1$ for $n$ big enough, see Figure
\ref{figure19} (b).
As $W_n$ converges to $H$, then for $n$ big enough $W_n$
intersects both $D_1$ and $D_2$ $-$ this is because
$\HH$ intersects $D_2$.

We study the wall $W_n$.  For $n$ big, $W_n$ intersects
$A$, the subset of $U$ described above. 
Fix one such $n$ and let $p$ in $W_n \cap A$. Then $p$ is in
a leaf $E$ of $\wls$. 
Let $\eta = E \cap W_n$.
Lemma \ref{nocross} shows that $\eta$ does not intersect
$A$ transversely, hence the direction of $\eta$ at $p$
is the direction of $r_E$: in other words $\eta$
is a flow line of $\wwp$, rather than being transverse
to the foliation by flow lines in $E$.
It follows that $\eta$ is a corner orbit of $W_n$.
Notice that $p$ is in $\eta$ and $p$ is in $U$, hence
$\eta$ is an orbit in $U = \wlu(\alpha)$.

Recall that we assumed 
that $n$ is big enough so that $W_n$
intersects $Y$ in the half leaf defined by $\alpha$
and contained in the boundary of $D_1$.
Let $\E$ be the bi-infinite block associated with 
$L_n$. 

First notice that $\eta$ is a hinge of $\E$.
The reason is that $W_n$ intersects both $D_1, D_2$
so if there were are adjacent elements of $\E$
at $\eta$, the basic block of $\E$ would be 
a stable adjacent block, and
$\eta$ would not be a corner.

This leads to a contradiction as follows: since $\eta$
is a hinge, then at least one of the two 
elements of $\E$ with a corner at $\eta$ is
contained in a quadrant at $\eta$ that intersects
the leaf $Y$ 
(these quadrants do not intersect $Y_1$ or $Y_2$).
In one of these quadrants $Y$ makes a perfect fit with
$U_1$, and in the other quadrant $Y$ makes a perfect fit with $U_2$.
In addition $\wls(\eta)$ intersects both $U_1, U_2$, and
$\wlu(\eta)$ intersects $Y$. The analysis of Configuration
0 now shows that this situation is impossible. This uses
that $\eta$ is a hinge of $L_n$.


The contradiction is derived from assuming that $\ls$ is
not $\rrrr$-covered.
We conclude that in this case $\ls$ is $\rrrr$-covered
and hence $\Phi$ is $\rrrr$-covered.

This finishes the analysis of the Main theorem in the 
case that there is a wall $L$ and for any $\gamma$ in $\pi_1(M)$
$\gamma(L)$ does not intersect $L$ transversely,
and finishes the proof of Theorem \ref{nontransverse}.
\end{proof}

\section{The case of transverse intersections}
\label{transverse1}

\noindent
{\bf {New assumption $-$}}
From now on in the proof of the Main Theorem we assume that for
any wall $W$ obtained by our limiting process there is a deck
transformation $g$ so that $g(W)$ intersects
$W$ transversely.

\vskip .05in
We will prove that the bi-infinite blocks $\C$ of $W$ and 
$\E$ of $g(W)$
are linked: the walls intersect exactly the same set of
leaves of $\wls$, and for each such leaf $E$ of $\wls$,
then $W \cap E$ and $g(W) \cap E$ are geodesics which
intersect transversely. 
An obvious  possibility is that both $\C$ and $\E$
are bi-infinite strings of lozenges which are interlocked
as in Figure \ref{figure20}.
But a priori there are more complicated possibilities 
with adjacent blocks with $(1,3)$ ideal quadrilaterals $-$
in particular we do not know a priori that a corner of one
of the blocks is in the interior of the other block,
this is one main property to be proved.

Given two pairs of points in a circle, we say that the
pairs are linked if each separates the other in the circle.

\begin{figure}[ht]
\begin{center}
\includegraphics[scale=0.75]{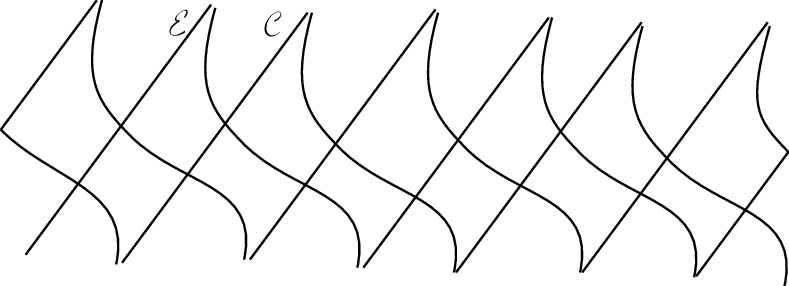}
\begin{picture}(0,0)
\end{picture}
\end{center}
\vspace{0.0cm}
\caption{Interlocked strings of lozenges $\C, \E$.}
\label{figure20}
\end{figure}

\begin{proposition} (linking) \label{linking}
Assume that $M$ is atoroidal and $\ls, \lu$ are transversely
orientable.
Let $W$ be a wall with associated bi-infinite block $\cC$
and $g$ in $\pi_1(M)$ so that $g(W)$
intersects $W$ transversely at some point. 
Let $\cE$ be the bi-infinite block of $g(W)$.
Then $\cC, \cE$ intersect the same set of stable
leaves and the same set of unstable leaves.
Any corner orbit of $L$ is contained
in the interior of a basic block of $\cE$ and vice
versa.
Finally for every leaf $F$ intersected by $L$, then 
$L \cap F, \ g(L) \cap F$ are geodesics in $F$
which intersect each other transversely.
Finally $\cC$ (hence also $\cE$) is a bi-infinite string
of lozenges.
\end{proposition}

\begin{proof}
Let $\X_0$ be the set of stable leaves intersecting $W$
and $\X_1$ the corresponding set for $g(W)$.
Both are properly
embedded copies of $\rrrr$ in the stable leaf space.
One main property we want to prove is that $\X_0 = \X_1$.
Let $v$ be a point of transverse intersection of $W$ 
and $g(W)$, and let $E = \wls(v)$,
hence $v$ is a point in both geodesics 

$$\gamma_1 = W \cap E, \ \ {\rm and} \ \ 
\gamma_2 = g(W) \cap E.$$

\noindent
Clearly the set of stable leaves where $W, g(W)$ intersect
transversely is open. Also the set of orbits of $\wwp$ contained
in $W, g(W)$ is discrete. Hence we can assume that
neither $\gamma_1, \gamma_2$ has an ideal point which is the positive
ideal point of flow lines in $E$.
Hence there are $4$ distinct flow lines which have
negative ideal points the ideal points of $\gamma_1 \cup \gamma_2$ in $E$.
Let $X, U, Y, V$ be the unstable leaves through
these flow lines. We assume that $X, U, Y, V$ are nested,
that is, $X \cap E, U \cap E, Y \cap E$ and $V \cap E$
are nested flow lines in $E$, this means that the first
and last do not separate the other $3$ in $E$, but the
second separates the first from the others, and the third
separates the fourth from the others.
We assume that the negative ideal points of 
$U \cap E, \ V \cap E$ are
the ideal points of $\gamma_1 = W \cap E$ and the negative
ideal points of $X \cap E, \  Y \cap E$ are the ideal points of $\gamma_2 =
g(W) \cap E$. 
We want to prove the following:

\vskip .08in
\noindent
{\bf {Claim $-$ Up to switching $W, g(W)$, there is a corner
of $\cC$ which is contained in the interior of $\cE$.}}
\vskip .05in

The proof is fairly involved and will have several cases.

Fix a component $A$ of $\mt - E$.
As long as a stable leaf $F$ in $A$ intersects all of $X, U, Y, V$
then $W, g(W)$ intersect $F$ and the ideal points of
$W \cap F, g(W) \cap F$ are the negative ideal points
of $W \cap F$ and $g(W) \cap F$ and they are linked.

Hence there is a first leaf stable leaf $E_0$ in $A$ and
in $\X_0$, which stops intersecting either $U$
or $V$. Recall that $E_0$ still intersects one of them.
The same happens for the pair $X, Y$ and $g(W)$.
There is a first leaf $E_1$ in $A$,
so that $E_1$ is in $\X_1$,
and $E_1$ intersects one of $X, Y$ but not the other.
There are two possibilities.

\vskip .05in
\noindent
{\bf {Case 1 $-$ 
There are leaves $F \in \wls$ in $A$ 
intersecting three of $X, U, Y, V$ but not all four of them.}}

This is the simpler case.
The leaf of $X, U, Y, V$ not intersecting $F$
has to be an end leaf, that is, either $X$ or $V$.
Without loss of generality we assume it is $V$. 
Hence there is a first such leaf with this behavior,
which is the leaf $E_0$ that we defined above. 
Notice that in this case $E_0$ is still in $\X_1$
and $E_0$ separates $E$ from $E_1$.
We know that $E_0$ intersects $U$, and
it follows that 
the intersection of $W$ with $E_0$ is the orbit
$U \cap E_0$. 

The intersection of $g(W)$ with $E_0$ is the geodesic
with ideal points which are  the negative ideal points of 
$X \cap E_0$ and $Y \cap E_0$. 
But the pair of ideal points of
the flow line $U \cap E_0$ (both positive and
negative) link with the negative ideal
points of $X \cap E_0, \ Y \cap E_0$ in $S^1(E_0)$.
This is because 
$X, U, Y, V$ are nested.
Hence $g(W) \cap E_0$ intersects $W \cap E_0$
transversely and $W \cap E_0$ is a corner of $L$.

This proves the claim in this case, that is, the corner 
orbit $U \cap E_0$ of $W$ is in the interior of an element
of the bi-infinite block $g(W)$.

Continuiing in this case we have the following:
The leaves of $\wls \cap W$ which are tangent to the
flow lines of $\wwp$ form a discrete set of leaves.
It follows that $\X_0, \X_1$ contain an open interval $I$ of stable leaves
with $E_0$ in the interior and there is $F$ in $I$ with
$F \cap W, \ F \cap g(W)$ intersecting transversely,
both not flow lines of $\wwp$, and $E_0$ separating $F$ from $E$.
In this case we can restart the analysis with $F$ instead
of $E$. The leaves $X, U, Y$ will remain the same, but
$V$ will be replaced by another leaf.

This case is what we want to show always happens:
if the intersection of (say) $W$ with stable leaves 
becomes a flow line in a stable leaf $F$ then $F$
still intersects $g(W)$ and 
$F \cap g(W)$ intersects $F \cap W$ transversely.
The other case to consider is:

\vskip .1in
\noindent
{\bf {Case 2 $-$ For any stable leaf $F$ in that
component of $\mt - E$, then all four $Z, U, Y, V$ intersect $F$
or at most two of them intersect $F$.}}

This happens when the stable leaves stop intersecting one of $U, V$
at the same time they stop intersecting one of $X, W$.
It follows that there are two possibilities here: either $E_0 = E_1$
or they are distinct but non separated from each other.

\vskip .1in
\noindent
{\bf {Case 2.1 $-$ $E_0, E_1$ are distinct but non separated
from each other.}}

In this case $E_0$ is in $\X_0$ but not in $\X_1$,
since $E_1$ is in $\X_1$.
Also $E_0$ intersects one of $U, V$ but 
not the other, and $E_1$ intersects one of $X, Y$ but not
the other.

In this case there are more possibilities for the relative
positioning of the relevant leaves.
Fix a transverse orientation to $\wls$ so that 
$E_1$ is on the positive side of $E$, and fix a transverse
orientation to $\wlu$ so that $U$ is on the positive
side of $X$.
We put an order in the pair $E_0, E_1$: we say that
$E_0 < E_1$ if $E_0, E_1$ intersect unstable leaves $Z_0, Z_1$
respectively so that $Z_1$ is on the positive side
of $Z_0$, and $E_1 < E_0$ otherwise.
This is well defined (independent of $Z_0, Z_1$): since
$E_0, E_1$ are non separated, then for any $Z_0, Z_1$ there
is a stable leaf $F$ intersecting both $Z_0, Z_1$ and
the transversal orientation to $\wlu$ says in $F$ which one
of the flow lines $Z_0 \cap F, Z_1 \cap F$ is in front
of the other.

\vskip .05in
Suppose first that $E_0 < E_1$.
In this case $E_0$ cannot intersect $V$, since $V$ is on the
positive side of both $Y$ and $X$, so $E_0$ must intersect
$U$, 
and $E_1$ must intersect $Y$.
Let $\alpha = E_0 \cap U$. By construction this is contained
in $W$ and it is a corner orbit of $\C$. 
Let $\R$ be the basic block of $\C$ which intersects $E$.
In addition
the leaf $E_0$ does not intersect all the unstable leaves
that the basic block $\R$ intersects $-$ this is because
$E_1$ contains part of the boundary of $\R$ and so 
there are unstable leaves intersecting $E_1$ which intersect
$\R$ and these unstable leaves do not intersect $E_0$.
See Figure \ref{figure21} (a).

\begin{figure}[ht]
\begin{center}
\includegraphics[scale=1.00]{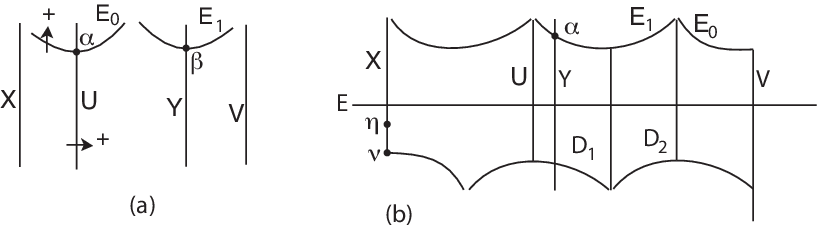}
\begin{picture}(0,0)
\end{picture}
\end{center}
\vspace{0.0cm}
\caption{a. The figure when $E_0 < E_1$, 
the arrows indicate the positive transversal 
orientation; \ b) The figure when $E_1 < E_0$.}
\label{figure21}
\end{figure}

It follows that
$\alpha$ is a hinge of $W$ and $\C$.
This is because $E_1$ intersects the boundary of $\R$.
Hence $\R$ is a stable adjacent block and $\alpha$ is a
hinge.
In the same way $\beta := E_1 \cap Y$ is a hinge of 
$g(W)$.
This is a contradiction to $g$ preserving orientation as follows:
in $W$ when one moves positively transverse to $\wls$, 
that is, from $E$ to $E_0$, then eventually 
one moves negatively transverse to $\wlu$. The last fact
occurs when $W$ crosses the hinge $\alpha$, in the direction
from $E$ to $E_0$. In terms of unstable leaves, $W$ moves
from intersecting $Y$ to not intersecting $Y$, but 
intersecting $E_0$, that is, in the negative direction
transverse to $\wlu$.
Since $\alpha$ is a hinge, it determines the global
direction of moving transversely to the unstable
foliation.

In $g(W)$ when one moves positively transversely
to $\wls$ one moves positively transversely to $\wlu$:
that is from $U$ to $Y$. By assumption all deck transformations
of $\pi_1(M)$ preserve transversal orientations, hence
this is a contradiction.
This contradiction shows that this case cannot happen.


\vskip .05in
The other possible situation is that 
$E_1 < E_0$.
There are two possibilities here. Suppose first that
$E_1$ intersects $Y$. This forces $E_0$ to intersect
$V$. We refer to Figure \ref{figure21} (b).
Notice that $X$ makes a perfect fit with a stable leaf
non separated from $E_0$, and $U$ makes a perfect fit
with a stable leaf non separated from $E_1$ 
(which is also non separated from $E_0$). The figure
displays the simplest situation. The full leaf $E_1$
is contained in $\partial(D_1 \cup D_2)$, where
$D_1, D_2$ are adjacent lozenges intersecting 
a common stable leaf (an example is the leaf $E$).

We want to prove that $\alpha := Y \cap E_1$ is the periodic
orbit in $E_1$. Suppose it is not. Let $\cB$ be the
basic block of $\E$ intersecting $E$. This block
has half leaves of $X$ and $Y$ contained in its boundary.
The properties above imply there is a stable leaf
$S$ making a perfect fit with $X$ and entirely contained
in $\partial \R$.
It follows that $\cB$ is a stable adjancent block with more
than one element. There is a half leaf of $E_1$ contained
in the boundary of the region $\cB$, and also a half 
leaf of $Y$ in this boundary. It follows that $\alpha$
is a hinge of $\cB$.
The other corner $\eta$ of $\cB$ is in $X$.
Let $\nu$ be the periodic orbit in $X$.
By assumption $\alpha$ is not periodic. Then either
\ 1) $\wls(\eta)$ intersects $\wlu(\alpha)$
producing a corner of $\cB$ in $\wlu(\alpha)$ besides
$\alpha$ $-$ a contradiction; \ 
or \ 2) $\wls(\nu)$ separates $\wls(\eta)$ from a stable
leaf in the boundary of $\cB$ which makes a perfect
fit with $\wlu(\alpha)$ $-$ also a contradiction.

We conclude that $\alpha$ is the periodic orbit in $E_1$.
It follows that all corners of $\C$ are periodic
and hence so are the corners of $\E$. 
Notice that $E_0, E_1$ are non separated from each other in
the leaf space of $\wls$.
Hence there is a non trivial deck transformation $h$ 
leaving both $E_0, E_1$ invariant. In addition the corners
of $\E$ are connected to the corners of $\cC$ by chain of
lozenges. In other words
$\alpha$ and $g(\alpha)$ are 
connected by a chain of lozenges.
As in the proof of Case 0 of Lemma \ref{sameband}
this implies 
that $g, h$ generate a ${\bf Z}^2$ subgroup of $\pi_1(M)$.
This is impossible, since $M$ is atoroidal.
This shows this case cannot happen.

In the same way one rules out the case that 
$E_0$ intersects $U$ by switching the roles of $W$ and $g(W)$.

The next possibility in this case is that $E_1$ intersects
$X$ and $E_0$ intersects $U$. This case is exactly the same as
the previous one by switching $W$ with $g(W)$.

Finally the remaining possibility in this case is
that $E_1$ intersects $X$ and $E_0$ intersects 
$V$. Then an analysis as in the situation that $E_0 < E_1$,
shows that $E_1 \cap X$ is a 
hinge of $g(W)$ and $E_0 \cap V$ is a hinge of 
$W$. As in that analysis this leads to the contradiction 
that $g$ preserves all orientations in $\mt$.
This shows that Case 2.1 cannot happen.

The remaining possibility to be analyzed is:

\vskip .05in
\noindent
{\bf {Case 2.2 $-$ $E_0 = E_1$.}}

Given that $E_0$ intersects one of $U,V$ but not both,
and one of $X, Y$ but not both there are two possibilities: 
\ i) $E_0$ intersects $U$ and $Y$, but does not intersect
either of $X$ or $V$,  or
\ ii) $E_0$ intersects both  $X, U$ but not none of $Y, V$;
or $E_0$ intersects both $Y, V$, but none of $X, U$.

\begin{figure}[ht]
\begin{center}
\includegraphics[scale=1.00]{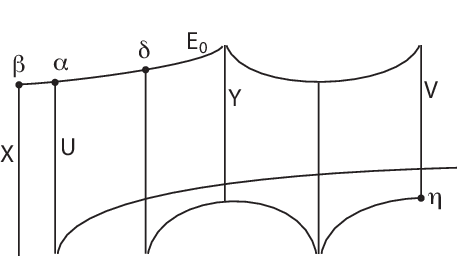}
\begin{picture}(0,0)
\end{picture}
\end{center}
\vspace{0.0cm}
\caption{The depiction of the leaves in Case 2.2.}
\label{figure22}
\end{figure}

We deal with possibility ii) First. Up to switching $W$ 
with $g(W)$ assume that $E_0$ intersects both 
$X, U$, but does not intersect either of $Y, V$. 
We depict this is in Figure \ref{figure22}.
The argument will be very similar to the one in Situation 2
of Lemma \ref{boundary}
producing two different walls with a common corner orbit,
which will be a contradiction.
Hence we will only review the main points.
The intersections $\alpha := W \cap E_0$, 
$\beta := g(W) \cap E_0$ are both tangent to the
flow and hence orbits of $\wwp$.
The leaf $V$ makes a perfect fit with a stable leaf 
which is distinct from $E_0$, but non separated from $E_0$.
It follows that $E_0, Y, V$ are all periodic, and left invariant
by a non trivial deck transformation $h$.
Let $\delta$ be the periodic orbit of $\wwp$ in $E_0$ and
$\eta$ the periodic orbit in $V$. We assume $h$ acts
pushing points backwards in $\delta$.
The basic block of $\C$ intersecting $E$ is necessarily
a stable adjacent block with $\alpha$ one corner orbit.
Either $\alpha = \delta$ and the other corner of this
basic block is $\eta$ or the other corner orbit is 
an orbit in $V$ different from $\eta$. In any case $h^n(W)$ converges 
(as $n \to \infty$) to 
a wall $W_0$ which has both $\delta$ and $\eta$ as corner
orbits.
These are the only corner orbits in this stable adjacent block.
Now apply $h^n$ to $g(W)$. Then $h^n(\beta)$
also converges to $\delta$, so $h^n(g(W))$ converges
to a wall $W_1$ with $\delta$ as a corner orbit as well.
Either $g(W)$ does not intersect $Y$ and hence $W_1$ does
not intersect $V$, or $g(W)$ intersects $Y$ in
an orbit, and then $W_1$ intersects $Y$ in the periodic orbit
of $Y$. In either case $W_0, W_1$ are distinct walls.
This contradicts the result of Proposition \ref{identical}.

\vskip .05in
Finally suppose that $E_0$ intersects both $U, Y$, but none
of $X, V$. Hence $X$ makes a perfect fit with a stable
leaf non separated from $E_0$ (could be $E_0$ itself).
The same happens for $V$.
This is very similar to Situation 1 in Lemma \ref{boundary}.
It follows that $X, V, E_0$ are all periodic, left invariant
by a non trivial deck transformation $h$, which we assume
is a generator of the isotropy group of this set of leaves. In addition
let $\delta$ be the periodic orbit in $E_0$, and assume that
$h$ acts pushing points backwards in $\delta$.
Applying $h^n$ one obtains that $h^n(W), \ h^n(g(W))$
converge to distinct walls $W_0, W_1$ both containing $\delta$
$-$ as in Situation 1 of Lemma \ref{boundary}.
This is a contradiction and shows that this cannot happen.

\vskip .05in
From this analysis we conclude that Case 2 cannot happen,
and therefore Case 1 happens always.
In particular this proves the Claim.
Once we hit a corner of one of $W$ or $g(W)$ $-$ without
loss of generality say it is a corner of $W$ which is
then in $E_0$. Then when we move slightly past $E_0$
in $\X_0$ to a leaf $F$ of $\wls$, we obtain that $F$ intersects
both $W, g(W)$, the intersections are transverse
and no intersection is a flow line. We can keep proceeding
in the elements of $\cC$ and $\cE$ and the process never
ends, and $\X_0$ never separates from $\X_1$.
This implies that 
that $\X_0 = \X_1$.  Hence $W, g(W)$ intersect
the same set of stable leaves. 
In addition if $F$ leaf of $\wls$ intersects
$W$, then the geodesics $F \cap W$ and $F \cap g(W)$
in $F$ intersect each other and transversely.
So the intersections are always linked.

Let $\Z_0$ (resp. $\Z_1$) be the set of unstable leaves intersected
by $W$ (resp. $\gamma(W)$).
The arguments show that
for any stable leaf $E$ of $\wls$
 intersected by $W$ then $E \cap W$, \ $E \cap g(W)$
are transverse geodesics in $E$.
Hence any corner orbit $\nu$ of $g(W)$ intersects
a basic block of $\C$. Therefore $\wlu(\nu)$ is in $\Z_0$.
This implies that $\Z_1 \subset \Z_0$. Switching the roles
of $W, \gamma(W)$ it follows that $\Z_0 \subset \Z_1$.

This proves that every corner orbit of $\cC$ is in the interior
of an element in $\cE$ and vice versa. It now follows
immediately that all the elements of $\cC$ are lozenges
and $\cC$ is a string of lozenges.

This finishes the proof of Proposition \ref{linking}.
\end{proof}

\section{The convex hull}

Next we introduce the set $\V$ which will be a form of a convex
hull associated with a collection of walls which intersect
transversely. Fix a wall $W$. By our standing assumption
there is $g$ in $\pi_1(M)$ so that $W, g(W)$ intersect
transversely at some point. Let

$$\PP \\ = \ \ \{ g \in \pi_1(M), \  {\rm so \ there \ are} \
g_0 = id, g_1,...,g_n = g,$$

$$ \ {\rm and} \ g_i(L), \ g_{i-1}(L) \ \  {\rm 
are \ either \ equal \ or \ intersect
\ transversely \ for \ all} \ 1 \leq i \leq n \}$$

\noindent
Let also

$$\G \ \ = \ \ \{ \ g(W), \ {\rm so \ that} \ g \in \PP \ \}$$

First we establish a basic lemma:

\begin{lemma}
For any $g$ in $\PP$, the walls $W, \ g(W)$ intersect
exactly the same set of leaves of $\wls$ and $\wlu$.
\end{lemma}

\begin{proof}
Let $g$ in $\PP$ and 
$g_0 = id, g_1,..., g_n = g$
so that $g_{i-1}(W), \ g_i(W)$ intersect transversely.
By Proposition \ref{linking} for any $i$, the walls 
$g_i(W), \ g_{i-1}(W)$ are linked and intersect exactly
the same set of stable and unstable leaves. 
The lemma follows by induction.
\end{proof}

\vskip .1in
\noindent
{\bf {Strategy for the completion of the proof of the Main theorem.}}

By the previous lemma any wall in $\G$ intersects exactly the
same set of stable leaves, and this set is homeomorphic 
to a properly embedded copy of the reals in the stable leaf space.
The goal is to show that this is all of the leaf space of $\wls$ $-$
which immediately implies that $\ls$ is $\rrrr$-covered and finishes
the proof of the Main theorem.

To prove this property we introduce the convex hull $\V$ of $\G$
as follows.

Let $E$ be a leaf of $\wls$ intersecting some wall in $\G$. Hence
there is $\beta$ in $\PP$ so that $E$ intersects $\beta(W)$.
By the previous lemma,
 for any $\gamma$ in $\PP$ then $E$ intersects $\gamma(W)$.
We first define the convex hull in $E$. Let

$$\V_E \ \ = \ \ \{ {\rm convex \ hull \ in} \ E \ {\rm of \
the \ union \ of} \ 
E \cap g(W), \ g \in \PP \}.$$

Recall that the convex hull of a set $A$ in 
$\mathbb{H}^2$ is the smallest convex set
which contains all points in $A$, or equivalently the
union of all segments with endpoints in $A$.
Finally let 

$$\V \ \ = \ \ \bigcup \ \{ \ \V_E, \ E \in \wls, \ E \cap \gamma(L) \not
= \emptyset, \ {\rm for} \ \gamma \in \PP \}$$

\noindent
We stress the following: \ \ 

\vskip .1in
{\bf {The wall $W$ is fixed throughout the discussion 
in this section.}}

\begin{define} (convex hull of $W$)
The set $\V$ is called the {\em convex hull of $W$}. 
\end{define} 

Our goal
is to show that under the
New assumption of Section \ref{transverse1}
the set $\V$ is all of $\mt$ and hence $\ls$ is
$\rrrr$-covered.
In fact we will show that if $\V$ is not all of $\mt$, then
the New assumption fails and we can use the previous
analysis.

\begin{remark} Notice that $\V_E$ is open for any $E$.
Otherwise for some $E$, $\V_E$ is not open, but since it
is convex, it follows that some boundary component of $\V_E$
is contained in $\V_E$. This boundary component
is a geodesic $\ell$ in $E$.
This can
only occur if $\ell$ is $g(W) \cap E$ for some $g$
in $\pi_1(M)$. But by the condition on $W$, $\ell$ transversely
intersects some $h(W) \cap E$. It follows that $\ell$ would not
be in the boundary of $\V_E$. 
Hence $\V_E$ is open.
\end{remark}

\begin{define} A set $P$ in $\mt$ is precisely invariant if for
any $h$ in $\pi_1(M)$ such that $h(P)  \cap P$ is not
empty then $h(P) = P$.
\end{define}

\begin{lemma} \label{precise}
The set $\V$ is precisely invariant.
\end{lemma}

\begin{proof}
Let $g$ in $\pi_1(M)$ so that $g(\V) \cap \V \not = \emptyset$,
so there is $x$ in the intersection. Let $E$ be the stable leaf
through $g^{-1}(x)$ $-$ a point in $\V$ as well.
Hence $g^{-1}(x)$ is in $E \cap \V$. 
It follows that $g(E \cap \V)$ intersects $g(E) \cap \V$.

Recall that $\V \cap E$ is the the convex hull 
in $E$ of the
set of geodesics $h(W) \cap E$, where $h \in \PP$.
Similarly $g(E) \cap \V$ is the interior of the convex
hull of $h(W) \cap g(E)$ where $h \in \PP$.
Since $g(E \cap \V)$ intersects $g(E) \cap \V$,
it follows that there are $h \in \PP$, $f \in \PP$
so that $g h(W)$ intersects $f(W)$
transversely.

$\G$ is the set of deck translates of $W$ so that
there is a finite sequence of deck translates of $W$ starting
with $W$, and consecutive ones are either 
equal or intersect transversely.
Hence $g$ sends some element $h(W)$ of $\G$ to
an element of $\G$.

Let now $f(W)$ ($f \in \PP$) be a generic element of $\G$. Either
$f(W) = W$ or there is
a sequence of deck translates of $W$ from $h(W)$ to
$f(W)$, so consecutive ones intersect transversely.
Just take the inverse of a sequence from $W$ to $h(W)$,
and then a sequence from $W$ to $f(W)$. Applying
$g$ one gets a sequence of consecutively transverse
translates from $g h(W)$ to $g f(W)$.
Since $g h(W)$ is an element of $\G$,
there is a sequence of consecutively
transverse translates from $W$ to $g h(W)$. Again
concatenating with the sequence from $g h(W)$ to $g f (W)$ one
produces a sequence from $W$ to $g f (W)$. 
Hence $g f (W)$ is in $\G$.

In particular $g(\G)$ is a subset of $\G$. In the same way
$g^{-1}(\G)$ is also a subset of $\G$, and consequently
$g(\G) = \G$. 
Since $\V$ is obtained by taking 
the convex hull of $f(W) \cap E$
with $f(W)$ in $\G$, it follows that $g(\V) = \V$.

This proves Lemma \ref{sameband}.
\end{proof}

%
%
%

This result will be used in the next section.
An immediate consequence is:

\begin{corollary} The subset $\PP$ of $\pi_1(M)$ is a subgroup
of $\pi_1(M)$.
\end{corollary}

\begin{proof}{}
If $g \in \PP$, then by definition $g(\V) \cap \V 
\not = \emptyset$. Conversely if $g(\V) \cap \V \not =
\emptyset$, then $g(W)$ is connected to $W$ by a sequence
of translates of $W$ so that consecutives intersect
transversely. If follows that $g \in \PP$ if and only
if $g(\V) \cap \V \not = \emptyset$. The result
is now immediate from the previous lemma.
\end{proof}

We now prove a further property:

\begin{lemma} The set $\V$ is open.
\end{lemma}

\begin{proof}
Let $p$ in $\V$. Then $p$ is in a leaf $E$ of $\wls$. There are finitely
many $g_1, ..., g_k$ in $\PP$ so that $p$ is in the
interior of the convex hull $c_E$ in $E$ of 

$$\bigcup_{i = 1}^k \ (g_i(W) \cap E)$$

\noindent
There is $\eps > 0$ so that a disk of radius $2\eps$ centered
at $p$ is contained in the interior (in $E$) of $c_E$.
Let 

$$\ell_i \ = \ g_i(W) \cap E, \ 1 \leq i \leq k$$

\noindent
Let $\tau = \tau(t), -1 < t < 1$ be an open transversal to $\wls$
with $c(0) = p$.
Let $E_t$ be the stable leaf through $\tau(t)$. By reducing
$\tau$ if necessary, assume it only intersects leaves of $\wls$
which intersect $W$ (hence intersect $g(W)$ for any
$g$ in $\PP$).

By the properties of walls, for any $1 \leq i \leq k$ and
if $t \rightarrow 0$, then $g_i(W) \cap E_t$ converges
to $\ell_i$ in $\mt$ as $t \rightarrow 0$.
Therefore for $t$ sufficiently small the interior of the convex hull
of

$$\bigcup_{i = 1}^k \ (g_i(W) \cap E_t)$$

\noindent
contains an $\eps$ disk in $E_t$ centered at $\tau(t)$. This shows
that $\V$ contains an open set in $\mt$ containing $p$.
Hence $\V$ is open as desired.
\end{proof}

\section{Boundary components of $\V$}

The wall $W$ is the same fixed in the last section.
The set $\V$ is obtained from $W$.
As mentioned before we eventually want to show that $\V = \mt$. The proof 
is by contradiction. We assume that $\V \not = \mt$
and eventually prove that the New assumption fails. 
We first analyze properties of 
boundary components of $\V$.

Let $\A$ be the the lef space of $\wls$ and let 
let $\A_0$ be the subset of $\A$ 
which is the set of stable leaves intersecting the wall $W$.
Recall that $\A_0$ is a properly embedded copy of 
$\rrrr$ in $\A$. 
First we prove the following very useful result:

\begin{lemma} \label{total}
Let $H$ be an arbitrary wall in $\mt$. Let $\eta$ 
be an intersection of $H$ with a leaf $E$ of $\wls$.
Suppose that one ideal point of $\eta$ is the negative
ideal point of $U \cap E$, where $U$ is a leaf in 
$\wlu$. Then every leaf $G$ of $\wls$ that
intersects $U$, also intersects $H$.
In other words if $U$ is an unstable leaf which is
part of the boundary of some element in
the bi-infinite block associated with $H$, then 
every stable leaf intersecting $U$ intersects $H$.
\end{lemma}

Notice that we do not require that $H$ is obtained by 
limits of partial walls associated with strings of lozenges.

\begin{proof}
Let $\cC$ be the bi-infinite block associated with $H$.
Clearly the set of stable leaves which intersects both $H$ and
$U$ is an open set. 
So we may assume that both ideal points of $\eta$ are
negative ideal points in the stable leaf
containing it. Let $C$ be the element of
$\cC$ containing $\eta$. 
Parametrize the stable leaves intersecting $U$
as $E_t,  t \in \rrrr$, where $E_0 = E$.
We will consider $t > 0$.
As long as $E_t$ intersects $C$ then $E_t$ intersects
$U$ and $E_t$ also 
intersects $H$.
The first possibility
is that this happens for all $t > 0$ and we get that
any stable leaf intersecting a half leaf of $U$ in that component of
$\mt - E$ also intersects $H$.

Otherwise there is a smallest $t_0 > 0$ so that 
$E_{t_0}$ does not intersect $C$. 
Notice that $E_{t_0}$ intersects $U$. This can only
happen if $E_{t_0} \cap U = \beta$ is a corner of $C$.
Consider the next element $D$ of $\cC$ which also
has a corner $\beta$. It could be that $\beta$ is
not a hinge, hence $C, D$ are adjacent along
a stable leaf (which is $\wls(\beta)$). 
In that case for any  $t > t_0$ sufficiently
close to $t_0$ then $E_t$ intersects $D$, and hence $H$,
and $H \cap E_t = \eta'$ has an ideal point which
is the negative ideal point of $E_t \cap U$. 
In this case we restart the analysis with the element
$D$ of $\cC$ instead of $C$.
Otherwise $\beta$ is a hinge. 
But the same conclusion happens.

The leaf $U$ can intersect at most one hinge, and the
number of corners in $U$ until hitting  a hinge
is finite.
This implies the result of the lemma.
\end{proof}

Notice that the result only holds for those $U$ which 
are part of boundary of the elements of $\cC$. For
an unstable leaf $X$ intersecting the interior of one
of those elements, the result will follow
from the Main theorem, but it needs the atoroidal
hypothesis.

\begin{proposition} \label{boundaryplane}
Let $H$ be a boundary component of $\V$, then $H$ is a properly
embedded plane transverse to $\wls$.
In addition the leaf space of $\wls \cap H$ is the reals.
\end{proposition}

\begin{proof}
This result is reminiscent of Lemma \ref{band2}.
But here we do not know a priori what are the
ideal points of intersections of $H$ with leaves of $\wls$,
in particular that they follow the negative ideal points
of intersections of some fixed unstable leaves with
varying stable leaves. This makes this proof more involved.

Let $E$ be a leaf intersecting $H$, in particular $\V \cap E$ is not
all of $E$. Let $\ell$ be a boundary component of $\V \cap E$.
Fix a transversal $\tau = \{ c(t), -\eps < t < \eps \}$
to $\wls$ with $c(0)$ in $\ell$. Let $E_t$ be the leaf of $\wls$
containing $c(t)$. If necessary, shorten $\tau$ so that
all leaves $E_t$ intersect $H$.
The geodesic $\ell$ in $E$ ($= E_0$)  has ideal points $p, q$ in $S^1(E)$.

\vskip .1in
\noindent
{\bf {Case 1 $-$ Suppose first that neither $p$ nor $q$ is the
positive ideal point of the flow lines in $E$.}}

Then for any $|t|$ sufficiently small $E_t$ is asymptotic to 
$\ell$ in both directions: this means that for each
ray $r$ of $\ell$, then $r$ is asymptotic to $E_t$.
This defines unique points $p_t, q_t$ in $S^1(E_t)$
corresponding to $p, q$ in $S^1(E)$: they are the
ideal point of asymptotic rays in $E_t$.
Let $\ell_t$ be the geodesic in $E_t$ with ideal points
$p_t, q_t$. We claim that $\ell_t$ is contained in the
boundary of $\V$ as well for $|t|$ sufficiently small.
In particular we show that $\ell_t$ is contained in the same
component of the boundary that contains $\ell$.

There is a single ideal point in $E$ which is 
the forward ideal point of all flow lines.
Let $I$ be the complementary component of $\{ p, q\}$ in
$S^1(E)$ which does not contain any ideal point
of $g(W) \cap E$ where $g$ is in $\PP$.
Let $I_t$ be the open interval in $S^1(E_t)$ with 
ideal points $p_t, q_t$ and $I_t$ close to $I$ in
the local trivialization of circles at infinity.
Suppose first that the positive ideal point of all flow lines
in $E$ is not in $I$.

Now for any $g \in \PP$ it follows that $g(W) \cap E_t$
is contained in the component of $E_t - \ell_t$ on the same
side that $g(W) \cap E$ is relative to $\ell$.
The reason is as follows:
if $g(W) \cap E_t$ had an ideal point in $I_t$
then $E_t$ is asymptotic to $E$ in that direction.
Since $g(W)$ is a wall, this can only happen
if $g(W) \cap E$ is asymptotic to $g(W) \cap E_t$.
In other words $g(W) \cap E$ has an ideal point in $I$,
contradiction.
Hence $\V \cap E_t$ is contained in the corresponding component
of $E_t - \ell_t$.

Finally we show that $\ell_t$ is in the closure of $\V \cap E_t$.
Suppose first that there is $g$ in $\PP$ with say $p$ 
being an ideal point of $g(W) \cap E$ $-$ this is forcibly
a negative ideal point in $E$.
Then $g(W) \cap E_t$ also has ideal point $p_t$. 
Similarly for $q_t$. Otherwise there is a sequence 
$\alpha_i \in \PP$ with $\alpha_i(W) \cap E$ having ideal
points converging to $p$. Then $\alpha_i(W) \cap E_t$
also has ideal points converging to $p_t$.

In any case there is $\eps > 0$ so that
for $|t| < \eps_0$, there are $v_i$ converging
to $p_t$ in $E_t \cup S^1(E_t)$
with $v_i$ in $\V \cap E_t$ and $u_i$ converging
to $q_t$ with $u_i$ in $\V \cap E_t$.
The geodesic arc from $v_i$ to $u_i$ in $E_t$ is contained
in $\V \cap E_t$ and converges to $\ell_t$ with $i$. 
This shows that $\ell_t$ is contained in the closure 
of $\V$ so it is contained in the boundary of $\V$.

\vskip .05in
These arguments deal with the possibility that the forward ideal point of
$E$ is not in $I$. If the forward ideal point
is in $I$, then we consider 
the other component $J$ of $S^1(E) - \{ p, q \}$ and
let $N$ be a small open neighborhood of $J \cup \partial J$.
Then all points in $J$ are negative ideal points, a
similar proof as above shows the same result that $\ell_t$
is in the boundary of $\V \cap E_t$.
In addition
in Case 1 notice that $\ell_t$ and $\ell$ are a bounded
distance apart in $\mt$.
Notice also that 
for any $|t_0| < \eps_0$ and $t \to t_0$, then $\ell_t$
converges to $\ell_{t_0}$ and only to $\ell_{t_0}$.

\vskip .05in
\noindent
{\bf {Final property in Case 1 $-$}}

We note the following:
let $\gamma_t$ be the flow line in $E_t$ whose negative
ideal point is $p_t$.
The $p_t$ correspond to directions in
$E_t$ which are asymptotic to $\ell$ in the direction
of $p$. Hence all $\gamma_t$ are in
the same unstable leaf $U$.
Likewise for $q_t$.

This finishes the analysis of Case 1.

\vskip .05in
\noindent
{\bf {Case 2 $-$ Suppose that (say) $p$ is the forward
ideal point of flow lines in $E$.}}

This case is more complex than Case 1. 
We recall (see eg \cite{Cal4}) the topology in
$\Z = \cup \{ S^1(E_t), -\eps < t < \eps \}$
(these $E_t$ are exactly those intersecting $\tau$).
Fix $t$. For each $x$ in $S^1(E_t)$ there is a unique
geodesic ray starting at $\tau \cap E_t$ with ideal point
$x$. This produces a bijection between $\Z$ and
$\cup \{ T^1 \wls|_{\tau} \}$. This last set has a natural topology
$-$ that is, independent of the transversal used $-$
as an annulus and this is the topology we put in $\Z$.

The ideal points in $S^1(E)$ of $\ell$ bound a unique interval
$I$ which is disjoint from all ideal points of $\gamma(L) \cap E$
for $\gamma \in \PP$.
We first prove a simple claim:

\begin{claim}
For any $x$ in $I$ there is no sequence $x_i$ of ideal
points of $g_i(W) \cap E_{t_i}$ with $t_i \rightarrow 0$,
$g_i \in \PP$ and $x_i$ converging to $x$.
\end{claim}

\begin{proof}
Notice that $g_i(W)$ intersects $E$ for all $i$.
Suppose by way of contradiction that there is such a sequence $x_i$.
Since $x$ is in $I$, $x$ is a negative ideal point in $S^1(E)$
$-$ because $p$ is the positive ideal point of the
flow lines in $E$. 
Then $x$ has an open neighborhood in $S^1(E)$ where all 
points correspond to
directions asymptotic with $E_t$ for all $|t| < \eps_1$, and
 for a {\underline {fixed}}
$\eps_1 > 0$. Using the trivialization above this implies
that $g_i(W) \cap E$ has ideal points in $I$ for $i$
big enough. This is a contradiction.
This proves the claim.
\end{proof}

Suppose that $E_t$ intersects $\V$ and let $Z$ be
a complementary region in $E_t$ of $\V \cap E_t$.
Then $Z$ is a hemisphere in a hyperbolic plane, including
the boundary geodesic. The ideal boundary of $Z$
is the set of accumulation points of $Z$ in $S^1(E_t)$.
By the claim there is a unique complementary region
of $\V \cap E_t$ with ideal boundary an interval in $S^1(E_t)$
which contains an interval close to the interval $I$
in the trivialization above. Let $\ell_t$ be the boundary
of this complementary region of $\V \cap E_t$ with
ideal points $p_t, q_t$. 

We first show that $p_t \to p$ and $q_t \to q$ as $t \to 0$.
Suppose by way of contradiction that there are
$p_{t_i} \to z \not = p$ and $t_i \to 0$. For simplicity
assume $t_i > 0$ for all $i$.
There is $w \in [p,z)$ which is an ideal point
of $g(W) \cap E$ for some $g \in \PP$, since
$\ell$ is the boundary of the complementary region
associated with $I$.
In Lemma \ref{band2} it was proved that ideal
points $g(W) \cap F$ in $S^1(F)$
vary continuously with $F \in \wls$ along an interval of
leaves.
So the corresponding ideal points of $g(W) \cap E_{t_i}$
as $i \to \infty$ 
show that $p_{t_i}$ cannot converge to $z$.

\vskip .05in
Now we show that $\ell_t$ converges to $\ell$ and
only to $\ell$.
Since the ideal points of $\ell_t$ converge to the ideal
points of $\ell$, then the local trivialization of
$\Z$ and continuous change of hyperbolic metrics show
that $\ell_t$ converges to $\ell_t$.

Finally we prove that $\ell_t$ converges only to $\ell$.
This is much trickier.
Let $x_i$ in $\ell_{t_i}$ converging to $x$,
and so that $t_i \to 0$.
If $x$ is in $E$, then the argument of Lemma \ref{band2}
implies that $x$ is in $\ell$.
Suppose that $x$ is in $F$ different from $E$.
In particular $F$ is non separated from $E$.
Since $x_i$ converges to $x$ in $F$ then up to subsequence
we may assume that $\ell_{t_i}$ converges to a geodesic
$\eta$ in $F$. One of the ideal points of $\eta$, call
it $y$ is a negative ideal point in $F$.

We now consider a transversal $\tau'$ to $\wls$ through
$x$ in $F$, and parametrize the leaves of $\wls$ intersecting
$\tau'$ as $F_t, -\eps \leq t \leq \eps$, with $F_0 = F$.
Reducing $\tau, \tau'$ and reparametrizing if necessary,
we may assume that $F_t = E_t$ if $0 < t <  \eps$.
 Notice again that $F$ is non
separated from $E$. Then put in 
$\Y  = \cup \{ S^1(F_t), -\eps < t < \eps\}$ the topology
coming from 
\  $\cup \ \{ \ T^1 \wls|_{\tau'} \ \}$.
Since $y$ is a negative ideal point of $F$ the following
happens: $y$ has a 
neighborhood in $S^1(F)$ consisting only of negative
ideal points.
In addition $y$ it has a neighborhood $N$ in $\Y$
consisting only of negative ideal points and
satisfying: there is $\eps_0 > 0$ so that 
for $F_t$ with $|t| < \eps_0$ if $z$ is an ideal point
in $S^1(F_t)$ with $z$ in $N$  and $U$ is the unstable
leaf so that $U \cap F_t$ has negative ideal point
$z$, then $U$ intersects $F_0$ also.

In the local parametrization given by $\Y$, then one
of the ideal points of $\ell_{t_i}$ $-$ call it $p_{t_i}$ $-$
satisfies that $p_{t_i}$ converges
to $y$
 (notice that this is in $\Y$ and not in $\Z$, this
is very different from the 
topology in $\Z$,
 since $E, F$ are non separated from each
other).
 Then the following happens:
 fix one $i$ big enough so that 
$p_{t_i}$ is in $N$ (again we are considering
$\tau'$ and $\Y$ here).
Then since $\ell_{t_i}$ is a boundary
component of the convex hull $\V$ in $E_{t_i}$ (which is
equal to $F_{t_i}$ for $i$ big)  it follows that
there is $g \in \PP$ so that $g(W) \cap E_{t_i}$ 
has an ideal point in $N$.
The important fact here is that $g$ is in $\PP$.
In particular there is $U$ unstable leaf so
that this ideal point of $g(W) \cap E_{t_i}$ is
the negative ideal  point of $U \cap E_{t_i}$.

By the defining property of $N$, it follows that
$U$ intersects $F$. By Lemma \ref{total} it follows
that $g(W)$ intersects $F$ also.

This is a contradiction: $g(W)$ intersects $F$ and
also intersects $E$ $-$ the second fact is because
$g$ is in $\PP$. But $E, F$ are non separated and
distinct from each other. 
This is impossible because $g(W)$ is a wall and
intersects an $\rrrr$ worth of leaves of $\wls$.

This shows that the only limits of $\ell_t$ are in $\ell$.
This finishes Case 2.

\vskip .05in
In addition since $\ell_t$ only converges to $\ell$ in
both cases, it follows that the leaf space of $\wls \cap W$
is Hausdorff, and hence it is the reals.

This completes the proof of Proposition \ref{boundaryplane}.
\end{proof}

We also need the following result.

\begin{lemma} \label{atmostcyclic}
Suppose that $M$ is atoroidal.
Let $H$ be an arbitrary component of $\partial \V$.
Then the isotropy group of $H$ is at most infinite
cyclic.
\end{lemma} 

\begin{proof}
The proof is entirely two dimensional.
Let $\K$ be the isotropy group of $H$. 
Then $H/\K$ is a surface.
Suppose by way of contradiction that $\K$ is not cyclic.
What we will prove is that this implies that $H/\K$ is
a torus or a Klein bottle. Since $H$ is a properly embedded plane,
this implies that $H/\K$ is $\pi_1$-injective,
so there is a $\mathbb Z^2$ subgroup of $\pi_1(M)$
which is a contradiction.

The group $\K$   acts on the leaf
space of the foliation $\fol := \wls \cap H$, which is homeomorphic
to the reals $\rrrr$ as proved in Proposition
\ref{boundaryplane}. By assumption, $\K$ is not cyclic.
If it is abelian, then 
since $H/ \K$ is a surface it follows that
$\K \cong \mathbb Z^2$ and we are done.

Parametrize the leaf space of $\fol$ as 
$\{ \ell_t, \ t \in \rrrr \}$.
If $\K$ is not abelian, then it does not
act freely 
on the leaf space of $\fol$, 
and let $g$ in $\K$ non trivial with a fixed leaf of $\fol$.
We assume $g$ indivisible, and up to changing the parametrization,
assume that $g(\ell_0) = \ell_0$.
Using that $\K$ is not cyclic, 
let $h$ in $\K$ which is not in the subgroup generated by $g$.

Let $t_0 = \sup \ \{ t \in \rrrr, \ g(\ell_t) = \ell_t \}$,
where $t_0$ can attain the value $\infty$.
We claim that $t_0$ is $\infty$. Suppose not. 
Since the leaf space of $\fol$ is $\rrrr$ it follows
that $g(\ell_{t_0}) = \ell_{t_0}$.
Let $s > t_0$. Then up to taking an inverse of $g$,
it follows that $g^n(\ell_s) = \ell_{s_n}$ with $s_n$ converging
to $t_0$ as $n \to \infty$. This implies that in $H/<g>$
the projection of the curve $\ell_s$ spirals towards
a closed curve $\ell_{t_0}/<g>$.
In particular it follows also that in $H/\K$ the projection
of $g_s$ is a non closed curve.

Up to taking an inverse of $h$, assume that $h(\ell_{t_0})
= \ell_s$ with $s > t_0$.
Then since $h(\ell_{t_0})$ is a deck translate of a curve
that projects to a closed curve in $H/\K$ it follows that
it also projects to a closed curve in $H/\K$. But since
$s > t_0$, we just proved that $h(\ell_{t_0})$ does not
project to a closed curve in $H/\K$. This contradiction
shows that $t_0 = \infty$.

Again up to taking an inverse of $h$ assume that
$h(\ell_0) = \ell_v$ with $v > 0$. 
The quotient $H/<g>$ is an annulus. 
Let $t_1 > v$ so that $g(\ell_{t_1}) = \ell_{t_1}$.
The projection of $\{ \ell_t, \ 0 \leq t \leq t_1 \}$
to $H/<g>$ is a compact subannulus $B$ of $H/<g>$.
The projection of $h(\ell_0) = \ell_v$ to this compact
subannulus $B$ of $H/<g>$ is a closed curve $-$ we already used this 
argument above. 
Now $h(\ell_0)$ and $\ell_0$ project in $H/\K$  to the 
same closed curve. It follows that 
the projection of the region in $H$ between 
$\ell_0$ and $h(\ell_0)$ to $H/\K$ is a closed surface,
and hence $H/\K$ is a closed surface. 
Since $H/\K$ has a one dimensional foliation
$\fol/\K$ it follows that $H/\K$ is either a torus or a Klein bottle.
This finishes the proof of Lemma \ref{atmostcyclic}.
\end{proof}

\begin{proposition} \label{asymptotic2}
Let $H$ be a component of $\partial \V$.
For any stable leaf $E$ intersecting $H$ and any 
ray $r$ in $H \cap E$ there is a
boundary component $G$ of $\partial \V$,  intersecting 
$E$, with $G$ distinct from $H$ and $G \cap E$ asymptotic
to the ray $r$.
In addition the isotropy group of $H$ is infinite cyclic.
\end{proposition}

\begin{proof}
For example it could be that $E \cap \V$ is a finite sided
ideal polygon. However, a priori the statement allows for
infinitely many components of $\partial (E \cap \V)$ $-$
in that case each such component would be part of an
infinite chain of boundary components so that consecutive
ones have asymptotic rays.

We start by proving the first assertion.
The proof is by contradiction. We assume that $H, E, r$ are
as in the statement of the proposition, and fail the
conclusion of the proposition.
Hence $r$ is not asymptotic to any other boundary
component of $\V \cap E$.

Any ideal triangle in the hyperbolic plane
has inradius $> a_1 > 0$. Since $r$ is not asymptotic
to any other ray $r' \subset \partial \V \cap E$, there
are points $p_i$ in $r$ escaping the ray and points $q_i$
in $\V \cap E$ so that the orthogonal
geodesic to $r$ in $E$ at $p_i$ intersects
$q_i$ at distance exactly $a_1$ and the segment
between $q_i$ and $p_i$ (except for $p_i$) is contained in $\V$.

Recall from Lemma \ref{precise} that $\V$ is precisely
invariant.

Up to subsequence and deck translations $g_i$ assume 
that $g_i(p_i)$
converges to $p_0$, and $g_i(q_i)$ converges to $q_0$.
We claim that there is $j > 0$ fixed so that
$g_i(\V) = g_j(\V)$ for all $i > j$
and $g_i(H) = g_j(H)$.
This is because we can assume
that $g_i(q_i)$ is very close to $g_j(q_j)$ for
all $i > j$ ($j$ is fixed, 
and hence $g_j(q_j), g_j(\V)$ are fixed).
This implies that $g_i(\V) = g_j(\V)$
for all $i > j$
because  $\pi(q_0)$ is in $\pi(\V)$: otherwise since
it is accumulated by elements $\pi(q_i)$ which are in $\pi(\V)$,
it follows that $q_0$ is in $\pi(\partial V)$. But the
projection of this boundary component
of $\V$ would intersect the disks in $\pi(g_i(\V))$ with
center at $\pi(q_i)$ with radius $a_1$ for $i$ big,
contradiction.
We are using the precise invariance of $\V$ in this argument.

We change our starting point with $g_j(\V)$ 
instead of $\V$ as follows. Let

$$\U = g_j(\V), \ \ T = g_j(H), \ \ b = g_j(p_j), \ \ 
h_i = g_i \circ g_j^{-1}, \ \ \eta_0 = g_j(H \cap E).$$

\noindent
Since $d(p_i,p_j)$ converges to infinity with $i$, then up
to subsequence it follows that the $\{ h_i \}$ are pairwise
distinct.
The set $T$ is a boundary component of $\partial \U$. 
This set $\U$ is precisely invariant, since $\V$ is.
Every $h_i$ preserves $\U$ and preserves $T$. 

All $h_i$ preserve $T$.
By the previous Lemma the isotropy group subgroup of $T$
is at most cyclic.
%
Let $h$ be a generator of this subgroup. 
We consider the foliation
$\fol := \wls \cap T$.
By Proposition \ref{boundaryplane}, 
this foliation has leaf space the reals $\rrrr$.
We consider the action of $h$ on this leaf space.

The points $g_j(p_i)$ are in $\eta_0$ and escape
in a ray of $\eta_0$. 
Their images under $h_i$ are
$g_i(p_i)$ which converge to $p_0$ (which is also in $T$).
Since $h_i = h^{n_i}$ with $|n_i| \rightarrow \infty$
as $i \rightarrow \infty$, it follows that 
$h^{n_i}(\eta_0)$ converges to the leaf $\eta$ of $\fol$ through
$p_0$. It follows that this leaf $\eta$ of
$\fol$ has to be fixed by $h$.

%

\vskip .05in
We now use the following intermediate cover of $M$:
which is $M_1 := \mt/<h>$.  
In $M_1$ the projection $T_1 = T/<h>$ is either
an annulus or a M\"{o}bius strip and $\eta_1 = \eta/<h>$
is a closed curve in $T_1$. The components of $T_1$ minus
a neighborhood of $\eta_1$ are annuli $-$ there is one such
component if $T_1$ is a M\"{o}bius band and two components if $T_1$
is an annulus. Let $Z$ be one such component.
Let $\U_1 = \U/<h>$. 
Suppose that there are points $z_i$ in $Z$ escaping
to infinity in $T_1$ so that the distance across
$\U_1$ to any other boundary component of $\partial \U_1$
does not go to zero. Then an argument exactly as above
produces a deck transformation $h'$ of $\mt$
preserving $\U$ and $T$. Since the points are farther 
and farther from $\eta$ (lifting to $\mt$) it follows
that $h'$ is not a power of $h$.
Again this is impossible by the previous lemma.

It follows that
for any sequence escaping
in $Z$ then distance to the union of the other components
of $\partial \U$ goes to zero.
Consider one such sequence and lift to points
$y_i$ in $T$. 
Then the distance from $y_i$ to the union of
other boundary components of $\U$ goes to zero.
Since $\U$ is the convex hull associated with the
wall $g_j(W)$, there must be a wall
separating $T$ from this other component of 
$\partial \U$. Hence there are $v_i$ in 
deck translates $G_i$ of $W$ which are in $\U$ and 
$d(y_i,v_i)$ converges to zero.

\vskip .05in
Up to a subsequence assume that $\pi(v_i)$ converges
in $M$. Let $f_i$ in $\pi_1(M)$ so that $f_i(v_i)$
converges to $v_0$.
Notice that $G_i$ are deck translates of $W$,
and so are $f_i(G_i)$. They are also walls.
It follows that $f_i(G_i)$ converges to a wall
$G$. By hypothesis we are
under the New assumption as prescribed
in the beginning of Section \ref{transverse1}, hence there
is a deck translate $g(G)$ which intersects $G$
transversely. 
Since $f_i(G_i)$ converges to $G$ there is $i$ so that
for any $n, m \geq i$ the following happens:
$f_n(G_n)$ intersects $g(f_m(G_m))$ transversely. Since all
of these are translates of $W$, then if  $\Y$ is the
convex hull associated with $f_i(G_i)$, it follows
that all $f_n(G_n)$ and $g(f_n(G_n))$ are
contained in $\Y$ if $n \geq i$.
In particular $\Y$ is the convex hull of all these
deck translates of $W$.

\vskip .05in
The $v_n$ are in $G_n$ for $n \geq i$, the
$f_n(v_n)$ are in $\Y$ for $n \geq i$.
It follows that $v_0$ is either in $\Y$ or in $\partial \Y$.
If $v_0$ is in $\Y$, then so are $f_n(y_n)$ for $n$ 
big enough as $\Y$ is open. This is impossible since
$f_n(y_n)$ is in $f_n(T)$, $T \subset \partial \U$,
and the sets $\U$ and $\Y$ are precisely invariant.

Suppose on the other hand that $v_0$ is in $\partial \Y$,
and let $Z$ be the boundary component of $\Y$ containing
$v_0$.  Then there is $m$ so that for $n \geq m$,
$f_n(y_n) \in Z$.
Since $y_n \in T$ for any $n$ it follows that $Z = f_m(T)$.
Recall that $h(T) = T$, so $f_m h f^{-1}_m(Z)
= Z$. 
In addition $f_n f^{-1}_m(Z) = Z$ for any $n \geq m$.
Since the projection of $y_n$ to $M_1 = M / <h>$ escapes
compact sets in $M_1$ it follows that $f_n f^{-1}_m$
are not all in the group generated by $f_m h f^{-1}_m$.
This implies that the isotropy group of $Z$ is not cyclic.
This again is disallowed by Lemma \ref{atmostcyclic}.

This contradiction shows that the assumption in the
beginning of the proof of this proposition was impossible.
This proves the first assertion of the proposition.

\vskip .05in
For the second statement we do the following:
the non empty intersection of $\V$ with any
$E$ in $\wls$ contains an ideal triangle with 
two ideal points the ideal points of $H \cap E$. 
These triangles can be chosen to vary continuously with 
$E$. Then an argument as in the first part of the 
proof produces a non trivial deck transformation $f$
preserving $\V$ and $T$. 

This finishes the proof of Proposition \ref{asymptotic2}.
\end{proof}

\section{Completion of the proof of the Main Theorem}
\label{completion}

We now finish the proof of the main theorem.

\begin{theorem}
Let $W$ be a wall which is the limit of partial
walls associated with strings of lozenges. 
Suppose that there is $g_0$ 
in $\pi_1(M)$ so that 
$g_0(W)$ intersects $W$ transversely. Then $\Phi$
is $\rrrr$-covered.
\end{theorem}

\begin{proof}
Let $\V$ be the convex hull  of $L$.
There are two possibilities:

\vskip .05in
\noindent
{\bf {Case 1 $-$ $\V = \mt$.}}

But $\wls$ restricted to $\V$ is a foliation which
has leaf space the reals. Since $\V$ is $\mt$,
the 
the leaf space of $\wls$ is $\rrrr$, and
$\Phi$ is
an $\rrrr$-covered Anosov flow.

\vskip .05in
\noindent
{\bf {Case 2 $-$ $\V$ is not all of $\mt$.}}

Up to a finite cover assume that $M$ is orientable
and all foliations are transversely orientable.
For a sufficiently small $\eps$ consider 
the following set in $\pi(\V)$: for any 
$E$ stable leaf intersecting $\pi(\V)$ consider the
set $Z$ of 
points in $E \cap \pi(\V)$ which are $\eps$ away along $E \cap 
\pi(\V)$ from 
a point in $\pi(\partial \V)$.
We have proved in Proposition \ref{asymptotic2} that 
boundary components of $\V \cap E$ are asymptotic to other
such boundary components in both directions.
Then a proof exactly as in Proposition \ref{complementary}
shows that every component of 
$Z$ is a two dimensional torus. This implies as 
in the proof of Proposition \ref{complementary} that 
$\pi(\V)$ is an open  solid torus,
$\pi(\partial V)$ has finitely many boundary components,
and so does $\partial \V$.
It also implies that any component $H$ of $\partial \V$ 
intersects exactly the same stable leaves that $\V$ does.
One can prove that any such $H$ is in fact a wall associated
to limits of strings of lozenges, but
we will not do that as we can prove the $\rrrr$-covered
property for $\ls$ more directly using what was done
before in this article.

The above implies that for any $E$ leaf of $\wls$ intersecting
$\V$, then $E \cap \V$ is a finite sided ideal 
polygon. Let $g(W)$ be a wall contained in $\V$. Then $g(W) \cap E$ is
a geodesic where each ray is asymptotic to two geodesics
in $\partial \V \cap E$. 
In particular there are only finitely many
such $g(W)$ in $\V$: the options for $g(W) \cap F$ are
finitely many geodesics in $F$ for any $F$ intersecting $W$. 
In addition $g(W) \cap F$ varies continuously with $F$,
and for each boundary component $H$ of $\V$ the intersections
$H \cap F$ also vary continuously with $F$. Hence the
rays in intersections $g(W) \cap F$ always are asymptotic
to the same pair of boundary components $H$ and $G$ of $\partial \V$. 

Any other $g(W)$ cannot intersect $\V$ and hence
does not intersect $W$.
Let $h$ be a generator of the isotropy group of $\V$.
There is a smallest positive power $h^i$ which fixes
all $g(W)$ contained in $\V$. 
Then 

$$W \ \ \cap \ \ \cup \{ g(W) \subset \V, g(W)  \not = W \}$$

\noindent 
is a finite collection of curves, all invariant under $h^i$.
This implies that there is a half plane $P$ in $W$ bounded 
by a curve $\gamma$ invariant under $h^i$ so that
$P$ does not intersect any other translate of $W$. 
Take points $p_i$  in $W$ farther and farther from 
$\gamma$, so that $\pi(p_i)$ converges in $M$. 
Using appropriate deck translates $f_i$ of $W$ it follows
that the limit of $f_i(W)$ is a wall $G$ whose 
$\pi_1(M)$ deck translates do not intersect each 
other transversely. 
In addition $W$ is obtained as the limit of partial walls
associated with strings of lozenges.
But we are under the New assumption as in the
beginning of Section \ref{transverse1} so this
cannot happen
(in other words, the existence of such $G$ would
imply that $\ls$ is $\rrrr$-covered by the proof of the first
case of the Main Theorem).

This finishes the proof of the Main theorem.
\end{proof}

\section{Examples of walls}
\label{examp}

In this section we describe some examples of walls.
We assume throughout this section
that the Anosov flow $\Phi$ has the standard form used in this
article:
the leaves of $\ls$ have the hyperbolic metric, and flow lines
in any such leaf are geodesics in their respective stable leaves.

\vskip .1in
\noindent
{\bf {Example 1 $-$ $\rrrr$-covered Anosov flows.}} 

We first describe walls for general $\rrrr$-covered Anosov
flows and then specialize to atoroidal, Seifert and toroidal
manifolds.
We will assume that $\Phi$ is an $\rrrr$-covered Anosov flow
not orbitally equivalent to a suspension Anosov flow.
These are called {\em skewed} Anosov flows \cite{Fe1}.

We provide more information. For proofs and much more detailed
explanations and properties we refer to 
\cite{Cal1,Cal4,Fe1,Th3}.
For $\Phi$ a skewed Anosov flow
the following happens:
The foliation $\ls$ is what is called 
{\em uniform} \cite{Th3}: this means
that any two leaves $E, F$ are a finite
Hausdorff distance from each other in $\mt$, where
the Hausdorff distance depends on the particular 
pair of leaves $E, F$. Let $a_0$ be
the Hausdorff distance between $E, F$. There is a coarsely
defined map:

$$\tau: \ \tau_{E,F}: E \rightarrow F, \ \ 
\tau_{E,F}(x) = y, \ \ {\rm for \ some} \ \ y \in F,
\ {\rm with} \ d_{\mt}(y,x) \leq  a_0 + 1.$$

\noindent
This map is clearly not uniquely defined, but it is coarsely
defined \cite{Fe3} in the following sense: there is a continuous 
monotone function $\rho: [0,\infty) \rightarrow [0,\infty)$,
so that if $y, z \in F$ satisfy $d_{\mt}(y,z) \leq  b$,
then $d_F(y,z) < \rho(b)$.
The map $\tau_{E,F}$ is a quasi-isometry
with quasi-isometry constants which depend only on
the Hausdorff distance $a_0$ between $E, F$ and the function
$\rho$ evaluated at $2a_0$.
Since $\tau_{E,F}$  is a quasi-isometry then it extends to 
a homeomorphism \cite{Gr,Th1,Th2} still denoted by $\tau_{E,F}:
E \cup S^1(E) \to F \cup S^1(F)$ \cite{Th1,Th2,Gr}.

We now describe walls for skewed Anosov flows.
Recall that a wall $G$ has the property that any two leaves
of the foliation $\fol := G \cap \wls$ are a finite
Hausdorff distance in $\mt$ from each other.
Start with a leaf $E \in \wls$ and an arbitrary geodesic
$\gamma$ in $E$. For any $F$ stable leaf the image 
$\tau_{E,F}(\gamma)$ of $\gamma$ is a quasigeodesic in $F$ and it
is a bounded distance in $F$ from a unique geodesic
which is denoted by $\gamma_F$. Let 

$$W \ \  = \ \ \bigcup_{F \in \wls}  \ \gamma_F.$$.

In terms of the universal circle of the foliation $\ls$
\cite{Th3,Cal1}: given
two distinct $p, q$ points in the universal circle,
they define two points $p_F, q_F$ in $S^1(F)$ for any
$F$ in $\wls$, and $g_F$ is defined by the above formula.

It follows from \cite{Fe3} that $W$ is a properly embedded plane in
$\mt$. Given $F$ in $\wls$ and $\beta$ a geodesic in $F$
suppose it is a bounded Hausdorff distance in $\mt$ from 
$\gamma_F$. Then $d_{\mt}(\beta,\gamma_F)$ is finite and by the property
of $\rrrr$-covered foliations (the function $\rho$), then 
$d_F(\beta,g_F)$ is finite. Since $F$ has the hyperbolic
metric, one has that $\beta = \gamma_F$. It follows that the only candidates
for walls are the sets $W$ described by the above formula. 

Notice that given an $\rrrr$-covered Anosov flow $\Phi$, then
every orbit of $\wwp$ is a corner of a bi-infinite chain 
of lozenges, which is obtained as the limit of periodic
strings of lozenges \cite{Ba1,Fe1}. So all the walls 
obtained here are limits of partial walls that are strings
of lozenges.

We explain additional properties that establish
that any such $W$ is a wall.
First we claim that there is a leaf $F$ of $\wls$ so that $W \cap F$
is a flow line of $\wwp$. 
Fix $F$ in $\wls$. If one of the ideal points of $F \cap W$
in $S^1(F)$ 
is the positive ideal point of $F$ we are done. Otherwise
let $x, y$ be the ideal points in $S^1(F)$ of
$F \cap W$ and $\beta, \alpha$ the flow lines in $F$ with
negative ideal points in $S^1(F)$ which
are  $x, y$ respectively. Let $U = \wlu(\beta),
V = \wlu(\alpha)$.
As long as $U$ (or $V$) keeps intersecting stable leaves $F'$ nearby
then $U \cap F'$ (or $V \cap F'$) is asymptotic with $U \cap F$ 
(or $V \cap F$) in the
negative direction, and this determines $F' \cap W$. 
At some point $U$ and  $V$ stop intersecting stable leaves:
this is because any unstable leaf in $\mt$ intersects
only a bounded interval in the leaf space of $\wls$ \cite{Fe1}.
Suppose that the unstable leaf $U$ stops intersecting
stable leaves before $V$.
This means that the intersection of stable leaves
with $U$ escapes to infinity and the
limit stable leaf then intersects $W$ in a flow line,
which is the intersection with $V$.
This can be iterated on either side of the flow line
intersection, producing  infinitely many flow line
intersections on either side.

Near a flow line intersection $W \cap F$ for some $F$ in $\wls$,
one sees that intersections $W \cap F'$ have one ideal point
corresponding to a direction asymptotic with $W \cap F$ and
the other direction is expanding away from $W \cap F$.
The set $W$ intersects all leaves of $\wls, \wlu$, which
are both homeomorphic to $\rrrr$. 
Finally using the structure of skewed Anosov flows \cite{Fe1} it is
easy to see that each $W$ is associated with a bi-infinite
chain of lozenges, which is in fact a string of lozenges.
This implies that $W$ is a wall.

In summary if one fixes a stable leaf $E$ in $\wls$, then
the collection of all walls in $\mt$ is in a one to one correspondence
with geodesics in $E$.

\vskip .1in
\noindent
{\bf {Example 2 $-$ $\Phi$ is $\rrrr$-covered and $M$ atoroidal.}}

We now specialize to the case that 
$M$ atoroidal or equivalently  $M$ hyperbolic.
These are the walls $W$ that are used to prove the Main
Theorem. Here we will understand better what happens when 
$\pi(W)$ is embedded, or self intersects transversely.
For simplicity assume that $\ls$ is transversely
oriented. This can be always achieved by taking a double
cover. 

Thurston proved in \cite{Th3} (see also \cite{Cal1,Fe3}) that,
under the transvere orientability
hypothesis and $M$ atoroidal, there is a pseudo-Anosov flow
$\varphi$ which is transverse to $\ls$ and regulating for it.
{\em Regulating} means that every orbit of $\widetilde \varphi$ intersects
every leaf of $\wls$ and vice versa.
The pseudo-Anosov flow is obtained by a collapsing operation applied
to a pair of 
leafwise geodesic laminations $\G^s, \G^u$, the stable and
unstable laminations transverse to
$\ls$: each leaf $Z$ of $\wgs$ (or $\wgu$) 
intersects each stable leaf $E$ of $\wls$ of the flow $\wwp$ in
a geodesic in $E$. 
Furthermore for any $F, F'$ in $\wls$ then $Z \cap F$, $Z \cap F'$
are a finite Hausdorff distance from each other in $\mt$.
In other words every leaf of of $\wgs$ or $\wgu$ is a wall as
described in Example 1.

In particular this describes a collection of walls $W$ in $\mt$
so that
$\pi(W)$ are embedded in $M$, so obviously no transverse
self intersections in $M$. 
In addition for any
leaf $E$ of $\wls$,
the complementary regions of $E \cap \wgs$
and $E \cap \wgu$, $E \in \wls$,  are finite sided ideal polygons
in $E$.
The laminations $\G^s, \G^u$ are transverse to each other
and both are transvese to $\ls$. The flow $\varphi$ is obtained
by collapsing the complementary regions of $\G^s \cup \G^u$.
The complemetary regions of $\G^s$ or $\G^u$ 
are 
(open) solid tori in $M$ (in this case the transverse orientability
condition on $\ls$ forces $M$ to be orientable).
This possibility 
for the walls was
described and analyzed in the first case of the proof of the 
Main Theorem.

We stress that there are 2 very different flows in these examples:
$\Phi$ is an Anosov flow, $\varphi$ is strictly a pseudo-Anosov flow $-$
there are $p$ prong singular orbits of $\varphi$ with $p \geq 3$.

\vskip .1in
We have described the walls $W$ in $\mt$ so that $W$ is
a leaf of $\wgs$ or $\wgu$. There are two other types 
of walls in $\mt$.
The first type is as follows: let $E$ be a leaf of $\wls$ and
let $\gamma$ be a geodesic which does not intersect
$\wgs \cap E$ transversely, but also $\gamma$ is a not a leaf of 
$\wgs \cap E$. Let $W$ be the wall associated with $\gamma$.
Then $\gamma$ is in a complementary region of
$E \cap \wgs$, which we explained is a finite sided ideal
polygon $P$. Hence $\gamma$ is a diagonal in $P$.
There are finitely many complementary regions
to $\gs$ in $M$  and therefore finitely many possibilities
for $\pi(W)$ as above. It could be that $\pi(W)$ has
no transverse self intersections, hence its completion is
a lamination. It contains a sublamination of $\gs$, but
$\gs$ is minimal, so $\pi(W)$ completes to $\pi(L) \cup \gs$.
The other option is that $\pi(W)$ self intersects transversely.
Still the deck translates $g(W)$ intersecting $W$ transversely
cannot intersect $\gs$ transversely, so $g(W) \cap E$ is still
contained in $P$. It follows that the convex hull $\V$
of $W$ intersected with $E$
is contained in $P$: it could be $P$ itself, in which
case a boundary component of $\V$ would be a leaf of $\wgs$,
yielding the lamination $\gs$ in $M$.
Or it could be that $\V \cap E$ is a strict subset of 
 $P$.
The boundary components of $\V \cap E$ are either
leaves of $\wgs \cap E$ $-$ leading to the production of $\gs$,
or they are other diagonals in $P$, but now produce walls
without self intersections when projected to $M$ and this
is analyzed above as well.
The same happens with the unstable lamination, $\wgu$ intead
of $\wgs$.

\vskip .1in
The final possibility is that the geodesic $\gamma$ in $E$ intersects
both $\wgs \cap E$ and $\wgu \cap E$ transversely.
This case is more involved and we analyze this now.

We will use the following:
Given leaves $E, F$ of $\wls$ with $F$ being $\widetilde \varphi$
flow forward from $E$, then 
the map $\tau_{E,F}$ is a bounded distance
from the map which is flow along $\widetilde \varphi$ 
from $E$ to $F$ as follows: given $x$ in $E$ take the flow line
of $\widetilde \varphi$ through $x$ and let $\Delta_{E,F}(x)$
be the intersection with $F$. Then there is a
constant $a_1$ (independent of $x$ but dependent on
the Hausdorff distance between $E, F$)
so that 
$d_F(\tau_{E,F}(x),\Delta_{E,F}(x)) < a_1$.


Let $\F^s, \F^u$ be the stable and unstable foliations of $\varphi$.
We stress again that $\F^s$ is not to be confused with $\ls$,
in fact $\F^s$ is transverse to $\ls$.
The the pseudo-Anosov behavior of the flow $\varphi$ implies 
that any curve transverse to the singular
stable foliation $\F^s$ flows forward (under $\varphi$) to curves
closer and closer to the unstable foliation $\F^u$.
In particular any geodesic $\gamma$ in a leaf $E$ of $\wls$ so that it is 
transverse to $\wgs \cap E$ then flowing by $\varphi$
moves them closer and closer to
the unstable foliation ($\G^u$).

But we explained that the map
$\tau_{E,F}$ is also a bounded distance 
from the flow along map along map $\Delta_{E,F}$.
%
It follows that as $F$ escapes in the leaf space of $\wls$ in
the positive direction then $\pi(F \cap W)$ is getting
closer and closer to $\gu \cap \pi(F)$.
This shows that $\pi(W)$ limits to $\gu$ in the
positive direction and to $\gs$ in the negative
direction.
Since $\gu \cup \gs$ completely fill $M$, it 
now follows that the convex hull $\V$ associated
with $W$ is all of $\mt$.
This corresponds to the second case of the Main theorem
when $\V = \mt$.
In the second case we assume that for any wall 
$W$ obtained in the limit process
then $\pi(W)$ self intersects transversely.
If instead one assumes that this is true for some,
but not necessarily all walls $W$, then
one can obtain the case $\V = \mt$, as well as the
case $\V \not = \mt$, which is in fact what happens.

We specialize this further. Suppose that
for some $E \in \wls$ then $\gamma = \pi(E \cap W)$ is a closed
geodesic. This can only happen when $E \cap W$ is a
flow line of $\wwp$, which projects to a periodic
orbit of $\Phi$. It is proved in \cite{Fe1} that
$\gamma$ is transverse to both $\wgs$ and $\wgs$ so it
falls in this case. Let $h$ be the deck transformation
associated with $\pi(g)$. Then $h(L) = L$ and $h$
preserves all corner orbits of $W$, which then project to
periodic orbits of $\Phi$ which are pairwise freely
homotopic. 
Notice the similarlity with a suspension pseudo-Anosov flow:
if one starts with a closed curve in the surface then
positive iterates will limit to the unstable lamination
of the monodromy of the fibration, and negative iterates
will limit to the stable lamination.

\vskip .1in
\noindent
{\bf {Example 3 $-$ $\rrrr$-covered Anosov flows in $M$ Seifert.}}

In this case Barbot \cite{Ba2} proved $\Phi$ is orbitally
equivalent to a finite lift of a geodesic flow.
Hence to simplify we assume that
$\Phi$ is the geodesic flow on $M = T^1 S$, where $S$
is a closed hyperbolic surface. We also assume that 
$S$ is orientable $-$ this is equivalent to $\ls$ being
transversely orientable.
First let $\gamma_*$ be a simple closed geodesic in $S$ which lifts
to 2 periodic orbits $\gamma$ and $\gamma'$ of $\Phi$ corresponding
to different directions of $\gamma_*$ in $S$. Then $\gamma, \gamma'$ 
respectively
lift to infinitely many pairs of periodic orbits $\gamma_{2i}, \gamma_{2i+1}$
of $\wwp$ in $\mt$, 
$i \in \mathbb Z$. The $\gamma_j$ are corners of
a bi-infinite string of lozenges and generate a wall $W$ just
as in example 1, the case of no transverse intersection.
In this case $W$ is left invariant by a $\mathbb Z^2$ subgroup
of $\pi_1(M)$ generated by (say) a deck translation
associated with $\gamma$ and a deck transformation 
corresponding to a Seifert fiber of $M$.

The projection $\pi(W)$ has no transverse self intersections
and it is a torus. This is a lamination in $M$. Notice that the
boundary of a tubular neighborhood of the torus $\pi(W)$ 
is a torus in $M$, which in this case is incompressible.
Recall the analysis of the proof of the Main theorem:
in the case of $M$ atoroidal this boundary of a neighborhood 
of the lamination was a torus, which had to be compressible.

Now instead of starting with a simple closed geodesic,
let $\gamma'$ be a closed geodesic in $S$ with self intersections,
which
are necessarily transverse.
As in the case of $\gamma'$ simple we construct a wall $L$
associated with $\gamma'$.
The isotropy group of $L$ is a $\mathbb Z^2$ subgroup of $\pi_1(M)$.
Let $S_0$ be the subsurface of $S$ which is the fill of $\gamma'$.
Suppose first that $S_0 = S$, that is, $\gamma'$ fills $S$.
Then the deck translates of $W$ fill $\mt$, that is,
the convex hull $\V$ of $W$ is $\mt$.
If on the other hand
$S_0$ is not $M$, let $\beta$ be a boundary
component of $S_0$. Let $W$ be a wall constructed from $\gamma'$.
The convex hull $\V$ of $W$ has a boundary component
which is a wall $H$ associated with curves which
project to $\beta$ in $S$. The projection
$\pi(H)$ is a $\pi_1$-injective torus in $M$.

These are the walls invariant under a $\mathbb Z^2$ group
of deck translations. Every other wall is invariant
under the cyclic group generated by the element associated
with a Seifert fiber.

%

\vskip .2in
\noindent
{\bf {Example 4 $-$ $\Phi$ which is $\rrrr$-covered 
in $M$ toroidal, but $M$ not Seifert.}}

There are many examples in this class. We describe one class.
This was analyzed in detail in \cite{Fo-Ha}.
Start with $M_0 = T^1 S$, $S$ hyperbolic surface, for
simplicity assume $S$ orientable. Let $\gamma_*$ be a closed
geodesic in $S$. To make it more interesting we assume
that $\eta_*$ is not simple, and fills a subsurface $S_1$ of $S$.
One can make several examples here, again for definiteness assume
that $S_1$ is a strict subsurface of $S$, that is 
$S_1$ is not just $S$ cut along a collection of non
separating closed geodesics.
Let $S_2$ be the closure of $S - S_1$.

Let $\gamma$ be an orbit of the geodesic flow associated
with $\gamma_*$.
Do an appropriate Fried-Dehn surgery on $\gamma$
to result in an Anosov flow $\Phi$ in 
a manifold $M$. One can do this so that the resulting
flow is still Anosos and
$\rrrr$-covered \cite{Go,Fr,Fe1,Sha}. The submanifold $T^1 S_1$ of 
$M_0$ becomes an atoroidal submanifold $M_1$ of $M$ after
surgery \cite{Fo-Ha}. The original orbit $\gamma$ of the geodesic flow
becomes an orbit $\beta$ of $\Phi$ which is freely homotopic
to the infinitely many distinct periodic orbits of $\Phi$. 
Since $\Phi$ is $\rrrr$-covered these free homotopies produce
a wall $W$ in $\mt$. This wall has translates intersecting
it transversely, as in Example 3. The projection $\pi(W)$
fills the submanifold $M_1$ of $M$. As oppposed to Example 2,
the projection is a not a closed torus. 
Each boundary component of $M_1$ can be obtained as $\pi(H)$,
where $H$ is a wall which is a boundary component of the
convex hull $\V$ of $W$. As in Example 3, each such $\pi(H)$
is an embedded, incompressible torus in $M$.

One can do this for various orbits of the geodesic flow, filling
pairwise disjoint subsurfaces of $S$.

\vskip .2in
\noindent
{\bf {Example 5 $-$ $\Phi$ non $\rrrr$-covered.}}

We describe one of many possible constructions.
Start with the example described in Example 4, with $M_1$
obtained from $M_0 = T^1 S$ by Fried Dehn surgery on
the flow line $\gamma$ corresponding to a geodesic
$\gamma_*$ of $S$ filling the subsurface $S_1$ of $S$.
Let $S_2$ be the closure of $S - S_1$, which
we assume is a two dimensional surface (no single geodesic
components).
Let now $\beta_0$ be a flow line of the geodesic flow 
corresponding to a geodesic in $S_2$ which is not
peripheral in $S_2$. This flow line is unaffected by
the Fried Dehn surgery becoming an orbit $\beta_1$ of 
the flow in $M$. Now do a derived from Anosov blow up
of $\beta_1$ producing a flow with a repelling orbit
$\beta_2$ in $M_1$. Remove a solid torus neighborhood
of $\beta_2$ to produce a manifold $M_2$ with boundary
a torus and a semiflow transverse to $\partial M_2$ which
is incoming for the flow. Glue a copy $M'_2$ with 
a reversed flow. Work of Beguin-Bonatti-Yu \cite{BBY}
shows how to do the
gluing so that the resulting flow $\varphi$ on the resulting
manifold $N$ is Anosov. The walls $W, H$ constructed
in Example 4 produce walls $W', H'$ for $\widetilde \varphi$.
They are not exactly the same because the blow up of $\beta_1$
changes the stable and unstable foliations of the flows.
In the same way as in Example 4), the projection $\pi(W')$
intersects itself transversely. It fills a submanifold
of $N$ with a boundary isotopic to $\pi(H)$ which is 
embedded.

There are many variations of this construction.

\vskip .1in
More generally one can relax the requirement that a wall
intersects exactly an $\rrrr$ worth of stable and unstable
leaves in $\mt$. Then any non trivial homotopy from 
a periodic orbit to itself generates a bi-infinite chain
of lozenges invariant under a $\mathbb Z^2$ sugroup of $\pi_1(M)$.
These ``generalized walls" may be useful in other contexts.
One example of these generalized walls occurs in the
Bonatti-Langevin \cite{Bo-La} example. This has many free 
homotopies, generating infinitely many such generalized walls. One can
have 
pairs of these walls sharing a corner without being the same wall.
This shows that the requirement that $M$ is atoroidal is in
general necessary to obtain equality of walls when they
share a corner. In the Bonatti-Langevin example one can
make counterexamples that are in fact walls as
defined in this article and not generalized walls.

\section{Applications}
\label{applications}

\subsection{More quasigeodesic pseudo-Anosov flows}
\label{moreqg}

In this subsection we will prove the following:

\begin{theorem} Let $M$ be a hyperbolic $3$-manifold admitting
an Anosov flow $\Phi$. Then up to perhaps a double cover, $M$
admits a quasigeodesic pseudo-Anosov flow. In any case $M$
admits a one dimensional foliation by quasigeodesics, with a
dense leaf.
\end{theorem}

\begin{proof}
Let $\Phi$ be an Anosov flow in $M$. If $\Phi$ is not 
$\rrrr$-covered, then the Main Theorem shows that $\Phi$ 
is quasigeodesic and this proves the result.

We assume from now on that $\Phi$ is $\rrrr$-covered. 
Then $\Phi$ is not quasigeodesic \cite{Fe1}. Suppose first that
the stable foliation $\ls$ of $\Phi$ is transversely
orientable. Then, as explained in the previous section,
there is a pseudo-Anosov
flow $\varphi$ in $M$, transverse to $\ls$ and regulating for
$\ls$ \cite{Th3,Fe3,Cal1}. This flow
is quasigeodesic \cite{Th3}.
This finishes the proof in
this case.

Finally suppose that $\Phi$ is $\rrrr$-covered, but $\ls$
is not transversely orientable.
This is a non empty class, see eg. \cite{Fo-Ha}, for
examples which are Dehn surgeries on geodesic flows.
Let $M_2$ be the double cover of $M$ corresponding to
the deck transformations which preserve transverse
orientation to $\wls$ in the universal cover. 
Let $\pi_2: \mt \to M_2$ the corresponding
universal cover map. The flow $\Phi$ lifts to an
Anosov flow $\Phi_2$ in $M_2$ with a stable foliation
$\lss$ which is $\rrrr$-covered and transversely
orientable.
Notice that since $M$ admits an $\rrrr$-covered
Anosov flow, which is not orbitally equivalent to 
a suspension, then $M$ is orientable.

In $M_2$ there is a pseudo-Anosov flow $\varphi$ which is 
transverse and regulating for $\lss$. This is a quasigeodesic flow.
It does not project to a flow in $M$ transverse to $\ls$
since $\ls$ is not transversely orientable. However
the one dimensional foliation by flow lines of $\Phi_2$
can be made to project to a foliation in $M$ transverse to
$\ls$. It is easier to  do this using the leafwise
geodesic laminations $\G^s, \G^u$ associated with $\varphi$
and transverse to the foliation $\lss$.
The stable and unstable singular foliations of $\varphi$ are obtained
by collapsing the complementary regions of these 
leafwise geodesic laminations,
see eg. \cite{Th3}. These are the laminations discussed
in Example 2 of Section \ref{examp}.

Let $g$ be an element of $\pi_1(M)$.  We show that $g$
preserves the pair $\wgs, \ \wgu$ $-$ even though not necessarily
each of them individually.
If $g$ is in $\pi_1(M_2)$ then by definition $g$ preserves
$\wgs$ and $\wgu$.

Suppose then that $g$ is not in $\pi_1(M_2)$. Hence it 
reverses the transversal orientation to $\wlss$.
Since $\pi_1(M_2)$ is normal in $\pi_1(M)$, then $g$
induces an isometry of $M_2$ (of order $2$) and sends
$\G^s$ to another leafwise geodesic lamination.

\begin{claim}
$g(\wgs) = \wgu$, 
and consequently $g(\wgu) = \wgs$.
\end{claim}

\begin{proof}
Parametrize $\wls$ as $\{ E_t, t \in \rrrr \}$,
so that the positive 
$\widetilde \varphi$ flow direction is with increasing
$t$.
Going up means increasing $t$.

Let $W$ be a leaf of $\wgs$, which is a wall. Then $g(W)$
is also a wall. Therefore $g(W) \cap E_t$ is a bounded
distance from $g(W) \cap E_0$ (the bound depends on $t$).
We analysed this situation in Example 2 of the previous 
section: if $g(W) \cap E_0$ is not in the intersection
of $\wgs$ with a stable leaf, then $g(W) \cap E_t$
is getting closer and closer to the unstable lamination 
$\wgu \cap E_t$ as $t$ increases. Notice that $g(W) \cap E_t$ cannot be 
in a complementary region of $\wgs \cap E_t$, since then
it would be an isolated leaf of $g(\wgs) \cap E_t$.
Then $\pi_2(g(W))$ limits in $\G^u$ and so its closure
contains $\G^u$. Since $\G^s$ is minimal in this case
then $g(\wgs) = \wgu$. 

Can $g(W) \cap E_0$ be tangent to a leaf in $\wgs$? 
In that case, the same arguments imply that $g(\wgs) =
\wgs$. But this is impossible: going up ($t$ increasing) 
then leaves of $\wgs$ expand away from each other.
Applying $g$ reverses the orientation to $\wls$, so 
this means going down ($t$ decreasing). So going
down one gets transverse expansion. This would 
imply that both going up and down
one gets expansion of leaves of $\wgs$ away from each
other, which is impossible.

We conclude that $g(\wgs) = \wgu$ and $g(\wgu) = \wgs$.
\end{proof}

In any case $\pi_1(M)$ preserves the pair of laminations 
$\wgu, \wgs$, hence preserves their intersection. It also 
preserves the set of complementary regions of $\wgs \cup \wgu$
in $\mt$.
The projection of $\wgs$ to $M$ is an immersed lamination
(it has self intersections), it is equal to the projection
of $\wgu$. The complementary regions are locally 
a relatively compact polygon times an interval.
Then collapse the complementary regions in $M$ 
as is done in the transversely orientable case 
by Thurston \cite{Th4,FLP,Bl-Ca}. We obtain
a one dimensional foliation in $M$ transverse to $\ls$ and regulating
for it. 

Finally all flows are pseudo-Anosov in $M$ or $M_2$ which
are atoroidal. Hence there is dense orbit. This projects
to a dense leaf in the last case.

This finishes the proof of the Theorem.
\end{proof}

\subsection{Continuous extension of Anosov foliations}
\label{continuousextension}

\begin{theorem}
Let $\Phi$ be an Anosov flow in $M^3$ closed, hyperbolic.
Then the stable and unstable foliations of $\Phi$ have
the continuous extension property.
\end{theorem}

\begin{proof}
Suppose first that $\Phi$ is $\rrrr$-covered. Then both
$\ls$ and $\lu$ are $\rrrr$-covered.
Then the result was proved in \cite{Th3}, see also
Theorem G of \cite{Fe6}.

Suppose on the other hand that $\Phi$ is not $\rrrr$-covered.
It follows that $\Phi$ is a quasigeodesic Anosov flow by the Main
Theorem. Then
the continuous extension property
was proved in Theorem 5.8 of \cite{Fe1}, see also 
Theorem 3.4 of \cite{Fe4}.
\end{proof}

\subsection{Group invariant Peano curves and Cannon-Thurston maps}
\label{groupinvar}

We prove the abundance of group invariant Peano curves for
Anosov flows in hyperbolic $3$-manifolds.

\begin{theorem} Let $M^3$ be a hyperbolic $3$-manifold admitting
an Anosov flow $\Phi$. Then there is a group invariant Peano
curve $\eta: \partial \oo \rightarrow \si$ where $\partial
\oo$ is the flow ideal boundary of a quasigeodesic flow 
in either $M$ or a double cover $M_2$ of $M$.
In addition the map $\eta: \partial \oo \to \si$ is 
boundedly finite to one. Finally the identifications of
points of $\partial \oo$
mapped to same points of $\si$ are generated by points
which are ideal points of stable or unstable leaves
of an appropriate pseudo-Anosov flow.
\end{theorem}

\begin{remark} The flow ideal boundary $\partial \oo$ can also
be thought of as the universal circle associated with 
the corresponding pseudo-Anosov flow.
\end{remark}

\begin{proof}
The map $\eta$ is associated with a quasigeodesic pseudo-Anosov flow.
Given a pseudo-Anosov flow $\varphi$, the orbit space $\oo$ of 
$\widetilde \varphi$ is homemorphic to the plane \cite{Fe-Mo}. The
stable and unstable foliations of $\varphi$ descend to one
dimensional foliations in $\oo$. In \cite{Fe6} we define a flow ideal
boundary $\partial \oo$ of $\oo$. It is homeomorphic to a circle.
We refer to the reader to \cite{Fe6} for definitions, constructions,
and details.

This can be done for any pseudo-Anosov flow in any manifold.
When $M$ is hyperbolic and $\varphi$ has an upper bound on
the cardinality of sets of pairwise freely homotopic periodic orbits,
then in \cite{Fe7} we produced a group invariant Peano curve roughly
as follows: Identify $\mt$ with $\oo \times \rrrr$ where 
$\rrrr$ corresponds to the flow direction. 
Reparametrize the flow direction to be $(-1,1)$ so
$\mt$ is parametrized as $\oo \times (-1,1)$.
Then extend this 
model to 

$$\D \ \ = \ \ (\oo \cup \partial \oo)  \times [-1,1],$$

\noindent
using the flow ideal compactification. 
The fundamental group
$\pi_1(M)$ naturally acts on $\D$. In \cite{Fe6,Fe7} we consider a quotient
space of $\partial \D$ using the stable and
unstable foliations of $\varphi$ (thought of as one dimensional
foliations in $\oo$). 
Let this quotient be 

$$\R \ \ = \ \ \partial \D/ \sim,$$

\noindent
where $\sim$ is the equivalence relation in $\partial \D$.
The equivalence  relation is generated by the following:
two points
in $(\oo \cup \partial \oo) \times \{ 1 \}$ are identified if the pair is
contained in the union of a 
stable leaf in $\oo \times \{ 1 \}$ and its ideal points in 
$\partial \oo \times \{ 1 \}$.
In the lower boundary  $(\oo \cup \partial \oo)  \times \{ -1 \}$
we use the unstable boundary instead. Finally each vertical
segment $p \times [-1,1]$ is collapsed to a point. 
In particular every point in $\D$ is identified with a 
point in $\partial \oo \times \{ 1 \}$.
In other words, there is an induced collapsing map

$$\eta_0: \partial \oo \ \rightarrow \ \R,$$

\noindent
where we identify $\partial \oo$ with $\partial \oo \times \{ 1 \}$.

\vskip .1in
There is an induced action of $\pi_1(M)$ on $\R$.
The important fact is that under the boundedness
condition on freely homotopic periodic orbits it is proved
in \cite{Fe7} that the action of $\pi_1(M)$ on $\R$ is
a uniform convergence group action.
Super briefly: if a stable leaf $A$ makes a perfect fit with
an unstable leaf $B$, then the corresponding ideal points
of the rays of $A, B$ (seen in $\oo$) are the same point
in $\partial \oo$. So if there are too many perfect fits, then
there are too many identifications, which can lead to
bad topological behavior in $\R$.

Using results of Bowditch [Bo] (under the boundedness
condition) this implies that $\R$ is
$\pi_1(M)$ equivariantly homeomorphic to the sphere
at infinity $\si$. In addition there is a quotient space
$\mt \cup \R$ of $\D$ which is $\pi_1(M)$
equivariantly homeomorphic to the canonical compactification
of $\mt$ into $\mt \cup \si$. The induced map $\eta: \partial \oo
\rightarrow \R \cong \si$ is the 
group invariant Peano curve or Cannon-Thurston map associated
with this situation.

\vskip .1in
Now we apply this to the situation we have at hand starting
with an Anosov flow $\Phi$ in $M^3$ hyperbolic.

$-$ 1) If $\Phi$ is not $\rrrr$-covered, let $\varphi = \Phi$.
Then $\varphi$ is quasigeodesic and we apply the above construction
to $\varphi$.
This proves this case.

$-$ 2) If $\Phi$ is $\rrrr$-covered and its stable foliation
is transversely orientable, let $\varphi$ be a pseudo-Anosov
flow transverse to and regulating for $\ls$.
Then $\varphi$ is quasigeodesic 
\cite{Th3,Fe3,Cal1}. Apply the construction
above to $\varphi$. Notice that the orbit space of $\wwp$
is definitely completely different from the orbit space
of $\widetilde \varphi$. We apply the construction to $\varphi$
and not to $\Phi$ in this case.

$-$ 3) Finally suppose that $\Phi$ is $\rrrr$-covered, but 
$\ls$ is not transversely orientable. As in the previous section
let $M_2$ be the double cover which orients the lift $\lss$
of $\ls$ to $M_2$. Let $\varphi$ be a pseudo-Anosov flow transverse
to $\lss$ 
and regulating for $\lss$. Let $\oo$ be the orbit space
of $\widetilde \varphi$. Then the construction above
yields a Cannon-Thurston map 

$$\eta: \  \partial \oo \ \rightarrow \ \si.$$

\noindent
What we know is that this map is $\pi_1(M_2)$ equivariant.
Now consider the one dimensional foliation $\F$ in $M$ induced
by the flow lines of $\varphi$ in $M_2$. This foliation
has the same leaf space in $\mt = \mt_2$ as the orbit
space of $\widetilde \varphi$, since $\varphi$ is a double
cover of $\F$. Let $g$ be a deck transformation of $\mt$
(an element of $\pi_1(M)$). Suppose that $g$ is  not
in $\pi_1(M_2)$. We proved in the last section that $g(\wgu)
= \wgs$ and $g(\wgs) = \wgu$. 
 Hence $g$ also induces
a homeomorphism of $\partial \oo$. 
We explained that $\R$ is obtained
as a quotient of $\partial \oo$ where one identifies different
points which are ideal points of either a stable or
an unstable leaf in $\oo$. 
Since $g$ switches the
stable and unstable foliations, the same identifications
occur, so $g$ also induces a 
homeomorphism of $\R \cong \si$.


It is easy to see that this induces
an action of $\pi_1(M)$ on $\R$ and that $\eta$ is
$\pi_1(M)$ group equivariant, that is, a Cannon-Thurston map.

This finishes the proof of the first part of the theorem.

The identifications of $\eta$ are in bijection with those
of $\eta_0$. By definition the identifications of $\eta_0$
are generated by points which are ideal points of the
same stable or unstable leaf, or ideal points of rays
which make a perfect fit. Since there is an upper bound
on the number of leaves in
chains of leaves with consective ones
making perfect fits, this shows
that the map $\eta: \partial \oo \to \si$ is boundedly
finite to one. This also specifies what are the
identifications of $\eta$ as described in the statement.

This finishes the proof of the theorem.
\end{proof}

\begin{remark}
The existence of group invariant Peano curves can be 
obtained by the main result of Frankel \cite{Fra2}:
we have proved that either $M$  or a double cover $M_2$
has a pseudo-Anosov flow $\varphi$ which is quasigeodesic.
Let $N$ be the manifold in question.
Any quasigeodesic flow in a hyperbolic $3$-manifold
has orbit space homeomorphic to $\rrrr^2$ \cite{Cal3}.
Let $\Theta: \mt \to \oo$ be the projection map.
Frankel \cite{Fra2} following Calegari \cite{Cal3} 
proved that any section of $\Theta$ extends to a continuous
map of a closed disk $D^2 = \oo \cup \partial \oo$.
The map restricted to $\partial \oo$ 
is a $\pi_1(N)$ equivariant Peano curve. This proves the
result.

We note that 
the compactications of \cite{Fra2} and the
one described above are naturally equivariantly
homeomorphic: this follows from Theorem 2 of \cite{Bo}.

The approach here also produces a group invariant Peano curve,
where we know exactly what are the identifications, and
it also proves that the map constructed is boundedly
finite to one.
\end{remark}

\section{Future questions}
\label{future}

A) 
Obviously one 
immediate question raised by the results of
this article is to decide what
happens for general pseudo-Anosov flows in hyperbolic $3$-manifolds.
The best result would be to show that if there is a $p$-prong
singular orbit (with $p \geq 3$, or equivalently 
it is the case of a ``true" pseudo-Anosov
flow, as opposed to a topological Anosov flow), then the flow is quasigeodesic.
One obvious problem that needs to be solved is to find a good
geometry for the leaves of (say) $\ls$, which in this case
is a singular foliation. In this article we used Candel's 
results to give a metric in $M$ so that every leaf of $\ls$
is a hyperbolic surface. Clearly this does not work in the general
situation. An appropriate metric with appropriate properties
needs to be found to analyze this problem.

\vskip .05in
B) Another very important problem is to decide when a given Anosov
flow in $M^3$ hyperbolic is $\rrrr$-covered or not.
In the Introduction we described several examples which 
are $\rrrr$-covered, and several which are not $\rrrr$-covered,
using the results of Bonatti and Iakovoglou \cite{Bo-Ia}.
It would be good to know more about when exactly the
$\rrrr$-covered condition holds.

\vskip .05in
C) One can study the quasigeodesic question for Anosov 
flows in toroidal manifolds. The methods of 
this article could be particularly useful to study quasigeodesicty
in an atoroidal piece $P$. This corresponds to a
finite volume hyperbolic manifold. Putting the
boundary $\partial P$ in 
good form with respect to the flow (union of Birkhoff tori)
one can make sense of the stable or unstable foliation
being $\rrrr$-covered when restricted to $P$.
We expect that the $\rrrr$-covered case will not be 
quasigeodesic, this should not be too complicated.
The non $\rrrr$-covered case could be analyzed using
the tools of this article, and hopefully show that the flow
restricted to $P$ is quasigeodesic. This has a fairly good
chance to succeed. It is quite possible the same holds for
pseudo-Anosov flows as well.

\vskip .05in
D) Bands and walls can easily be defined for more general Anosov
flows. They could be useful to attack other problems involving
the topological and geometric structure of Anosov flows.

\vskip .05in
E) Let $\fol$ be a Reebless foliation, By Novikov's theorem
\cite{No} any leaf of $\fn$ is a properly embedded plane
in $\mt$. If $M$ is also hyperbolic, then for $L \in \fn$,
its limit set is the set of accumulation points in
the sphere at infinity $\si$.
If $\fol$ is $\rrrr$-covered (Anosov foliation or not),
then the limit set of any $L \in \fn$ is the whole
$\si$ \cite{Fe1}.
Suppose now that $\Phi$ is a non $\rrrr$-covered Anosov flow
in $M^3$ hyperbolic. Given $L \in \wls$ (or $\wlu$) what is
the limit set of $L$? Could it be $\si$? This is a good
test case for the following question: suppose that
$\fol$ is Reebless in $M^3$ hyperbolic. Is it true that
there is $L \in \fn$
with limit set of $L$ the whole $\si$ if and only if $\fol$ is
$\rrrr$-covered? The results of this article should be 
able to give light to this question in the case of
Anosov foliations.



\end{document}